\documentclass[12pt,leqno]{amsart}
\makeatletter
\AtBeginDocument{%
  \let\original@@tocwrite=\@tocwrite
  \newif\ifAHVflag
  \def\jlreq@uniqtoken{\jlreq@uniqtoken}
  \def\jlreq@endmark{\jlreq@endmark}
  \long\def\jlreq@getfirsttoken#1#{\jlreq@getfirsttoken@#1\bgroup\jlreq@endmark}
  \long\def\jlreq@getfirsttoken@#1#2\jlreq@endmark#3\jlreq@endmark{#1}
  \renewcommand{\@tocwrite}[2]{%
    \begingroup
      \AHVflagfalse % true if section is not empty
      \@ifempty{#2}{}{%
        \expandafter\expandafter\expandafter\ifx\jlreq@getfirsttoken#2\jlreq@uniqtoken{}\jlreq@endmark\Sectionformat\expandafter\@firstoftwo\else\expandafter\@secondoftwo\fi
        {%
          \def\Sectionformat##1##2{\@ifempty{##1}{}{\AHVflagtrue}}%
          #2
        }{\AHVflagtrue}%
       }%
      \def\@tempa{}%
      \ifAHVflag\def\@tempa{\original@@tocwrite{#1}{#2}}\fi
    \expandafter\endgroup
    \@tempa
  }%
}
\makeatother

\usepackage{a4wide}
\usepackage{amssymb}
\usepackage{amsmath}
\usepackage{amsthm}
\usepackage[all]{xy}
\usepackage[usenames,dvipsnames]{xcolor}
\usepackage{lmodern}
\usepackage[T1]{fontenc}

\newcommand{\charf}{\operatorname{char}}

\newcommand{\Ind}{\operatorname{Ind}}

\newcommand{\ind}{\operatorname{ind}}

\newcommand{\Ker}{\operatorname{Ker}}
\newcommand{\End}{\operatorname{End}}

\newcommand{\Supp}{\operatorname{Supp}}

\newcommand{\Lie}{\operatorname{Lie}}

\newcommand{\val}{\operatorname{val}}

\newcommand{\St}{\operatorname{St}}

\DeclareMathOperator{\diag}{diag}

 \DeclareMathOperator{\Ad}{Ad}
\DeclareMathOperator{\tr}{tr}
\DeclareMathOperator{\vol}{vol}
\DeclareMathOperator{\Aut}{Aut}

 \DeclareMathOperator{\Irr}{Irr}

\DeclareMathOperator{\red}{red}

   \DeclareMathOperator{\trace}{trace}

   \DeclareMathOperator{\Rep}{Rep}
   \DeclareMathOperator{\trd}{trd}
   \DeclareMathOperator{\nrd}{nrd}
    \DeclareMathOperator{\Gr}{Gr}

\theoremstyle{plain} 
\newtheorem{theorem}{Theorem}[section]
\newtheorem{corollary}[theorem]{Corollary}
\newtheorem{lemma}[theorem]{Lemma}
\newtheorem{proposition}[theorem]{Proposition}

\newtheorem{question}[theorem]{Question}

\newtheorem{definition-proposition}[theorem]{Definition-Proposition}

\theoremstyle{definition}

\theoremstyle{remark}
\newtheorem{remark}[theorem]{Remark}
\newtheorem{example}[theorem]{Example}

\numberwithin{equation}{section}

\title{Representations of $GL_n(D)$ near the identity }

\author{Guy Henniart and  Marie-France Vign\'eras}
\date{\today}

\begin{document}

\begin{abstract}Let $p$ be a prime number, $F $ a finite extension of $\mathbb Q_p$ or of $\mathbb F_p((t))$.
We consider the group $G=GL_n(D)$ for a positive integer $n$  and a central finite dimensional division $F$-algebra $D$  of $F$-dimension $d^2$. For 
an irreducible smooth   complex representation $\pi$ of $G$, inspired by work of R. Howe  when $D=F$,  
we establish the existence and uniqueness of 
integers $c_\pi(\lambda)$, for partitions $\lambda$ of $n$, such that for any small enough compact open subgroup $K$ of $G$ the
restriction of $\pi$ to $K$ is the same as that of the virtual representation $\sum c_\pi(\lambda) \Ind_{P_\lambda}^G 1$, where the
sum is over partitions $\lambda$  of $n$ and $P_\lambda$ is a parabolic subgroup of $G$ in the associate class determined by $\lambda$.
When $P_\lambda$  is minimal such that $c_\pi(\lambda) \neq 0$ we prove that $c_\pi(\lambda) $ is positive, equal  to the dimension of a generalized Whittaker model of $\pi$. 
We elucidate the behaviour of $c_\pi$ under the Jacquet-Langlands
correspondence $LJ $ of Badulescu from $GL_{dn}(F)$ to $G$.  
We extend  the above result on $\pi$ near identity to 
a representation of $G$  over a field $R$ with characteristic not $p$. For  any Moy-Prasad pro-$p$ subgroup $K$ of $G$,  we determine from the integers $c_\pi(\lambda) $ a polynomial $P_{\pi,K}$  with integral
coefficients  and   degree $d(\pi)$   independent on  $K$, such that, for large enough integers $j$, the dimension of fixed points in $\pi$ under
the $j$-th congruence subgroup $K_j$ of $K$ is $P_{\pi,K}(q^{dj})$ where $q$ is the cardinality of the residue field of $F$. 
\end{abstract}

\maketitle

  
\medskip Authors:

Guy Henniart,
Laboratoire de Math\'ematiques d'Orsay,
Universit\'e de Paris-Saclay,
Orsay
France

Marie-France Vign\'eras,
Institut de Math\'ematiques de Jussieu-Paris Rive Gauche,
Universit\'e de Paris-Cit\'e, 
Paris
France

\setcounter{tocdepth}{2}  
\tableofcontents
 \section{Introduction}
 Let $p$ be prime number, $F$ a finite extension of $\mathbb Q_p$   or of  $\mathbb F_p((T))$. Let $\underline G$ be a reductive  connected group over $F$,
and put $G=\underline G(F)$.
Let $R$ be a field, and $\pi$ a smooth admissible representation of  $G$  on an $R$-vector space $V.$

Our first motivation was in the following question, when $\pi$  is   of finite length:
Let $x$ be a point in the Bruhat-Tits building of $G$ and $r$ a positive real number. For any integer $j\geq 0$,
let $d (j) $ be the dimension of the space of fixed points in $V$ under the Moy-Prasad subgroup $G_{x,r+j }$ of $G$. 
\begin{question}Is there a polynomial $P$ with integer coefficients such that $d(j)=P(p^j)$ for large enough $j$ ?
If so, what can we say about its degree and its leading coefficient ?
\end{question}

When the characteristic $\charf_ R$ of $R$ is $p$, precise knowledge of those dimensions for irreducible $\pi$  is available only for $G=GL(2,\mathbb Q_p)$
(S. Morra, see  \S \ref{=p}). Apart for groups $G$ of relative rank one  those dimensions seem unknown.

Our paper studies the case where  $\charf_ R\neq p$. Then a smooth finite length $R$-representation $\pi$  of $G$
is automatically admissible, and its restriction to a pro-$p$ subgroup $K$ of $G$ is semisimple, with
finite multiplicities. We write $[\pi]_K$  for the image
of that restriction in the Grothendieck group of admissible $R$-representations of $K$.
We ask a more ambitious question:
\begin{question} Is there an open pro-$p$ subgroup $K$ of G where we
can control $[\pi]_K $ ? 
\end{question}
 
In the case of $GL_2(F)$ an answer to that question was offered by Casselman \cite{Cas73}. In this paper we consider the case where $G=GL_n(D)$ for a central division algebra $D$ over $F,$
with finite degree $d^2$ over $F$.  For a partition $\lambda=(\lambda_1,\dots, \lambda_r) $ of $n$, we let $P_\lambda$ be the upper block triangular subgroup
of $G$ with blocks of size $\lambda_1,\dots, \lambda_r$ down the diagonal, and put $d_\lambda= \sum_{i<j } \lambda_i\lambda_j$.
We have $d_\lambda\geq d_\mu$ if  $\lambda \leq \mu$ for the classical partial order $\leq $ on partitions.
We let $\pi_\lambda$  be the representation of $G$ non-normalized parabolically induced from the trivial representation of $P_\lambda$; it has finite length. 

\bigskip Let  $\pi$ be  a finite length smooth representation of $G$ on an $R$-vector space $V$.

\begin{theorem}\label{1.1} There is a unique function $c_\pi$  from partitions of $n$ to $\mathbb Z$
and an open pro-p subgroup $K=K_\pi$  of $G$ such that $[\pi]_K = \sum_\lambda c_\pi (\lambda)[ \pi_\lambda]_K$.

 If  $\lambda$ is minimal in the support
  of $c_\pi$, then $c_\pi(\lambda)$ is positive.
 \end{theorem}
 
 Theorem \ref{1.1}   has consequences to our first question. We let $q$ be the cardinality of the
residue field of $F$, so the residue field of $D$ has cardinality $q^d$. Let  $x$ a point in the Bruhat-Tits building of $G$
and $r $ a positive real number.  

\begin{theorem}\label{2.2}    Let $P= P_{\pi, G_{x,r}}$ be the polynomial  \begin{equation} \label{polynom}P_{\pi, G_{x,r}}(X)= \sum _{\lambda}  |P_\lambda \backslash G/G_{x,r} | \, c_\pi(\lambda)\,  X^{ d_\lambda}. 
 \end{equation}
Then $\dim_R V^{G_{x,r+j}}=P(q^{dj})$ for large enough integers $ j.$ The degree of $P$  is $d(\pi)=\max(d_\lambda)$
where the maximum is over partitions $\lambda$ in the support of $c_\pi$. The leading coefficient is
\begin{equation} \label{dc} a_{\pi,G_{x,r}}= \sum _{\lambda ,  d_\lambda=  d(\pi)} |P_\lambda \backslash G/G_{x,r} | \, c_\pi(\lambda).
\end{equation}
\end{theorem}

The function $c_\pi$  has good properties with respect to natural operations, apart from being
 additive on exact sequences, hence factoring to a function on the Grothendieck
group of finite length smooth representations of $G$.
If $\chi$  is a character of $G$, $c_{\chi \pi}=c_\pi$. If  $\pi'$ the base change of $\pi$  to an extension  $R'$   of $R$, 
then $c_{\pi'} = c_\pi$; in particular $c_\pi$ is invariant under automorphisms of $R$.
When $p\neq 2, G=GL_n(F)$ and $\charf_ F=0$  the support of $\pi$ contains  a single partition $\lambda$ with $d_\lambda= d(\pi)$ \cite{MW87}. This may be true for any $p,F$ and $D$.  
 
\bigskip  {\bf Parabolic induction} Let $P$ be an upper block triangular subgroup of $G$, with block diagonal Levi subgroup
$M$   a product $GL_{n_1}(D)\times \ldots \times GL_{n_r}(D)$. For $i=1,\ldots, r$  let $\sigma_i$ be a finite length representation of $GL_{n_i}(D)$,  and put  $\sigma=\sigma_1\otimes \ldots \otimes \sigma_r$ a finite length representation of $M$.
Given   a partition  $\lambda_i$ of $n_i$ for  $i=1,\ldots, r$, we   have the induced partition $\lambda$ of $n$ obtained by gathering all the
parts of the $\lambda_i$'s and putting them in decreasing order.
 
\begin{theorem}\label{3.3}Let 
$ \pi = \ind_{P}^{G} (\sigma )$.
For each partition $\lambda$ of $n$,  $c_\pi (\lambda)=\sum \prod_{i=1,\ldots, r }c_{\sigma_i} (\lambda_i)$, where the sum is over
r-tuples of partitions $(\lambda_1, \ldots,\lambda_r) $ inducing to $\lambda$.   
\end{theorem}
{\bf Whittaker models} Assume that $R$ contains all the roots of unity of $p$-power order. We have the notion
of Whittaker models, possibly degenerate. Let $U$  be the upper triangular subgroup of $G$, and $\theta$  a character of $U$.
We let $V_\theta$ be the maximal
quotient of the space $V$ of $\pi$ on which $U$ acts via $\theta$. Its dimension is finite and depends on $\theta$  only up to conjugation by the
diagonal subgroup $T$ of G. The orbits of  $T$ on the characters  of $U$ are parametrized by the compositions of $n$.  To each composition $\lambda$ of $n$ is attached a partition $\lambda ^\dag$ obtained by gathering the parts of $\lambda$  in decreasing order.
The  Whittaker support of $\pi$ is the set of partitions of $n$ of the form $\lambda ^\dag$ where $\lambda $  is a composition of $n$ such that $V_\theta\neq 0$
  for $\theta$  corresponding to the composition $\lambda$.

\begin{theorem}\label{4} The minimal elements in the support of $c_\pi$  and in the Whittaker support of $\pi$  are the same.
If $\mu$  is such a minimal partition, $\lambda$ is a composition of $n$ with $\lambda ^\dag=\mu$ and $\theta$ a character of $ U$ corresponding to $\lambda$,
then $c_\pi(\mu)=\dim_R V_\theta$.
\end{theorem}

{\bf Jacquet-Langlands correspondence} I.Badulescu  has extended the classical
Jacquet-Langlands correspondence to a morphism $LJ_{\mathbb C}$ from the Grothendieck group of smooth finite length
complex representations of  $GL_{dn}(F)$ to that of $G$. Let $\ell$ be a prime number different from $p$.
For an algebraic closure $\mathbb Q_\ell^{ac}$ of $\mathbb Q_\ell$, with a chosen square root of $q$, A.Minguez and
V.S\'echerre have transported $LJ_{\mathbb C}$ to  $\mathbb Q_\ell^{ac}$-representations, and showed that it
descends to a map  $LJ_{\mathbb F_\ell ^{ac}}$ of $\mathbb F_\ell ^{ac}$-representations, where $\mathbb F_\ell ^{ac}$ is the residue field of $\mathbb Q_\ell^{ac}$.
We define $LJ_R$ for our field $R$, provided it be algebraically closed, and get:

\begin{theorem}\label{5}Assume  $R$ to be algebraically closed. Let $\tau$ be a finite length smooth $R$-representation of $GL_{dn}(F)$
and $\pi =LJ_R(\tau)$. For any partition $\lambda$ of $n$, we have $(-1)^n c_\pi (\lambda)=(-1)^{dn} c_\tau(d\lambda)$.
\end{theorem}
For $R=\mathbb C$ and a discrete series $\pi$, the result is due to D.Prasad \cite{P00}.

\bigskip  We show in  \S \ref{11} how to get Theorem \ref{2.2} from Theorem \ref{1.1};  this amounts to computing the dimensions
of fixed points for the $\pi_\lambda$'s.
Our method establishes the other results first for $R=\mathbb C$, and then extends them to $R$. 
Let us hasten to mention that when $R=\mathbb C$ part of the results were known. Indeed when $D=F$ and $\charf _F=0$,
the first part of Theorem \ref{1.1} is due to \cite{Howe74}. We actually adapt Howe's arguments to our setting.
Similarly when $\charf_F=0$  one can obtain Theorems \ref{3.3}, \ref{4} (and the second part of Theorem \ref{1.1}) from the much
more general results of \cite {MW87}, and we get inspiration from their proofs. \footnote{While we were writing our results, the preprint \cite{Su22} reached us.
When $R=\mathbb C, D=F$ and $\charf_F=0$, Suzuki establishes Theorem \ref{2.2} for the congruence subgroups $K_j=1+M_n(P_F^j)$ of  $GL_n(F)$.
He also gets a result equivalent to Theorem \ref{3.3} and Theorem \ref{5} for square integrable $\tau$. His methods are
similar to ours.}

We now give more detail on our method of proof.
First we take $R=\mathbb C$. In that case, knowing $[\pi]_K$ for an
open compact subgroup $K$ of $ G$ is equivalent to knowing the character $\trace (\pi) $ on smooth functions  on $G$
supported in $K$. An expression of $\trace (\pi)$ on small enough $K$  as a linear combination of finitely many easier
distributions is usually called a  germ expansion for $\pi$.  
When $\charf _F=0$, the theory of germ expansions has a long history. For a reductive group $ G$  and $\pi$  irreducible  
Harish-Chandra  established a germ expansion of  $\trace(\pi)$ as a linear combination of Fourier transforms of nilpotent
orbital integrals on the Lie algebra $\mathfrak g$ of $G$, with coefficients a priori only complex numbers  \cite{HC70}.
To get from functions on $G$ to functions on $\mathfrak g$, he used the exponential map, which is not available to us
 when $\charf_F>0$.
The interest of our group $G=GL_n(D)$ is that $\mathfrak g=M_n(D)$, so that nilpotent orbits of $G$ in $\mathfrak g$ are  
parametrized by partitions  of $n$, and that one can use the map $e:X\mapsto 1+X$ from $\mathfrak g$  to $G$ as a substitute for
the exponential. When $D=F$, Howe  proved using $e$ that the Fourier transform of the nilpotent
orbital integral corresponding to a partition $\lambda$ is proportional to $\trace (\pi_\lambda)$, and got a germ expansion
$ \trace (\pi)= \sum_\lambda c_\pi(\lambda) \trace (\pi_\lambda) $ on the $i$-th congruence subgroup $K_i$  for $i $ large enough. He  showed that the  $c_\pi(\lambda)$ are integers by  constructing for any $i>0 $  a character $\xi_\lambda$ of $K_i $ which appears with multiplicity $1$ in $ \pi_\lambda$ and multiplicity 0 in $\pi_\mu$
unless $\lambda \geq \mu$ \cite{Howe74}. We show the existence of such characters for $D$ in Lemma \ref{lemma6}.
For our group $G$  and $\pi$ irreducible, B.Lemaire proved 
the local integrability of the distribution $\trace( \pi)$  (that was new when $\charf_F=p $)  and adapted Howe's arguments to get a germ expansion  as a linear combination of Fourier
transforms of nilpotent integrals \cite{L04}, which by our Proposition \ref{hatpi}  translates into a germ expansion as in Theorem \ref{1.1}.
Our  characters $\xi_\lambda$ then yields the integrality statement and the positivity statement.

Theorem \ref{3.3} follows from the known behaviour of traces with respect to parabolic induction. In \S \ref{s:PI},
we give a treatment valid whatever $\charf_F$ is.

As already said, when $\charf_F=0$, Theorem 4 can be obtained from  results of C.Moeglin  and J.-L.Waldspurger for a reductive group $G$ and $\pi$ irreducible. They attach to a nilpotent orbit $ \mathfrak O$ of $G $ in $\mathfrak g$ a number of  generalized Whittaker
spaces of $\pi$. 
 They consider the Harish-Chandra germ expansion of $\pi$ as a linear
combination $\sum  c_\pi (\mathfrak O) D_{\mathfrak O}$ over the nilpotent orbits $\mathfrak O$,  where $ D_{\mathfrak O}$ is the Fourier
transform of the orbital integral along $\mathfrak O$. They show that if $\mathfrak O$ is maximal in the support of $c_\pi$
then the dimension of any Whittaker space attached to $\mathfrak O $ is $c_\pi(\mathfrak O)$. The nilpotent orbits with that maximality
property go by the name of wave front set of $\pi$  and there is a large literature on that subject.
 In our more restricted setting, but allowing $\charf_F=p$, we get Theorem \ref{4} by adapting arguments of  \cite{Rodier}\footnote{Rodier assumed $\charf_F=0, G $ split and the support of $c_\pi$ contains the maximal nilpotent orbit}
and \cite{MW87}.

Still with $R=\mathbb C$, to prove Theorem \ref{5} in  \S \ref{s:JL} we use that the Jacquet-Langlands correspondence LJ is expressed by
character identities, where the characters  are considered as locally $L^1$ functions on regular semisimple elements  (by the result of B.Lemaire alluded to above).  

In  \S \ref{s:10} we pass from $R=\mathbb C$ to the general case. To transfer the results from a field $R$ to an isomorphic field $R' $
we use  that the theory of smooth
representations is essentially algebraic. That gives the case of $\mathbb Q_\ell ^{ac} $ which is isomorphic to $\mathbb C$. We then get the case
of $\mathbb F_\ell ^{ac} $ by reduction, using the results of  \cite{MS14}. To transfer the results
from an algebraically closed field $R$ to an algebraically closed extension $R'$, we use the fact that for a cuspidal $R'$-representation
$\pi$  of  $G$, there is a  character $\chi$ of $G$ into $R'^*$ such that $\chi \pi$  comes by base change from an $R$-cuspidal representation of $G$.
To get the result for any $R$ we show that Theorem \ref{1.1} over an algebraically closed extension $R^{ac}$  of $R$ implies Theorem  \ref{1.1}
over $R$ essentially because base change  preserves finite length.
 
When $n=2$ and $D=F$, we compute in  \S \ref{n=2} the two coefficients $c_\pi(\lambda)$ for all irreducible $\pi$.
When $ n=3$ or $4$, $D=F, \charf_F=0, R=\mathbb C$, F.Murnaghan computes the three  coefficients $c_\pi(\lambda)$
for cuspidal representations $\pi$  of $G$ induced from $F^*GL_n(O_F) $\cite{M91}. 
For any   split reductive group $G$ over $F$,  R.Meyer and  M.Solleveld 
using the Bruhat-Tits building of $G$,  give an upper bound on $\dim_R V^{C_r}$ for some
  special cases $C_r,$ of Moy-Prasad subgroups   (\cite{MS12}Theorem 8.5). Their result is  far less precise than ours.  

\bigskip Acknowledgments.  We thank I.Badulescu, D.Bernardi, P.Broussous, B.Lemaire, G.McNinch, A.Minguez, S.Morra, V.S\'echerre for a number of conversations about the topic of the paper. 
The second author has talked about this work at numerous conferences in 2022 (Stockholm,  Singapore, Grenoble, Heidelberg, Bangalore), and  
thanks the organizers for their invitation. A part of this work was done at 
 the Institute for Mathematical Sciences of the National University of Singapore. The final version of that paper was written when the first-named author was enjoying the hospitability of the Graduate School of 
Mathematical Sciences of the University of Tokyo.

\section{Notations}\label{not} Let $p$ be a prime number, and $F$ a  local non archimedean field of residual characteristic $p$. We denote by  $O_F$ the ring of integers of $F$, $P_F$ the maximal ideal of $O_F$,  $p_F$ a generator of 
$P_F$,  $k_F=O_F/P_F$  the residue field of order $q= p^f$ where $f= [k_F:\mathbb F_p]$ is the residual degree, and  $F^{ac}$ an algebraic closure of $F$.  Let   $| \ |$ denote the absolute value  of $F^{ac}$ such that for  $x\in  F^{ac}$ non-zero, and   $N_{E/F} $ the norm of  a finite extension $E$ of $F$ containing  $x$,  we have
 $|x|^{[E:F]}=|N_{E/F}(x)|=|O_F/ N_{E/F}(x) O_F|$ (\cite{Cassels67} 10.Theorem). 
In particular $|p_F|=q^{-1}$.

Let $D$ be  a central division $F$-algebra  of finite dimension $d^2$. We denote by $O_D$ the maximal  order of $D$, $P_D$ the maximal ideal of $O_D$,  $p_D$ a generator of $P_D$, $k_D=O_D/P_D$ the residue field of cardinal $q^d$; we have $p_F O_D=P_D^d$ \cite{R75}.

\medskip Let $n$ be a positive integer and $G=GL_n(D)$. Put $K_0=GL_n(O_D)$ and 
$K_i=1+M_n(P_D^i)$ for a positive integer $i$. Let  $Z\simeq F^*$ denote  the center and 
 $\mathfrak g=M_n(D)$  the Lie algebra  of $G$. Let $\trd, \nrd:M_n(D)\to F$ be  the reduced trace, the reduced norm. 
The symmetric $G$-invariant bilinear form    $(X,Y)\mapsto \trd(XY): M_n(D)\times M_n(D)\to F$ is not degenerate and  $G=  \{Z\in M_n(D) \ | \ \nrd (Z)\neq 0\}$. 

\medskip The letter  $P$ will denote  a  parabolic subgroup of $G$, its unipotent radical is usually written $N$, and $M$ is used for a    Levi subgroup so that $P=MN$. We write $\mathfrak p, \mathfrak m, \mathfrak n $  for their Lie algebras.

\medskip A composition  $ \lambda=(\lambda_i)$ of $ n=\lambda _1+\ldots + \lambda_r, \ \lambda _i\in \mathbb N_{>0},$   is called a partition of $n$ when the sequence $ (\lambda_i)$ is decreasing. To a composition $\lambda$ of $n$ is associated a parabolic subgroup 
$P_\lambda$ of $G=GL_n(D)$ with Levi subgroup $M_\lambda $ block-diagonal with blocks of size $\lambda_1, \ldots, \lambda_r$ down the diagonal,  and unipotent radical $N_ \lambda$ contained in  
 the upper  triangular subgroup $B$.   We let  $P_\lambda^-=M_\lambda N^-_\lambda$ the parabolic subgroup opposite to $P_\lambda$ with respect to $M_\lambda$. We have $G=P_{(n)}$ and $P_{(1,\ldots, 1)}=B $. We denote  by $T$ and $U$ the group $M_{(1,\ldots, 1)}$ of diagonal matrices with entries in $D^*$ and the strictly upper triangular group   $ N_{(1,\ldots, 1)}$.
 A parabolic subgroup $P$  of $G$ is conjugate to $P_ \lambda$ for a unique composition  $ \lambda$ of $n$  and  is associated to 
$P_{\lambda^\dag}$  for the   unique partition $\lambda^\dag$ of $n$ deduced from  $ \lambda$ by re-ordering its elements.
 Let  $\mathfrak P(n) $  denote the set of partitions of $n$. For $\lambda =(\lambda_1, \ldots, \lambda_r)\in \mathfrak P(n) $, $d\lambda=(d\lambda_1, \ldots, d\lambda_r)\in \mathfrak P(dn) $.

 \medskip Let $R$ be a  field.  We denote by  $\charf _R$ the characteristic of $R$, and by  $C_c^\infty(X;R)$ the $R$-module  of  locally constant functions on a locally profinite space $X$   with compact support and values in $R$.  The map $$\varphi\mapsto f(1+X)=\varphi(X):  C_c^\infty( M_n(P_D);R)\to C_c^\infty(K_1;R)$$ is a $K_0$-equivariant isomorphism.
The extension by $0$ embeds $C_c^\infty( M_n(P_D);R)$  in  $ C_c^\infty( \mathfrak g;R)$ and 
 $C_c^\infty(K_1;R)$ in  $ C_c^\infty( G;R)$.  
  An  $R$-distribution on $G$ or on $\mathfrak g$  is a linear form on  $C_c^\infty(G;R)$ or $C_c^\infty(\mathfrak g;R)$.  The group $G$ acts on   $G$ and   on  $\mathfrak g$ by   conjugation, and by functoriality 
    on $C_c^\infty(G;R)$,  $C_c^\infty(\mathfrak g;R)$ and on the distributions. A $G$-invariant distribution is called invariant.
 
 \medskip  For $R=\mathbb C$,   $dg$ will denote  the Haar measure on $G$ such that $dg$  gives the volume $1$ to  {\color{blue} $K_0 $}, et $dZ$  the Haar measure on  $\mathfrak g$ giving the volume $[K_0:K_1]^{-1}=|GL_n(k_D)|^{-1}$ to $M_n(P_D)$, hence the volume $a= q^{dn^2} |GL_n(k_D)|^{-1}$ to $M_n(O_D)$. The Haar measures $dZ$ and  $dg=a |\nrd Z|_F^{-n}dZ$ (\cite{W67} X, \S 1 Lemma 1) are compatible with the map $x\mapsto 1+x: M_n(P_D)\to  K_1 $. 
  The  modulus  of $P$  is 
$\delta_P(p) =|\det (\Ad p)_{\mathfrak n}|_F$  (\cite{V96} I.2.8).
  Let $dk$ denote the restriction of $dg$ to  $K_0$, 
 $dp$  the   left Haar measure  on $P$  such that $dg= \delta_P(p) dk dp$, $dn^-$ the Haar measure on $N^-$ such that $dn^-\, dp$ is the restriction of $dg$ to $N^-P$ (open in $G$), $dn$   the Haar measure on $N$  giving the same volume 
to $N\cap K_0$  as the volume of  $N^-\cap K_0$ for $dn^-$, and $dm$  the Haar measure on $M$ such that $dp=dm\, dn$. For each $f\in C_c^\infty(G;\mathbb C)$, 
$$  \int_{G} f(g) dg =\int_{K_0\times P}f(p^{-1}k) \,  dk \, dp= \int_{K_0\times P}f(kp) \,  \delta_P(p)\,  dk \, dp$$
$$= \int_{K_0\times M \times N}f(kmn) \,  \delta_P(m)\,  dk \, dm \, dn .$$
 Let $dW, dY^-, dY$ be the Haar measures on $\mathfrak h= \mathfrak p, \mathfrak n^-, \mathfrak n$  such that $dp$ and $dW$, $dn^-$ and $dY^-$, $dn$ and $dY$ are compatible 
 with the map $x\mapsto 1+x$ for  $x\in \mathfrak h(P_D)=\mathfrak h\cap M_n(P_D)$. We have $dZ= dW dY^-$.

 \medskip Let $\pi$ be a  smooth representation of $G$  on an $R$-vector space $V$. Each vector is fixed by some open compact subgroup $K$ of $G$, 
 \begin{equation}V=\cup_KV^K \ \ \ \text{where} \ V^K=\{\text{vectors of } V \ \text{fixed by } K\}.
  \end{equation} $\pi$ is called  admissible when  the dimension $\dim_RV^K$ of $V^K$ is finite  for any open compact subgroup $K$.  
 The categories  $\Rep_R^\infty(G)$ of smooth $R$-representations of $G$, $\Rep_R^{\infty,f}(G)$ of finite length smooth representations  are abelian.  
 When  $\charf_ R\neq p$,   the category   of admissible $R$-representations of $G$ is abelian and contains $ \Rep_R^{\infty,f}(G)$ (this is not true when  $\charf_ R= p$). 
 We denote by $\Gr_R^{\infty}(G)$ the  Grothendieck group  of  $\Rep_R^{\infty,f}(G)$, and
  $$\pi\mapsto [\pi]: \Rep_R^{\infty}(G)\to \Gr_R^{\infty}(G) $$ the natural homomorphism.    The map $\chi \mapsto \chi \circ \nrd$ is a bijection from the  smooth characters  $F^*\to R^*$  onto the smooth characters  $G\to R^*$. 

\medskip For  a set $X$ and a function $f$ on $X$ with value in $\mathbb Z$ or in $R$, the support $\Supp f$ of $f$ is the set of $x\in X$ with $f(x)\neq 0$ and $1_Y$ will denote the characteristic function of a subset $Y$ of $X$.

\section{Nilpotent orbits}
\subsection{} \label{geometric} An element $X\in \mathfrak g $ is nilpotent if and only if $X^r=0$ for some $r\in \mathbb N$.  The set  $\mathfrak N $ of nilpotent elements in  $\mathfrak g$ is stable by $G$-conjugation. A $G$-orbit in   $\mathfrak N $ is called a {\bf nilpotent orbit of $G$}.  
     The set   $G\backslash \mathfrak N$  of  nilpotent orbits of $G$ is  finite, in bijection with the set $\mathfrak P(n)$ of partitions  of $n$ (\cite{BHL10}  \S 2.4-2.6).

  \subsection{}\label{parti} Let $V$ be the  right $D$-vector space $D^n$. The group $G$  identifies with $\Aut_D(V) $ and   its Lie algebra $\mathfrak g$ with $\End_D V $. 
  Let $X\in \End_D V $ be nilpotent.
 The  composition $\lambda = (\lambda_1, \ldots  )$ of $n$,   
 \begin{align}\label{parti} \lambda_i= \dim_{D} \Ker X^i-  \dim_{D} \Ker X^{i-1}  \ \ \ {\text for \ } i\geq 1,\end{align}
  is   a partition because the multiplication by $X$  induces an injection from $ \Ker X^i/  \Ker X^{i+1} $ to $  \Ker X^{i-1}/  \Ker X^{i}$.  
We get a canonical map  $\mathfrak N \to \mathfrak P(n)$ sending $0$ to $(n)$.  The map is bijective. Let  $ \mathfrak O_\lambda$ denote  the nilpotent orbit  of $G$ containing $X$. 
The dual  partition of $\lambda$ 
is $\hat \lambda=(\hat \lambda_i= |\{j \ | \ \lambda_j\geq i \}|)$. There is a  partial order on $\mathfrak P(n) $ 
$$ \mu \leq \lambda\  \Leftrightarrow  \  \hat  \lambda  \leq  \hat \mu \ \Leftrightarrow  \ \sum_{i=1}^j\mu_i\leq \sum_{i=1}^j\lambda_i  \text{ for all $j$}.$$
There is also a  partial order on $G\backslash \mathfrak N$ 
$$\mathfrak O' \leq\mathfrak O\  \Leftrightarrow \mathfrak O'\subset \overline {\mathfrak O} \text{ where $ \overline {\mathfrak O}$ is the closure of $\mathfrak O$ in $\mathfrak g$.  }
$$
The bijection reverses the partial order.
 \begin{equation} \overline{\mathfrak O} _\lambda= \cup _{\hat \mu \leq \hat \lambda}\mathfrak O _\mu.
 \end{equation}
 The  unique maximal partition $(n)$ corresponds the null orbit  $\{0\}=\mathfrak O_{(n)}$.
 The  unique minimal  partition $(1,\ldots, 1)$ corresponds to the  unique maximal nilpotent orbit $\mathfrak O_{(1,\ldots, 1)}$, called regular,  of closure $\overline{\mathfrak O}_{(1,\ldots, 1)} =\mathfrak N$.
The  parabolic  subgroup  $P$ of $\Aut_D(V)$ preserving the flag  $(\Ker X^i)_i$ 
   of  the iterated kernels of $X$, 
  is associated to  $P_\lambda$. The intersection $\mathfrak O_\lambda\cap \mathfrak n_\lambda$ is open dense in  $\mathfrak n_\lambda$ \cite[\S 13.17]{Jan04}. The dimension of $\mathfrak O_\lambda $ as an $F$-variety is even and equal to   (loc.cit.)
 \begin{align}\label{dind}\dim_F \mathfrak O_\lambda&= 2 \dim_F \mathfrak  n_\lambda =  2    d^2\dim_D\mathfrak  n_\lambda,\\ 
     \label{dlambda}  \dim_D \mathfrak n_\lambda  &= \sum_{i<j} \lambda_i \lambda_{j} . \end{align} 
We  denote $ d_\lambda=\sum_{i<j} \lambda_i \lambda_{j} $ and $d(\mathfrak P(n))= \{d_\lambda \ | \ \lambda \in \mathfrak P(n)\}$,  \begin{equation}\label{d}
d(\mathfrak P(n))=\{ d_{(n)}=0 <  d_{(n-1, 1)}=n-1 <\ldots <d_{(1,\ldots, 1)}=n(n-1)/2 \} . \end{equation}
 The map $\lambda\mapsto d_\lambda: \mathfrak P(n)\to \mathbb N$ is injective  only when $ n\leq 5$. 

$d(\mathfrak P(2))= \{0<1\}$.

$d(\mathfrak P(3))= \{0<2 =d_{(2,1)}< 3\}$.

$d(\mathfrak P(4))= \{0<3 =d_{(3,1)}  < 4=d_{(2,2)}< 5 =d_{(2,1,1)}<6\}$.

$d(\mathfrak P(5))= \{0<4 =d_{(4,1) } <  6= d_{(3,2)} < 7=  d_{(3,1,1)}  < 8=d_{(2,2, 1)} <  9=d_{(2,1,1, 1)} <10\}$.

$d(\mathfrak P(6))= \{0<5 =d_{(5,1)}  < 8=d_{(4,2)} < 9=d_{(4,1,1)}  =d_{(3,3)}   < \ldots  < 15\}$.

 \section{ Nilpotent orbital integrals} Assume $R=\mathbb C$. The   nilpotent orbital integral  of the zero nilpotent orbit $\{0\}$ is the value at $0$,
 $$\mu_{\{0\}}(\varphi)=\varphi(0)  \ \ \ (\varphi \in C_c^\infty (\mathfrak g;\mathbb C)).$$
    Let $\mathfrak O$  be a non-zero nilpotent orbit of $G$  and $ \lambda\in \mathfrak P(n) \setminus \{(n)\}$ such that $\mathfrak O= \mathfrak O_\lambda$.     The nilpotent orbital integral of $\mathfrak O$  is a linear form sending $\varphi \in C_c^\infty (\mathfrak g;\mathbb C)$ to
 \begin{align}\label{g1=0}  \mu_{\mathfrak O_\lambda} (\varphi)&= \int_{ \mathfrak n_\lambda}\varphi_{K_0} (Y ) \,  dY  \\    \label{phiK}\varphi_{K_0}(Z) &=\int_{K_0}\varphi(k Zk^{-1}) dk  \ \  \text{for} \ Z\in \mathfrak g\end{align}
$dY$ and  $dk$ are the  Haar measures 
on  $\mathfrak n_\lambda$  and   $K_0 $ given in \S \ref{not}. 
For  \eqref{g1=0}, see \cite{Howe74} when $D=F$, \cite{L04} for $D$ general.

\subsection{}Homogeneity.  
 For $t\in F^*, \varphi \in C_c^\infty (\mathfrak g)$, write $\varphi_{t}(Z)=\varphi(t^{-1}Z) $  for  $Z\in \mathfrak g$.  
\begin{proposition} \label{homo}  
The  nilpotent integral orbital   of  $\mathfrak O$  satisfies  the  homogeneity relation:    $$ \mu_{\mathfrak O }(\varphi_{t})= |t|_F^{d(\mathfrak O) } \, \mu_{\mathfrak O }(\varphi),  \ \ \ \dim_F(\mathfrak O)=2d(\mathfrak O). $$
 \end{proposition} 
 
 \begin{proof} For $ \lambda\in \mathfrak P(n) \setminus \{(n)\}$, we have $d(\mathfrak O_\lambda )= \dim_F \mathfrak n_\lambda $  by \eqref{dlambda}  and  by \eqref{g1=0}
 $$\mu_{\mathfrak O_\lambda}(\varphi)= \int_{\mathfrak n_\lambda}  \varphi_{K_0} (Y )  \, dY=|t|_F^{\dim_F \mathfrak n_\lambda } \int_{\mathfrak n_\lambda  }  ( \varphi_{K_0} (  tY ) \, dY= |t|_F^{\dim_F \mathfrak n_\lambda  } \mu_{\mathfrak O_\lambda  }(\varphi_{t^{-1}}). 
$$ \end{proof} 

For  a nilpotent orbit $\mathfrak O$ of $G$ and a lattice $\mathfrak L$  in $\mathfrak g$, we denote by  $\mu_{\mathfrak O, \mathfrak L}$ the restriction of  $\mu_{\mathfrak O}$ to $C_c^\infty (\mathfrak g/\mathfrak L;\mathbb C)$ (identified to the functions on $\mathfrak g$ invariant by translation  by $ \mathfrak L$). The homogeneity implies (\cite{HC78} Lemma 14  when the characteristic of $F$ is $0$): 

\begin{corollary} \label{LI} For any lattice $\mathfrak L$  in $\mathfrak g$, the linear forms  $\mu_{\mathfrak O ,\mathfrak L}$ of $C_c^\infty (\mathfrak g/\mathfrak L;\mathbb C)$ for the nilpotent orbits $\mathfrak O$ of $G$ are linearly independent.
 \end{corollary}
 \begin{proof} For any  $d\in \mathbb N$, let $  \mathfrak N_d$ denote the union of nilpotent orbits of dimension $\leq d$.  Any nilpotent orbit  $\mathfrak O$ of dimension $d>1$ is open in $ \mathfrak N_d$ and $\mathfrak O\cup  \mathfrak N_{d-1}$  is closed.  We   choose:
 
 a) $\varphi_{\mathfrak O}\in C_{\mathbb C}^\infty (\mathfrak g;\mathbb C)$ such that 
 $$\mu_{\mathfrak O }(\varphi_{\mathfrak O'})=\begin{cases}1 
   \ \text{if} \  \mathfrak O=\mathfrak O'\\
 0  \ \text{if} \  \mathfrak O\neq \mathfrak O'
 \end{cases},
 $$
by induction on $\dim \mathfrak O$.

b)  a  lattice $\mathfrak L_0$   in $\mathfrak g$ such that  $\varphi_{\mathfrak O}\in C_c^\infty (\mathfrak g/\mathfrak L_0 ;\mathbb C)$ for each $\mathfrak O\in G\backslash \mathfrak N$,  

c) $t\in F^*$ such that $\mathfrak L\subset t \mathfrak L_0$. 

\noindent Then, $ (\varphi_{\mathfrak O})_t$ belongs to 
 $ C_c^\infty (\mathfrak g/\mathfrak L;\mathbb C)$ and  by homogeneity.
 $$\mu_{\mathfrak O}( (\varphi_{\mathfrak O'})_t)= |t|^{d(\mathfrak O)} \mu_{\mathfrak O} (\varphi_{\mathfrak O'})= \begin{cases} |t|^{d(\mathfrak O)}  \ & \text{if} \  \mathfrak O=\mathfrak O'\\
 0  \ & \text{if} \  \mathfrak O\neq \mathfrak O'
 \end{cases}.
 $$

\end{proof}

 \subsection{}\label{Nmeas} Fourier transform. The bilinear map $(Z, Y)\mapsto \trd(ZY): \mathfrak g\times  \mathfrak g\to F$ is non degenerate.
Let  $\psi:F\to \mathbb C^*$ be a non-trivial  additive character on $F$.
 The Fourier transform in $C_c^\infty (\mathfrak g;\mathbb C )$ with respect to $\psi$  and  the Haar measure $dZ$ (fixed in \S1) is the  endomorphism of $ C_c^\infty (\mathfrak g;\mathbb C)$:
\begin{equation}\label{Fou} \varphi \mapsto \hat \varphi (Y)= \int_{\mathfrak g} \varphi(Z)\, \psi (\trd(ZY) )\, dZ \ \ \ ( Y\in \mathfrak g, \ \varphi \in C_c^\infty (\mathfrak g;\mathbb C )).
\end{equation}
 There exists a positive real number $c_\psi>0$ such that  $\hat{\hat \varphi}(Z) = c_\psi \varphi (-Z) $ for $Z\in \mathfrak g$ \footnote{   The non-trivial additive characters  $F\to \mathbb C^*$  are $\psi^a(x)=\psi(ax), x\in F$ for $a\in F^*$. As $d (aZ)=|a|_F^{d^2n^2} dZ$, we have $c_{\psi^a}= |a|_F^{-n^2d^2} c_\psi$}. In particular
 \begin{equation}\label{Fou}\int _{\mathfrak g}\hat \varphi (Y) dY = c_\psi \varphi (0).\end{equation} 
For an $O_F$-lattice $\mathfrak L$ in  $\mathfrak g$,  the Fourier transform of $ 1_{\mathfrak L}$ is $ \vol ( \mathfrak L, dZ) \, 1_{\mathfrak L^*_\psi}$
where
    $$ \mathfrak L^*_\psi=\{Z\in M_n(D) \ | \  \psi (\trd(Z\mathfrak L))) =1\}= \{Z\in M_n(D) \ | \  \trd(Z\mathfrak L))) \subset \Ker (\psi)\}.$$
\begin{example}\label{expsi}When $\psi$ is trivial on $P_F$ and not on $O_F$, $M_n(O_D)_\psi^*=M_n(P_D )$ (\cite{W67} X, \S 2, Proposition 5).
\end{example}
For an  open  subset $\mathfrak C$ of $ \mathfrak g$, the extension by zero embeds  $C_c^\infty(\mathfrak C;\mathbb C)$ into $C_c^\infty(\mathfrak g;\mathbb C)$.

  \begin{proposition}\label{LIhat} 
  Let  $\mathfrak C$ be an open  neighborhood of zero in $\mathfrak g$. 
  The linear forms 
  $$\varphi\mapsto  \mu_{\mathfrak O} (\hat {\varphi }) : C_c^\infty(\mathfrak C;\mathbb C)\to  \mathbb C $$ for  $\mathfrak O\in G\backslash \mathfrak N$, are linearly independent.  
  \end{proposition}
  
 \begin{proof}This follows from the linear independence of the $\mu_{\mathfrak O, \mathfrak L}$ for any lattice $\mathfrak L$ (Corollary \ref{LI})  (\cite{HC78} corollary of Lemma 14).
 \end{proof}
\subsection{}  Let  $\mathfrak O$ be a nilpotent orbit of $G$ and   $\psi$ a non-trivial smooth character of $F$.  We compute 
  the  nilpotent orbital integral $ \mu_{\mathfrak O } (\hat{\varphi}) $  \eqref{g1=0}
 of the Fourier transform $\hat{\varphi}$ with respect to $\psi$ of  
 $\varphi \in C_c^\infty (\mathfrak g; \mathbb C)$. 
  Let $\lambda $ be the partition of $n$ such that $\mathfrak O=\mathfrak O_\lambda$. Write $(P,M,N)$ for $(P_\lambda, M_\lambda, N_\lambda)$. The bilinear map $(Y, Y^-)\mapsto \trd(YY^-): \mathfrak n\times  \mathfrak n^-\to \mathbb C$ is non degenerate  because  $\trd(YW)=0$ for $Y\in \mathfrak n, W\in \mathfrak p$.  
  The  corresponding 
 Fourier transform  with respect to $\psi$  is the linear map : $$\varphi_2\mapsto \hat \varphi_2(Y)=\int_{\mathfrak n^-}\varphi_2( Y^-  )\psi (\trd(YY^-))\,  dY^- : C_c^\infty( \mathfrak n^-;\mathbb C)\to C_c^\infty( \mathfrak n;\mathbb C).$$ There exists a positive real number $c_{\psi, \mathfrak n}$ such that 
 $$\int_{ \mathfrak n}  \int_{ \mathfrak n^-} \varphi_2( Y^-  )\psi (\trd(Y Y^-))\,  dY^- \,   dY = \int _ { \mathfrak n} \hat \varphi_2 ( Y ) dY =  c_{\psi, \mathfrak n}\,  \varphi_2(0).$$ 

  For $ \varphi\in C_c^\infty(\mathfrak g;\mathbb C)$ of Fourier transform  $\hat{\varphi}$ with respect to $\psi$, 
put 
  \begin{equation}\label{hatmuo}{\bf \hat{\mu}_{\mathfrak O} (\varphi) =
     \mu_{\mathfrak O }(c_{\psi, \mathfrak n}^{-1}\, \hat{\varphi})}.
     \end{equation}
  \begin{proposition}\label{hatmu}   
We have $ \hat{\mu}_{\mathfrak O} (\varphi) = \int_{ \mathfrak p  }\varphi_{K_0}(  W) \, dW $  .
   \end{proposition}
This  was proved only ``for some Haar measures''  when $D=F$ and  the  characteristic  of $F$ is $0$   \cite{Howe74}. 
The proposition follows from the next three lemmas where  $\varphi\in C_c^\infty (\mathfrak g;\mathbb C)$. 
    
\begin{lemma}\label{Hl} 
$ \int_{\mathfrak p}   \int_{ \mathfrak n}  \int_{\mathfrak n^-}\varphi ((Y^-+W) )\psi (\trd(YY^-))\,  dY^- \,   dY  \, dW= c_{\psi, \mathfrak n}\,  \int_{ \mathfrak p }\varphi( W)  \  dW.$
\end{lemma}
 \begin{proof} We have $C^\infty_c( \mathfrak g;\mathbb C)= C^\infty_c( \mathfrak p;\mathbb C)\otimes C^\infty_c( \mathfrak n^-;\mathbb C) $. For  $\varphi_1\in   C^\infty_c( \mathfrak p;\mathbb C), \varphi_2\in   C^\infty_c( \mathfrak n^-;\mathbb C)$ and $\varphi\in C^\infty_c( \mathfrak g;\mathbb C)$ such that $\varphi (Y^-+W)=\varphi_1 (W)\varphi_2(Y^-)$ for $Y^-\in  \mathfrak n, W\in  \mathfrak p$, we have
$$ \int_{ \mathfrak p}\int_{ \mathfrak n} \int_{ \mathfrak n^-}\varphi ((Y^-+W) )\psi (\trd(YY^-))\,  dY^- \,   dY  \, dW= c_{\psi, \mathfrak n}\,\int_{ \mathfrak p }\varphi_1 (W) \varphi_2(0) \, dW=
c_{\psi, \mathfrak n}\, \int_{ \mathfrak p }\varphi  (W)  \, dW.
$$
 \end{proof}
 \begin{lemma} \label{HP} The  integration over $\mathfrak n$  of the Fourier transform  
  is integration over   $\mathfrak p$: 
  \begin{align}\label{N} 
 \int_{ \mathfrak n }\hat {\varphi}(Y) \, dY= c_{\psi, \mathfrak n}\, \int_{ \mathfrak p  }\varphi (  W) \, dW  . \end{align}
   \end{lemma}
   
    \begin{proof}  The left hand side of \eqref{N} is
 $$
    \int_{ \mathfrak n } \int_{\mathfrak g}\varphi(Z )\psi (\trd( Y    Z ))\,  dZ \,  dY     =   \int_{ \mathfrak n} \int_{ \mathfrak p} \int_{ \mathfrak n^-}\varphi(Y^-+W) \psi (\trd(Y (Y^-+W))\,  dY^- \, dW \,   dY 
$$
because $dZ= dY^- \, dW$, and as  $\trd(YW)=0$ for $Y\in \mathfrak n, W\in \mathfrak p$ $$=     \int_{ \mathfrak n} \int_{ \mathfrak p} \int_{ \mathfrak n^-} \varphi(Y^-+W)  \psi (\trd(YY^-))\,  dY^- \, dW \,   dY=  c_{\psi, \mathfrak n}\, \int_{ \mathfrak p }\varphi( W)  \  dW 
  $$
because we  can invert the integrals on $\mathfrak n$ and on  $\mathfrak p$ \footnote{taking $\varphi= \varphi_1 \varphi_2$ as above one wants to compute the integral on $\mathfrak n$ then on $\mathfrak p$ of $ \varphi_1 (W)\hat \varphi_2(Y)$ and we can exchange the integrals because both functions have compact support }  and by Lemma \ref{Hl}. \end{proof}
   \begin{lemma}  
  The Fourier transform of $\varphi_{K_0}$ is $(\hat {\varphi})_{K_0} $ for 
 $ \varphi \in C_c^\infty (\mathfrak g;\mathbb C)$.
  \end{lemma}
 \begin{proof} Write $K=K_0$. Then $(\hat {\varphi})_K(Y) = \int_{  K }\hat {\varphi}(k Y k^{-1}) \, dk $ for $Y\in \mathfrak g$  is equal to 
    $$ 
 \int_{ K }\int_{ \mathfrak g}\varphi(Z)\psi (\trd(k Y k^{-1}Z))\,  dZ \, dk 
   =  \int_{K \times \mathfrak g}\varphi(k Z k^{-1})\psi (\trd(k Y k^{-1}\, k Z k^{-1}))\,  dZ, \, dk  $$
  because $dZ$ is $K$-invariant. This is      
    \begin{align*} & \int_{K \times \mathfrak g}\varphi(k Z k^{-1})\psi (\trd(k Y  Z k^{-1}))\,  dZ \, dk = \int_{K \times \mathfrak g}\varphi(k Z k^{-1})\psi (\trd( Y  Z ) )\,  dZ \, dk\\
   &= \int_{ \mathfrak g}\varphi_K( Z)\psi (\trd( Y  Z ) )\,  dZ.
   \end{align*}
        \end{proof}

\section{Trace of an admissible representation and parabolic induction}   
  
\subsection{} Let $R$ be a field of characteristic  $\charf_ R\neq p$ and $dg$ a Haar measure on $G$ with values in $R$. Let  $\pi\in \Rep_{R}^{\infty}(G)$ be an admissible  representation of $G$ on an $R$-vector space $V$. The   linear endomorphism  of $V$
 \begin{equation}\pi (f(g)dg)=\int_Gf(g)\pi (g) dg
  \end{equation}
  has a finite rank. Its trace is  an invariant  $R$-distribution on $G$
    $$\trace(\pi):f  \mapsto \trace (\pi (f(g)dg), \ \  f\in C_c^\infty(G;R), $$
  called the character of $\pi$.
  
   The characters of the  irreducible   smooth  complex representations  of $G$ are linearly independent (\cite{V96} I.6.13 where  $c=0$ should be $0$).

   For any exact sequence $0\to \pi_1\to \pi\to \pi_2\to 0$ of admissible $R$-representations of $G$,  $\trace(\pi)= \trace(\pi_1)+ \trace(\pi_2)$.   
    Any finite length smooth   $R$-representation of $G$ is admissible.           By the universal property of Grothendieck groups,   the  character induces a linear map from the Grothendieck group 
   $\Gr_{R}^{\infty}(G)$ of $\Rep_R^{\infty, f}(G)$ 
 to the  space of 
    invariant    $R$-distributions on $G$. 
    
  For any open compact subgroup $K$ of $G$, the restriction to $K$ 
 induces a linear map
 \begin{equation}\nu\mapsto \nu|_K: \Gr_{R} ^{\infty}(G)\to \Gr_{R} ^{\infty}(K)  \end{equation}
 from   $\Gr_R^{\infty}(G)$  to  the  Grothendieck group $\Gr_R^{\infty}(K)$   of admissible smooth $R$-representations of $K$.  When $K$ is a pro-$p$ group,
the category  $\Rep_{R} ^{\infty}(K)$ is semi-simple. 
    
 \subsection{} Parabolic induction. Let $R$ be a field and 
 $P$   a  parabolic subgroup of $G$  of  Levi subgroup $M$    and unipotent radical $N$.
 The  parabolic induction  $\ind_P^G: \Rep_R^\infty(M)\to \Rep_R^\infty(G)$
 sends $(\sigma , W) \in \Rep_R^\infty(M)$  to  $(\ind_P^G(\sigma), V)\in  \Rep_R^\infty(G)$ where $V$ is the space of functions $f:G\to W$ right invariant by some open subgroup of $G$ and satisfying $f(pg)=\tilde \sigma (p)f(g)$ for $(p,g)\in P\times G$ and    $\tilde \sigma $  is the inflation to  $P$  of $\sigma$. 
 It is  an  exact functor   respecting admissibility and finite length.

 Replacing $P$ by a $G$-conjugate does not change the isomorphism class of $ \ind_P^G(\sigma)$ and a $G$-conjugate of $P$ contains $B$.  
 
We  suppose in this section that  $B\subset P$. This implies
 $G=K_0P=PK_0=K_0P^-=P^-K_0$ where $K_0=GL_n(O_D)$ and  $P^-=MN^-$ the opposite parabolic subgroup with respect to $M$.

 The  parabolic induction of  the trivial $R$-character of  $M$
$$\pi_P=\ind_P^G 1$$  
will play an important role.  
 As our parabolic induction is   not normalized,  $[\pi_P]\in \Gr_R^\infty (G)$ depends    on the choice of $P$ of Levi  $M$.
\begin{lemma} \label{piK0} Assume $\charf _R\neq p$ and let  $P'$ be  a  parabolic subgroup  of $G$ associated to $P$.  The representation $\pi_P $  has the same restriction to $K_0$ as  $\pi_{P'} $. 
\end{lemma}
 \begin{proof}\footnote{This proof suggested by the referee simplifies our original proof using \cite{MS14}}  Let $R^{ac}$ be an algebraic closure of $R$. In the group  of unramified smooth $R^{ac}$-characters of $M$, the set of $\chi$ such that $\ind_P^G \chi$ is irreducible is Zariski dense \cite[Theorem 1.2]
 {D05}. There exist unramified  smooth $R^{ac}$-characters $\chi$ and $\chi'$ of $M$ such that the $R^{ac}$-representations $\ind_P^G \chi$ and $\ind_{P'}^G \chi'$ are irreducible and isomorphic
  \cite[Lemma 4.13]{D09}. Let $R'$ be the finite extension of $R$ generated the values of $\chi$ and $\chi'$.
  The $R'$-representations  $\ind_P^G \chi$ and $\ind_{P'}^G \chi'$ are irreducible and isomorphic. 
 We deduce that  the restriction to $K_0$ of the $R'$-representations $\pi_P $ and $\pi_{P'} $ are isomorphic. As  $R$-representations of $K_0$, $\oplus^r \pi_{P'} \simeq \oplus^r \pi_{P}$ where $r=[R:R']$.  For any $j\geq 1$,  taking the invariants by $K_j$, the finite dimensional representations  $\oplus^r(\pi_{P'})^{K_j}$  and $\oplus^r(\pi_{P})^{K_j}$   of the finite group $K_0/K_j$ are isomorphic. By Krull-Remak-Schmidt, $(\pi_{P})^{K_j}\simeq(\pi_{P})^{K_j}$. As this  is true for any $j$, we have $\pi_{P'}\simeq \pi_P$.
 \end{proof}
\subsubsection{} 
   When $R=\mathbb C$  and $ \sigma\in \Rep_\mathbb C^\infty(M)$ is admissible, we  compute the character of $\ind_P^G(\sigma)$ in terms of the character of $\sigma$.

 \begin{lemma} \label{Sf}  For $f\in C^\infty_c(G,\mathbb C)$, the  function 
$   Sf(m)= \int_N \int_{K_0} f(k mn k^{-1}) dk dn  $ on $M$
   belongs to $C^\infty_c(M,\mathbb C)$.  
  \end{lemma}
  \begin{proof}  The normal open compact subgroups $K$ of $K_0$ form a  fundamental  system of neighborhoods of $1$ in $G$ and for $g\in G$ the open compact sets $KgK$ form a  fundamental  system of neighborhoods of $g$ in $G$. For $g\in G$  and $m\in M$,   $m^{-1}KgK\cap N$ is   open in  $N$.  The  set  of   $m\in M $  such that  $m^{-1}KgK\cap N\neq \emptyset$ is  open compact in $M$ {\color{blue}\footnote{ $P\cap KgK$ is compact and the quotient map $P\to M$ is continuous}},  $   S1_{KgK}$ is $0$ outside of  this set   and   $ S1_{KgK}(m) = \vol (m^{-1}KgK\cap N, dn) \ \ \     \text{for} \ m^{-1}KgK\cap N\neq \emptyset.$
  \end{proof}
 \begin{remark}\label{S1} For  a normal open compact subgroup $K$ of $K_0$  such  that  
  $K\cap P= (K\cap M)(K\cap N)$,  $S1_{K}=  \vol (K\cap N, dn)1_{M\cap K}$. For $f\in C^\infty_c(G,\mathbb C)$ with  $\Supp f \subset K$, then  $\Supp Sf \subset K\cap M$.
  \end{remark}
   
\begin{proposition}\label{proind} We have  
$\trace ( \pi(f(g) dg) )= trace (\sigma (Sf (m) dm))$ for  $\sigma \in \Rep_{\mathbb C}^{\infty}(M)$ admissible, $\pi=\ind_P^G(\sigma)$,  and $(f,Sf)$ as in Lemma  \ref{Sf}.
 \end{proposition}
  
\begin{proof}  a) Preliminaries.  As $G=PK_0$,  a function in the space $V$ of $\pi$  is determined by its  restriction   to $K_0$, and $\pi|_{K_0}\simeq \ind_{P\cap K_0}^{K_0}( \sigma|_{M\cap K_0})$. Denote $V|_{K_0}$ the restrictions to $K_0$ of the functions in $V$. 
Let  $W$ denote the space of $\sigma$ and  $\rho$ the  action of $K_0$ on $C^\infty(K_0;W) $ by  right translation. We identify  $C^\infty(K_0;W)$ and $C^\infty(K_0;R)\otimes_R W$. Then  $( \ind_{P\cap K_0}^{K_0}( \sigma|_{M\cap K_0}), V|_{K_0})$ is  a subrepresentation of $(\rho, C^\infty(K_0;R)\otimes_R W)$.   
Let $dx$ denote the restriction to $P\cap K_0$ of  $dp$ (equal to the restriction  of $dk$).
 The  map 
$B:(\rho, C^\infty(K_0;R)\otimes_R W)\to (\ind_{P\cap K_0}^{K_0}( \sigma|_{M\cap K_0} , V|_{K_0})$  
$$ B(h\otimes w)(k)=\vol (P\cap K_0, dx)^{-1} \int_{P\cap K_0} h (x^{-1}k)  \tilde \sigma(x) (w) dx \ \ \ \ (h\in C^\infty(K_0;R), w\in W,k\in K_0),$$
  is   a $K_0$-equivariant  projection.    The function  $B(h\otimes w)$ on $K_0$  extends to a function $F_{h,w}\in V$ 
$$F_{h,w}(pk)=  \vol (P\cap K_0, dx)^{-1} \int_{P\cap K_0} h(x^{-1}k) \tilde  \sigma(px) (w) dx\ \ \ ((p,k)\in P \times K_0).$$ 
  
b) Choose a normal open pro-$p$ subgroup $K$ of $K_0$ such that $f $ is binvariant by  $K$. The endomorphism $ \pi (f(g) dg )$ of $V$   restricted to $V^K$   is  an endomorphism $A$ of $V^K$ of trace  $\trace(A)=\trace( \pi (f(g) dg ) $.  Choose a disjoint decomposition $K_0=\sqcup_i y_i K$.  The $1_{y_iK}$ form a  basis of  $ C^\infty(K_0;R)^{K}$,   the support of $B( 1_{y_iK}\otimes w)$ is  in $y_iK$,  
   and $\trace (A)$ is the trace of the endomorphism  $w\mapsto \sum_i B(F_{1_{y_iK},w})(y_i)  $ of $W$.
 For $y\in K_0$,  $B(F_{1_{yK},w})(y)$  is equal to \begin{align*} 
&\int_G f(g)  F_{1_{yK},w} (yg) dg =\int_G f(y^{-1}g)  F_{1_{yK},w}  (g) dg =\int_{K_0\times P}f(y^{-1}p^{-1} k)  F_{1_{yK},w} (p^{-1} k) dk \, dp\\
& = 
 \vol (P\cap K_0, dx)^{-1}   \int_{K_0\times P  \times P\cap K_0}f(y^{-1}p^{-1} k) h_y(x^{-1} k)  \sigma(p^{-1} x) (w) \, dk \, dp \,  dx\\
& =  \int_{K_0\times P }f(y^{-1}p^{-1} k) h_y(k)  \sigma(p^{-1}  ) (w) \, dk \, dp
=  \int_{K_0\times P }f(y^{-1}p k) 1_{yK'}(k) \tilde  \sigma(p ) (w) \, dk \, dp \\
&= \vol (K',dk)  \int_{ P }f(y^{-1}p y) \tilde  \sigma(p  ) (w)   \, dp.
  \end{align*}
     Therefore   
\begin{align*}&\sum_i B(F_{1_{y_iK},w})(y_i)  =   \vol (K,dk)  \int_{ P }\sum_i f(y_i^{-1}p y_i) \tilde  \sigma(p ) (w) \, dp 
=   \int_{K_0\times P }f(k^{-1}p k) \tilde  \sigma(p  ) (w) dk \, dp\\
&= \int_{K_0\times M\times N }f(k^{-1}mn k)   \sigma(m ) (w) dk \, dm \, dn =  \sigma (Sf(m) dm) (w).
 \end{align*}
 We deduce that the trace of $\pi (f(g) dg) $ is the trace of  $\sigma (Sf(m) dm)$.
 \end{proof}
 The set
 $\{P_\lambda \ | \ \lambda\in \mathfrak P(n)\}$ represents the parabolic subgroups of $G$ modulo association. 
 
 \begin{proposition} \label{hatpi} When  $P$ is a parabolic subgroup of $G$ associated to $P_\lambda $,  we have
$$  \trace (\pi_{P } (f(g)\,  dg))= \hat{\mu}_{\mathfrak O_\lambda} (\varphi). 
$$for $f \in  C_c^\infty(K_1;\mathbb C) $  and $\varphi \in  C_c^\infty(M_n(P_D);\mathbb C) $ such that $f(1+X)=\varphi(X)$ for $X\in  M_n(P_D)$, and $\hat{\mu}_{\mathfrak O_\lambda}$ as in  \eqref{hatmuo}. \end{proposition} 
 
\begin{proof} For $ (f,\varphi) $ as in the proposition, 
the functions \begin{align} f_{K_0}(g)=\int_{K_0} f(kgk^{-1})\, dk \   \ (g\in G),\ \  \ \varphi _{K_0}(X)=\int_{K_0} \varphi (kXk^{-1})\, dk  \ \ (X\in M_n(D)), 
\end{align}
 belong also to $ C_c^\infty(K_1;\mathbb C) ,  C_c^\infty(M_n(P_D);\mathbb C) $ and $f_{K_0}(1+X)=\varphi_{K_0}(X)$ for $X\in  M_n(P_D)$,  
 $$\int_P  f(p)\, dp=   \int_{\mathfrak p}  \varphi (W)\, dW,$$
  $ \trace (\pi_{P } (f(g)\,  dg))= \trace (\pi_{P_\lambda } (f(g)\,  dg))$ as $\pi_P= \pi_{P_\lambda}$  on  $K_0 $(Lemma \ref{piK0}), and 
\begin{align*}\trace (\pi_P (f(g)\,  dg) )=\int_M Sf (m) \, dm = \int_P  f_{K_0}(p)\, dp=   \int_{\mathfrak p}  \varphi_{K_0}(W)\, dW=  \hat{\mu}_{\mathfrak O} (\varphi).
\end{align*}
 for $P=P_\lambda, \mathfrak O=\mathfrak O_\lambda$,  by Propositions \ref{Sf} and \ref{hatmu}.
  \end{proof}

 \begin{corollary}\label{LIpi} For any non zero map $c:\mathfrak P(n)\to \mathbb C$,   the restriction of
$$\sum_{\lambda\in \mathfrak P(n)}c(\lambda) \, [\pi_{P_\lambda}]\in \Gr_{\mathbb C}^{\infty}(G)$$ 
to  an arbitrary  open compact subgroup $K$ of $G$ is not $0$.
\end{corollary}
  \begin{corollary}\label{LIpi} For any non zero map $c:\mathfrak P(n)\to \mathbb C$,   the restriction of the invariant $\mathbb C$-distribution on $G$
$$\sum_{\lambda\in \mathfrak P(n)}c(\lambda) \, \trace ( \pi_{P_\lambda})$$ to  an arbitrary  open compact subgroup $K$ of $G$ is not $0$.
\end{corollary}

\begin{proof} By Propositions \ref{hatpi} and  \ref{LIhat}, the characters of $\pi_{P_\lambda}$ are linearly independant on any neighborhood of $1$, 
because their values on $f\in C_c^\infty(K_1;\mathbb C)$  are  the Fourier transforms of the nilpotent orbital integrals of $\mathfrak O_\lambda$ on $\varphi\in C_c^\infty(M_n(P_D);\mathbb C) $ when $f(1+X)=\varphi (X) $. \end{proof}

   \section{Complex representations of $G$ near the identity}

\subsection{}\label{germ} By \cite{HC78} when  $\charf_ F=0$ (for any reductive $p$-adic group) and (\cite{L04} Proposition 4.3 with $E=F$), 
any  non-zero representation $\pi\in \Rep_{\mathbb C}^{\infty, f}G$ non-zero  $\pi$ has a {\bf  germ expansion} of map $c_\pi$ on  $K_\pi $, meaning that:

There exists a  map $c_\pi: G\backslash \mathfrak N\to   \mathbb C$ (the   coefficient map) and  an  open subgroup $K_\pi$ of $K_1=1+M_n(P_D)$
 such that  \begin{align}\label{tm}\trace (\pi(f(g) dg))= \sum_{\mathfrak O\in G\backslash \mathfrak N} c_{\pi }(\mathfrak O )\, \hat{\mu} _{\mathfrak O} (\varphi) 
\end{align}
 for   $f \in C_c^\infty(K_\pi;\mathbb C), \varphi \in C_c^\infty(M_n(P_D);\mathbb C) $ such that $f(1+X)=\varphi(X)$ for $X\in M_n(P_D)$.

It is convenient to see   $c_\pi$ as a map  on the set  $ \mathfrak P(n)$ of partitions of $n$, or on the set of parabolic subgroups $P$ of $G$, 
   \begin{equation}\label{cpi}c_\pi(\lambda)=c_\pi(\mathfrak O_\lambda) =c_\pi(P)\ \text{for $\lambda\in \mathfrak P(n)$ and $P$ associated to $P_\lambda$}.
   \end{equation}
 For example, $c_\pi((n))= c_\pi(\{0\})=c_\pi(G)$. By Proposition \ref{hatpi}, we have for $f\in C_c^\infty(K_\pi; \mathbb C)$,
 \begin{align}\label{trm}
\trace (\pi(f(g) dg))  = \sum_{\lambda \in  \mathfrak P(n)} c_{\pi }(\lambda)\, \trace (\pi_{P_\lambda}(f dg)) =  \sum_{P } c_{\pi }(P)\, \trace (\pi_{P}(f dg)).
 \end{align}
the last sum is over   a system of representatives $P$ of the parabolic subgroups of $G$ modulo association. 
We list some properties of the  map $c_\pi$ for $\pi \in \Rep_{\mathbb C}^{\infty, f}(G)$. 
  \begin{itemize}
  \item The map $c_\pi$ is unique by Corollary \ref{LIpi} 
  and  is not $0$ because  \begin{equation}\dim_{\mathbb C}\pi ^K= trace (\pi (1_K \vol (K,dg)^{-1} \, dg)\neq 0 \ \  \text{for small open subgroups  $K $ of $K_\pi$.}
\end{equation}  
\item  Two representations $\pi, \pi' \in \Rep_{\mathbb C}^{\infty, f}(G)$ have the same
 coefficient map  if and only if their  restrictions  to some open compact subgroup of $G$ are isomorphic, because the linear forms $\hat{\mu} _{\mathfrak O}$ restricted to 
 $C_c^\infty(-1+K_\pi; \mathfrak C)$ are linearly independent (Proposition \ref{LIhat}).
 
\item  In particular, 
 \begin{equation}\label{twist}c_\pi  = c_{\pi \otimes \chi}
 \end{equation}  for any  smooth character $\chi$ of $G$,   because $\chi$ is trivial on some open compact subgroup. 
 \item The map $c_\pi$  depends only on the image $[\pi]$ of $\pi$ in the Grothendieck group $\Gr_\mathbb C^\infty(G)$.  It passes to a  linear  map $\nu\mapsto  c_\nu$ on the  Grothendieck group  $\Gr_{\mathbb  C}^{\infty} (G )$ such that $c_\pi=c_{[\pi]}$ for $\pi\in \Rep_{\mathbb C}^{\infty,f}(G)$. But $c_\nu=0$ does not imply $\nu=0$. For example, $c_\nu=0$ for  $\nu = [\ind_{P_\lambda}^G 1]- [\ind_{P_\lambda}^G \theta]$ when $\theta$ is any unramified character of $M_\lambda$.

  \item When $\pi$ is finite dimensional, it is trivial on some $K_\pi\subset K_1$ hence 
 \begin{equation} \label{cfd} c_\pi ((n))=\dim_{\mathbb C} \pi, \ \ \ c_\pi (\lambda)  =0 \ \ \text{for} \ \lambda\neq (n). 
 \end{equation}
 Conversely,  if $c_\pi (\lambda)=0$ for  $\lambda \neq  (n) $  then 
\begin{equation}\label{lce0}
 \trace (\pi(f(g) dg))= c_{\pi }(\{0\})\, \hat{\mu}_{\{0\}}(\varphi) = c_{\pi }((n))\, \int_G f(g) dg 
\end{equation}
for $(f,\varphi) $ as in \eqref{tm}. Hence  $\dim_{\mathbb C}\pi^K  = c_{\pi }((n))$ for  any open subgroup $K$ of $K_\pi$, so $\pi$ is finite  
dimensional.
 
\item When $D\neq F$, a finite dimensional irreducible smooth representation  of $D^*$ may have dimension  $>1$, but:
\end{itemize}   
 \begin{lemma}\label{fd} When  $R$ an algebraically closed field, $D=F$ or $n>1$,  then  a finite dimensional irreducible $R$-representation   of $G$ is of the form $\pi= \chi \circ nrd$ for some   $R$-character $\chi$  of $F^*$.\end{lemma}
 
\begin{proof}
This clear  when $G=F^*$ because $F^*$ is commutative and the Schur's lemma is valid for $G$. When $n>1$,  then  $\Ker(\pi)$ is an open subgroup of $G$, and in particular contains an open subgroup of $U$.
But $\Ker(\pi) $ is also normal in $G$, so it contains $U$, and all the conjugates of $U$.
Those conjugates generate $\Ker (\nrd)$, so $\pi$  factors through $\nrd$  implying the lemma.  
\end{proof}
 
\subsection{} \label{inte}  We revert to $R=\mathbb C$ and show  that the values of $c_\pi$ are integers   
  (proved in \cite{Howe74}  when $D=F$ has  characteristic $0$ and $\pi$  is irreducible supercuspidal).  The key of the proof is the next lemma \ref{lemma6}  inspired by Howe (\cite{Howe74} Lemma 6).
  
 \medskip For a   partition $\lambda=(\lambda_1,\ldots, \lambda_r)$  of   $n$,  let $A_\lambda $ be the matrix of the endomorphism of the right $D$-vector space $D^n$  operating on the canonical basis $e_1, \ldots, e_n$ by sending 
 $e_1,\ldots, e_{\lambda_1}$ to $0$,  $e_{\lambda_1+1},\ldots, e_{\lambda_1+\lambda_2}$ to 
$e_{ 1},\ldots, e_{\lambda_2}$, and $e_{\lambda_1+ \ldots +\lambda_i +j} $ to  $e_{\lambda_1+ \ldots +\lambda_{i-1} +j} $ for $  i=2,\ldots, r-1, j =1,\ldots, \lambda_{i +1}$. Then,  $\Ker A_\lambda ^i$  is the $D$-subspace 
 generated by $e_1, \ldots, e_{\lambda_1+ \ldots +\lambda_i }$. The parabolic subgroup of $G$ stabilizing the flag $(\Ker A_\lambda ^i)_i$ is $P_\lambda$,  and $A_\lambda \in \mathfrak n_\lambda$.
  Fixing  a character   $\psi$  of $F$ trivial on $P_F$ and not on $O_F$,  for   an integer   $j\geq 1$, let $\xi_\lambda$ be the character of  $K_j=1+ M_n(P_D^j) $  trivial on $K_{2j}$  defined by
 \begin{equation}\label{xi}\xi_\lambda(1+x)= \psi \circ \trd (A_\lambda  \,  p_D^{1-2j} x) \ \ \text{  for}  \ x\in  M_n(P_D^j).
 \end{equation} 
  \begin{lemma} \label{lemma6} 
For  $\mu \in \mathfrak P(n)$, the  multiplicity $m(\xi_\lambda, \pi_{P_\mu})$ of  $\xi_\lambda$  in $\pi_{P_\mu}$ is $0$
 unless $ \lambda \geq \mu$. We have  $m(\xi_\lambda, \pi_{P_\lambda})=1$. \end{lemma}

 \begin{proof}

For $\mu \in \mathfrak P(n)$,   $m(\xi_\lambda, \pi_{P_\mu})$ is  the cardinality of 
$$ (P_\mu\cap GL_n(O_D) )\backslash \{ k\in GL_n(O_D) \ | \ \xi_\lambda (k^{-1}(P_\mu\cap K_j) k)=1\}/K_j.$$
 Let $k\in K_0=GL_n(O_D)$. We have 
 $\xi_\lambda (k^{-1}(P_\mu\cap K_j) k)=1$ if and only if 
\begin{equation}\label{*}\xi_\lambda  (k^{-1}(1 + \mathfrak p_\mu  (P_D^j )) k)=1,
\end{equation}
where  $ \mathfrak p_\mu  (P_D^j )=  \mathfrak p_\mu\cap M_n  (P_D^j )$.  The weaker condition $\xi_\lambda  (k^{-1}(1 +  \mathfrak p_\mu (P_D^{2j-1})) k)=1$ already implies $m(\xi_\lambda, \pi_{P_\mu})=0$ unless $ \lambda \geq \mu$.
Indeed, it reads 
$\psi \circ \trd (A _\lambda \,  k^{-1} \mathfrak p_\mu (O_D)k  )=1$. It depends on the images $\overline k,  \overline A_\lambda $ of $k,  A_\lambda $ in $GL_n(k_D)$  and says that   
$ \trd (\overline k\, \overline A_\lambda  \,  \overline k^{-1}  \, \mathfrak p_\mu  (k_D))=0$, that is, $\overline k\, \overline A_\lambda  \,  \overline k^{-1}\in \mathfrak n_\mu (k_D)$. Let  $0\subset W_1 \subset \ldots $ be the flag of $k_D^n$ whose stabilizer is $P_\mu(k_D)$. Then $\overline k\, \overline A_\lambda \,  \overline k^{-1}\in  \mathfrak n_\mu (k_D)$ means  $\overline k\, \overline A_\lambda  \,  \overline k^{-1} (W_i)\subset W_{i-1}$  for $i\geq 1$, and  in particular that  $\Ker (\overline k\, (\overline A _\lambda)^i \,  \overline k^{-1}) =\overline k (\Ker (\overline A_\lambda )^i )$  contains $W_i$. As $\dim_D W_{i+1}-\dim_D W_i =\mu_i$, one obtains $\lambda_1+\ldots +\lambda_i \geq \mu_1+\ldots +\mu_i$ for each $i$, that  is $\lambda \geq \mu$.

 Suppose now $\mu=\lambda$. We  prove that  \eqref{*} is equivalent to $k\in P_\lambda (O_D) K_j$.
 By its definition  $\xi_\lambda$ is trivial on $1+\mathfrak p_\lambda(P_D^j)$  because $A_\lambda \in  \mathfrak n_\lambda$  hence $\trd (A_\lambda  \mathfrak p_\lambda)=0$, so $P_\lambda (O_D) K_j$ does satisfy  \eqref{*}.  Conversely, $B_\lambda =A_\lambda p_D^{1-2j}\in \mathfrak n_\lambda (P_D^{1-2j})$. The condition
\eqref{*} means  that 
$\trd (B_\lambda  k^{-1}\mathfrak p_\lambda (P_D^j) k) \in P_F $  
and implies 
 $$B_\lambda=k^{-1}Xk+Y, \ \text{where} \ X\in  \mathfrak n_\lambda, Y\in  M_n(P_D^{1-j}).$$
  Indeed, writing $ kB_\lambda  k^{-1} = X +Y$ with $X\in \mathfrak n_\lambda , Y\in  \mathfrak p_\lambda^-$, we have:  
  \begin{align*}&\trd (B_\lambda  k^{-1}\mathfrak p_\lambda (P_D^j)
  k)= \trd (kB_\lambda  k^{-1}\mathfrak p_\lambda (P_D^j)) = \trd  (Y\mathfrak p_\lambda (P_D^j))=  \trd  (Y M_n(P_D^j)),\\
 & \trd  (Y M_n(P_D^j))\in P_F \Leftrightarrow 
   \trd  (P_D^{j-d} Y M_n(O_D))\in O_F \Leftrightarrow  Y\in M_n(P_D^{1-j}).
   \end{align*}
 See Example \ref{expsi} for the last equivalence.
    One gets 
$B_\lambda k^{-1}= k^{-1}X+Y_1 $ with $Y_1\in M_n(P_D^{1-j})$. 
Note that  $B_\lambda\in M_n(P_D^{1-2j})$ hence also  $X$.
 We get 
 $ B_\lambda^2k^{-1}=B_\lambda k^{-1}X+B_\lambda Y_1=k^{-1}X^2+Y_1X+B_\lambda Y_1=k^{-1}X^2+Y_2$ with 
  $Y_2\in  M_n(P_D^{j+2(1-2j)})$. By induction 
$B_\lambda^i k^{-1}=k^{-1}X^i+Y_i $ with $Y_i\in M_n(P_D^{j+i(1-2j)})$ for $1\leq i \leq r$.
For   a basis vector $e\in \Ker A_\lambda ^i$, we have $X^ie=0$ because $X\in \mathfrak n_\lambda$, and $B_\lambda^i k^{-1}e=k^{-1}X^ie+Y^ie= Y^ie$. As $(A_\lambda p_D^{1-2j})^{i} k^{-1}e\in  M_n(P_D^{j+i(1-2j})e  \Leftrightarrow A_\lambda ^i  k^{-1}e\in  M_n(P_D^{j})e$,
 the coefficients of  $k^{-1}e$  on the basis vectors which are not in 
$\Ker A_\lambda ^i$ are in $P_D^j$.   This means 
$ k^{-1} \in  K_jP_\lambda(O_D)  $, what we wanted.
 \end{proof}
 We shall need more properties of   $\xi_\lambda$  in the section on Whittaker spaces.
 
\begin{lemma}\label{ouf} The normalizer of $\xi_\lambda$ in $K_0=GL_n(O_D)$ is $P_\lambda(O_D)K_j$.
\end{lemma}
\begin{proof}  For $k\in K_0$, the property   $\xi_\lambda(1+x)=\xi_\lambda(1+kxk^{-1})$  for  all $x\in M_n(P_D^j)$ means $ k^{-1}B_\lambda k- B_\lambda  \in M_n(P_D^{1-j}) $.  As in the proof of Lemma \ref{lemma6}
 one deduces $ B^i k-k B^i 
\in M_n(P_D^{j-i(1-2j)}) $ for $i\geq 1$ and one sees that $ k\in P(O_D)K_j$.
\end{proof}
 \begin{remark}There is a unique function in $\pi_{P_\lambda}$ with support $P_\lambda K_j$ and restriction $\xi_\lambda$  to $K_j$ since $\xi_\lambda$   is trivial on $1+\mathfrak p_\lambda(P_D^j)$. That function is a basis of the line of vectors in $\pi_{P_\lambda}$ transforming according to  $\xi_\lambda$  under the action of $K_j$. \end{remark}

  We  prove  now  that the $c_\pi ( \lambda)$ are integers.  By \eqref{trm}, when  $K_j=1+ M_n(P_D^j) \subset K_\pi$ and $\delta\in  \Rep_{\mathbb C}^{\infty}(K_j) $  irreducible, the multiplicity $m(\delta, \pi)$ of $\delta$   in $\pi\in \Rep_{\mathbb C}^{\infty}(G)$ satisfies \begin{align}\label{multD}
m(\delta, \pi)=  \sum_{\mu \in \mathfrak P(n)} c_{\pi }(\mu)\,m (\delta,\pi_{P_\mu}) .
\end{align}
  Lemma \ref{lemma6} and  \eqref{multD}  imply:
   \begin{align}\label{cm} 
   c_\pi (\lambda)&=m(\xi_{\lambda }, \pi) - \sum_{\mu \in \mathfrak P(n), \mu <\lambda } c_\pi(\mu) \, 
m(\xi_{\lambda }, \pi_{P_\mu}).\end{align}   In particular when  $\lambda $ is minimal in $\Supp c_\pi$, $c_\pi (\lambda)=m(\xi_{\lambda }, \pi ) $  is positive and  independent of   the choice  of $j$  such that $K_j=1+ M_n(P_D^j) \subset K_\pi$.
By upwards  induction on $  \mathfrak P(n)$ (downwards induction on the nilpotent orbits), we obtain that the $c_\pi ( \lambda)$ are integers.  
  
 \bigskip As the values  of the  map $c_\pi $ are integers, we get 
more properties: 
 \begin{itemize}
 \item  $c_{\pi}= c_{\sigma (\pi)}$  when $\sigma$ is an automorphism  of $\mathbb C$.
 \item
 For $\nu\in \Gr_R ^{\infty}(G)$, there exists  a map $  c_\nu: \mathfrak P(n)\to \mathbb Z $ and an open  subgroup $K_{\nu} $  of $G$ 
such that  $\nu$ and  $\sum_{\lambda\in \mathfrak P(n)} c_{\nu}(\lambda) \, [\pi_{P_\lambda}] \in  \Gr_R ^{\infty}(G)$  have isomorphic restrictions  to $K_\nu$.   \end{itemize}

\bigskip  When $R=\mathbb C$, the first part of Theorem \ref{1.1} is a version of  the  germ expansion.  For any $R$,  when $\pi$ satisfies the first part of Theorem \ref{1.1} we say sometimes that $\pi$ has a germ expansion with map $c_\pi$ on $K_\pi$.

 \section{Parabolic induction} \label{s:PI} 
 In this section $R$ is a field and $\charf_ R\neq p$. We prove now that the first part of Theorem \ref{1.1} implies Theorem \ref{3.3}.
  Let $P, M, (n_i), \sigma_i, r, \sigma, \pi$ as in Theorem \ref{3.3}. Write $pr:P\to M$ for the projection of kernel $N$. Given partitions $\lambda_i$
  of $n_i$ for $1\leq i \leq r$, we have the  parabolic subgroup $P_{(\lambda_i)}$ of $M$ corresponding to  the parabolic subgroups $P_{\lambda_i} $ of $GL_{n_i}(D)$.   Given functions $c_i:\mathfrak P(n_i)\to \mathbb Z$ for $1\leq i \leq r$, the function $c:\mathfrak P(n)\to \mathbb Z$ defined by 
  $$c(\lambda)= \sum \prod_{i=1,\ldots, r }c_i (\lambda_i),$$
where the sum is over
$r-$tuples of partitions $(\lambda_1, \ldots,\lambda_r) $  inducing to $\lambda$ before Theorem \ref{3.3}, is called induced by $(c_1, \ldots, c_r)$. 
  
  \begin{theorem}\label{th:71}Assume that     for  $i=1,\ldots, r$, there exists a function $c_{\sigma_i}:\mathfrak P(n_i)\to \mathbb Z$
and an open compact subgroup $K_{\sigma_i}$ of $GL_{n_i}(D)$ such that 
 $\sigma_i= \sum_{\lambda_i\in \mathfrak P(n_i)} c_{\sigma_i} (\lambda_i) \ind_{P_{\lambda_i}}^{GL_{n_i}(D)} 1$  on  $K_{\sigma_i}$. 
 Then   $$\pi= \sum_{\lambda\in \mathfrak P(n)} c _\pi(\lambda) \ind_{P_\lambda}^G 1$$ on $K_\pi$, 
where $c_\pi:\mathfrak P(n)\to \mathbb Z$ is the function induced by $(c_{\sigma_1}, \ldots, c_{\sigma_r})$ and  $K_\pi$   is 
 any open compact subgroup  of $G$ such that   
 $\cup  _{g\in P \backslash G/K_\pi}  \, pr(P \cap g K_\pi g^{-1})$ is contained in $K_{\sigma_1}\times \ldots \times K_{\sigma_r}$.\end{theorem}
\begin{proof}  The theorem follows from  the fact that  for any field $R$,    $\ind_P^G (\ind_{P_{(\lambda_i)}}^M 1)$ has the same restriction to $K_0$ than $\ind_{P_\lambda}^G 1$ by Lemma \ref{piK0}, and for
given  a open compact  subgroup  $C_M$  of $M$, there exists an open compact  subgroup  $C$ of $G$ such that
  \begin{equation}\label{KMKG}\cup  _{g\in P \backslash G/C}  \, pr(P \cap g Cg^{-1}) \ \subset \ C_M.
  \end{equation}
  The existence of $K_\pi$ follows from  \eqref{KMKG} applied to $C_M= K_{\sigma_1}\times \ldots \times K_{\sigma_r}$. 

 The restriction of a smooth $R$-representation $\sigma$  of $M$ to   $C_M$  determines the restriction of $\ind_P^G\sigma$ to  $C$,  
    $$(\ind_P^G\sigma)|_C \simeq  \oplus  _{g\in P \backslash G/C} \ind_{C\cap g^{-1}P  g }^{C}( \sigma^g) 
  $$
   where  $\sigma^g(k)= \sigma(gkg^{-1})$ for $g\in G, k\in g^{-1}P g \cap C$, and $\sigma^g$  depends only on the restriction of $\sigma$ to $pr(P  \cap gCg^{-1} )$. 
If  $\sigma'\in \Rep_R^{\infty, f} (M)$ is isomorphic to $\sigma$ on $C_M$, then $\ind_P^G \sigma'$  and $\ind_P^G \sigma$ are isomorphic on $C$.  
The same holds true for  virtual representations  $\nu, \nu' $  of $M$.
Take 
$\nu= \sigma_1\otimes \ldots \otimes \sigma_r$ and $\nu'= \nu'_1\otimes \ldots \otimes \nu'_r$ with $\nu'_i= \sum_{\lambda_i\in \mathfrak P(n_i)} c_{\sigma_i} (\lambda_i) \ind_{P_{\lambda_i}}^{GL_{n_i}(D)} 1$.
\end{proof}

     \begin{corollary} \label{variant} {\rm (Variant of Theorem \ref{th:71})}
  Assume that  $GL_{n_i}(D)$ satisfies the first part of Theorem \ref{1.1} for $i=1,\ldots, r$. 
    Then  for $\sigma\in \Rep_R^{\infty,f} (M)$, there exists  an open compact subgroup $K_\sigma$ of $M$  and a  unique map $c_\sigma:  \mathfrak P(n_1)\times \ldots \times P(n_r)\to \mathbb Z$ such that $\sigma=\sum_{(\lambda_i)\in  (\mathfrak P(n_i))} c_\sigma((\lambda)_i) \pi_{P_{(\lambda_i)}}$ on $K_\sigma$, and $\pi= \ind_P^G \sigma$ is equal to $\sum_{\lambda} c_\pi(\lambda) \pi_{P_\lambda}$ on  any open compact subgroup 
  $K_\pi$ of $G$  such that  $K_\sigma \subset \cap  _{g\in P \backslash G/K_\pi}  \, M \cap g K_\pi g^{-1}$ and $c_\pi: \mathfrak P(n) \to \mathbb Z$ is induced by $c_\sigma$.
    \end{corollary}

\begin{remark} When $G=GL_n(F)$, given partitions $\lambda_i$ of $n_i$ for $i=1,\ldots, r$,  and $\lambda\in \mathfrak P(n)$ induced by the $\lambda_i$, the nilpotent orbit $\mathfrak O_\lambda$   is  the nilpotent orbit induced by the nilpotent orbit  $\mathfrak O_{(\lambda_i)}$  of $M$ corresponding  to the $\lambda_i$, in the sense of \cite{LS79} (see \cite{Jan04}). 
If $R=\mathbb C$, $\charf_ F=0$, $p\neq 2$,  $D=F$, the formula  for $c_\pi $ follows from 
   (\cite{MW87} \S II.1.3  where $G$ is a classical group). 
   \end{remark}

       \section{Whittaker spaces}
 Our purpose in this section is to relate the coefficient map $c_\pi$ to the dimensions of the different Whittaker spaces of $\pi$ when $\pi\in \Rep_{\mathbb C}^\infty(G)$ is irreducible. We first introduce those subspaces. 
 
  The commutator subgroup of  the group $U$ of upper  unipotent matrices  is the group $U'$ of  upper  unipotent matrices with coefficients $u_{i,i+1}=0$ for $i=1,\ldots, n-1$ (use the identities $E_{a,b} E_{c,d}=E_{a,d}$ if $b=c$ and $0$ otherwise).  The map sending $(u_{i,j})\in U$ to $(u_{1,2},\ldots, u_{n-1, n})$ induces an isomorphism from  $U/U'$  to the additive group $D^{n-1}$. The action of the group $T \simeq (D^*)^n$  of diagonal matrices  by conjugation  on $U$ and on $U'$ induces an action  on $D^{n-1}$, the  diagonal matrix $\diag(a_1, \ldots, a_n)\in T$ sends $(d_1, \ldots, d_{n-1})\in D^{n-1}$ to  $(a_1d_1 a_2^{-1}, \ldots, a_{n-1}d_{n-1}a_{n}^{-1})$. 
  
 \medskip Let us fix a non-trivial smooth character $\psi$  of $F$. Then $\psi_D=\psi \circ \trd$ is a non-trivial character of $D$. 
Sending $y\in D$ to the character $\psi_D^y(x)=\psi_D (yx)$ for $x\in D$, is an isomorphism from the additive group $D$ to its group of smooth characters. Sending $y=(y_1\ldots y_{n-1})\in D^{n-1}$ to $(\psi_D^{y_1},\ldots, \psi_D^{y_{n-1}})$, is an isomorphism from $D^{n-1}$ to its group of smooth characters. The above action of $T$ 
on $D^{n-1}$ induces an action on its groups of characters, the  diagonal matrix $\diag(a_1, \ldots, a_n)$ sends $y=(y_1, \ldots, y_{n-1})\in D^{n-1}$ to  $(a_2^{-1}y_1 a_1, \ldots, a_{n}^{-1}y_{n-1}a_{n-1})$. 

Let $y=(y_1, \ldots, y_{n-1})\in D^{n-1}$, $r$ be the number of indices $i$ where $y_i=0$, and   
\begin{equation} \label{I}I=I(y) =\begin{cases} \emptyset  \ \text{ if } r=0,\\
 \{i_1 < \ldots < i_r \} \ \text{ the set of indices $i$ where $y_i=0$ if $r\neq 0$}.
 \end{cases}
 \end{equation}
The smooth character  of $U$ corresponding to $y $ is 
$$\theta _y(u)=\psi \circ \trd (X_y v)  \ \ \   u=1+v\in U,$$ where $X_y\in M_n(D) $ is the nilpotent matrix with $(y_1, \ldots, y_{n-1})$ just below the  diagonal and $0 $ elsewhere. 
The character $\theta_y$ is called {\bf non-degenerate} if  $I(y)=\emptyset$, and {\bf degenerate} otherwise. 
The character  $\theta_y$ is trivial if and only if   $I(y)=\{1,\ldots, n-1\}$. 
The group $B=TU$ is its own normalizer in $G$, so the $G$-normalizer of  $\theta$  is of the form $T_{\theta_y }U$ where $T_{\theta_y}$ is the $T$-normalizer of $\theta_y$. It is the intersection of $B$ with the commutant of $X_y$.

The element $y$ is conjugate under $T$ to the element $\delta_I\in D^{n-1}$  with coefficient $0$ in $I$ and $1$ elsewhere.  
The nilpotent matrix $X_{\delta_I}$ is  a diagonal of Jordan blocks of sizes forming   a composition $\lambda_I$  of $n$, 
 \begin{equation}\label{cI}\lambda_I=\begin{cases} (n) \ \text{ when } \ I=\emptyset,\\
 (i_1, i_2-i_1, \ldots, n- i_r) \ \text{ when } \ I\neq \emptyset.
 \end{cases}
 \end{equation} Any composition  $\lambda$  of $n$ is equal to $\lambda_I$ for a unique subset $I$ of    $I(y)=\{1,\ldots, n-1\}$.      Put $X_\lambda= X_{\delta_I}$, 
   \begin{equation}\label{thetal}\theta_{\lambda} (u)=\psi \circ \trd (X_\lambda  v)  \ \ \   u=1+v\in U,
  \end{equation} and $T_{\lambda}$ the $T$-normalizer of $\theta_{\lambda}$. The group  $ T_{\lambda}$ contains the group $ T_{(n)}=\{\diag(d,\ldots, d) \ | \ d\in D^*\}$  isomorphic to $D^*$.  

\medskip We fix a representation $\pi\in \Rep_{\mathbb C}^\infty(G)$ of space $V$. Given a smooth character  $\theta$ of $U$, we look at the space $V_\theta$ of $\theta$-coinvariants of $U$ in $V$, or at its dual,  the (Whittaker) space  of linear forms $\Lambda$ on $V$ such that 
 $\Lambda (u v)=\theta(u) \Lambda(v)$ for $u\in U, v\in V$. It is customary to say that $\pi$  has a {\bf Whittaker model} with respect to 
$\theta$ if $V_\theta\neq 0$.  Indeed any choice of non-zero linear form  $\Lambda$ on $V_\theta$  gives a non-zero
intertwining from $\pi$ to $\Ind_{U}^G(\theta)$ by sending $v\in V$ to the function taking value
$ \Lambda(gv)$ at $g \in G$; that intertwining is an embedding if $\pi$  is irreducible, hence the name ``model''. We   say that $\pi$  has a {\bf non-degenerate Whittaker model}, or that $\pi$ is {\bf generic} if $V_\theta\neq 0$ 
for some (equivalently all) non-degenerate characters $\theta$ of $U$. We say that $\pi$ has a 
Whittaker model if it has a {\bf Whittaker model} with respect to some choice of $\theta$.

Using the action of $T$ on $U$ by conjugation, we see that to analyse
the $V_\theta$ for all choices of $\theta$, it is enough to consider the $\theta_{\lambda}$ associated to the compositions $\lambda $  of $n$.  
  
 \begin{remark} 1) It is known that  if $\pi$ is irreducible then $V_\theta$ is finite dimensional (when $\theta$ is not degenerate \cite{BH02}, in  general \cite{ABS22}; these papers treat the case of  a general reductive group $G$).  The  group $T_\theta$ acts on $V_\theta$; since $T_\theta$ is not commutative if $D\neq F$, we cannot expect $V_\theta$ to have always dimension $0$ or $1$ (as when $D=F$ and $\theta$ not degenerate).
 
 2) Moeglin and Waldspurger \cite{MW87} consider more general Whittaker spaces, but ours are enough for our purpose (Theorem \ref{Wh} below). Also they use the exponential map, which is not available when $F$ has positive characteristic. Instead we use the map $X\mapsto 1+X: M_n(P_D)\to 1+M_n(P_D)$, as in  \cite{Howe74} and \cite{Rodier} when $D=F$.
 
3)  If  $\pi$ is irreducible cuspidal,  $\pi$ can only have  non-degenerate Whittaker models because  $\theta_I$ is trivial on the unipotent radical  $N_{\lambda_I}$ of the parabolic group $P_{\lambda_I}$. Hence $\pi_{\theta_I}$ is a quotient of the $N_I$-coinvariant space $\pi_{N_{\lambda_I}}$ of $\pi$. If $\pi_{N_{\lambda_I}}=0$ then $\pi_{\theta_I}=0$,   
and $N_{\lambda_I}$ is trivial if and only if 
 $I=\emptyset$.
 
 4) It is possible to extend to $GL_n(D)$ the theory of \cite{BZ77} 5.1 to 5.15 to show
that a  non-zero $\pi$ has a Whittaker model (see \cite{AH23} 3.4).
But that is  a consequence of our theorem below (Corollary \ref{cWh}).  

 \end{remark}
   
 We now prove Theorem \ref{4}  {\color{red}(}for $R=\mathbb C${\color{red})}. We can assume that $\pi$ is irreducible. We want to relate the coefficient map $c_\pi : \mathfrak P(n)\to \mathbb Z$ of  the germ expansion of  $\pi$ 
with the dimensions of  the spaces $V_{\theta_\lambda}$ for the  compositions $\lambda$ of $n$, following \cite{MW87}.   We define the {\bf Whittaker support} of $\pi$  as  the set of partitions $\mu$  of $n$ such that $V_{\theta_\lambda} \neq 0$ for some
composition $\lambda$  of $n$ with associated partition $\hat{ \mu}$ (the partition dual to $\mu$).

\begin{theorem}\label{Wh} The minimal elements in  $\Supp c_\pi$ and in the Whittaker support of $\pi$
are the same.

Let $\mu$  be a partition of $n$  minimal in  $\Supp c_\pi$ and let $\lambda$  be a composition of $n$
with  associate partition   $\hat{ \mu}$. Then $c_\pi(\mu)=\dim_{\mathbb C}V_{\theta_\lambda}$.
\end{theorem} 
Since $\pi$ has a non-zero germ expansion, the theorem implies:
\begin{corollary} \label{cWh}Any irreducible smooth complex representation $\pi$ of $G$  has a Whittaker model. 
\end{corollary}
  
\begin{remark} 

1) By the theorem $(1,\ldots, 1)$ is minimal in  $\Supp c_\pi$  if and only if  $V$ has a non-degenerate Whittaker model. This was proved when $D=F$  \cite{Rodier}.

2) $(n)$ is minimal in  $\Supp c_\pi$  if and only if $\dim_{\mathbb C}(V)$ is finite. By the theorem that  happens if and only if  $V$ has only the trivial Whittaker model.

3) In part 2 of the theorem, $\dim(V_{\theta_\lambda})$ does not depend on the choice of the composition $\lambda$
with associated partition $\hat \mu$. It is
  the multiplicity in $\pi$ of the character $\xi_\mu$ of $K_j$ defined in \eqref{xi}, if $j$  is large enough.

  \end{remark}
We turn back to the proof of the theorem. As said at the beginning of this section, our proof   is based on the
method of \cite{MW87}, replacing the exponential by  $X\to 1+X$. The starting idea
is already in \cite{Rodier}, but that paper is restricted to the non-degenerate
Whittaker models, and $D=F$. Compared to those works, we work with the germ
expansion of  $\pi$  in terms of the $\pi_{P_\lambda}$ rather than with Fourier transforms of nilpotent orbits. We find
that it simplifies matter a bit, and it is coherent with our approach.

\begin{proof} We fix a composition $\lambda=(\lambda_1,...,\lambda_r)$  of $n$. We write $\theta$  for the character $\theta_\lambda$
of  $U$ and $X$ for the lower triangular nilpotent  matrix in
Jordan blocks of size $\lambda_1,...,\lambda_r$ down the diagonal (if $I$ is the subset of  $\{1,\ldots, n-1\}$ such that $\lambda=\lambda_I$, then $X=X_{\delta_I}$). For each positive integer $j$ we define a character $\psi_j$ of $K_j = 1+ M_n(P_D^j)$ trivial on $K_{2j}$, 
\begin{equation}\psi_j (1+x)=\psi \circ \trd (Xp_D^{1-2j}x),  \ \ \ x\in M_n(P_D^j),
\end{equation} where $\psi$  is a character of $F$ trivial on $P_F$ but not on $O_F.$ In other words, $\psi_j$
is obtained, in the formula  \eqref{xi}  for $\xi_\lambda$ by replacing the matrix
  $A_\lambda$ there with the matrix $X$. 
\medskip We let $\lambda' $ the partition of $n$ obtained from $\lambda$ by putting its parts in decreasing order, and   $C$  the
matrix $A_{\lambda' \, \hat{}}$  associated as in Lemma \ref{lemma6} to the partition $\lambda' \, \hat{}$.

\begin{lemma}\label{B} The matrices   $C$   and $X$ are
conjugate  by permutation matrices (corresponding to permutations of the canonical basis of $D^n$).
\end{lemma}
\begin{proof} A suitable permutation of the canonical basis puts the blocks of $X$ in decreasing size order, and we get the matrix $X'$  analogous to $X$ but corresponding to $\lambda'$. Let us describe a permutation of the basis which conjugates
$X'$ to   $C$. 
Let $d$ be the size of the largest blocks of $X'$. Put at the end the first vectors of the blocks
of $X' $ of size $d$. Before them, put a bunch of vectors: the images under $X'$ of the previous ones,
completed with the first vectors of the blocks of size $d-1$ of $X'$, if any. Once you have the vectors
corresponding to size $i $, put before them the images under $X'$ of the already chosen vectors,
completed with the first vectors of the blocks of size $i-1$. Reaching $i=1$  completes the process.
\end{proof}
 
\begin{remark}\label{+}  By this lemma, we can apply  Lemma \ref{lemma6} to $\psi_j$. Hence,
For any positive integer $j$, one has $m(\psi_j, \pi_{P_{\lambda' \, \hat{}}})=1$ and $m(\psi_j, \pi_{P_\mu})=0$ 
 unless $\lambda' \, \hat{}\geq \mu$.  If $\lambda' \, \hat{}$  is minimal in  $\Supp c_\pi$, then we have
$c_\pi(\lambda' \, \hat{})=m(\psi_j,\pi)$   for  any  positive integer  $j$ such that the germ expansion  of $\pi$ is
valid on $K_{j}$.

 \end{remark}
 
We now turn to the Whittaker quotient $V_\theta$, approaching it (following Rodier's
initial idea) by a suitable conjugate $\psi'_j$  of $\psi_j$  and letting $j $ go to infinity.

The diagonal matrix $t=\diag(1,p_D, ...,p_D^{n-1})$ acts by conjugation on $M_n(D)$,
multiplying the $(a,b)$-coefficient $x$ of a matrix  by $p_D^a x p_D^{-b}$.
Conjugating $\psi_j $ yields
a character $\psi'_j $ of the group $K'_j = t^{2j-1} K_j  t^{-2j+1}$ which  satisfies also  Remark \ref{+}.
The group $U$ is the increasing union of $U\cap K'_j$ over $j$, whereas the  decreasing subgroups $B^- \cap K'_j$ have trivial intersection.
 The restriction of $\psi'_j $ to $K'_j\cap U$ is equal to that of $\theta$, whereas its restriction
to  $K'_j\cap B^-$  is trivial. The multiplication induces a bijection (an Iwahori decomposition):
$$(K'_j \cap U) \times (K'_j\cap B^-)\to K'_j$$
The projector $e'_j :V\to V(\psi'_j) $  of $V$ onto its
$\psi'_j$-isotypic space $V(\psi'_j) $  (which has dimension $m(\psi'_j,\pi) = m(\psi_j,\pi)$) can be obtained by first
projecting onto vectors fixed by $K'_j\cap B^-$, and then applying the projector $f_j$
$$ f_j (v) =\int_{K'_j\cap U}   \theta (u)^{-1} \, \pi(u) v \,  du , \ \ \  v\in V, $$
 with respect to the Haar measure $du$ giving measure $1$ to $K'_j\cap U$. 

We write  $p:V\to V_{\theta}$ for the projection of $V$ onto $V_{\theta}$ and  $p_j:V(\psi'_j) \to V_{\theta}$ for its restriction to $V(\psi'_j) $.

\begin{lemma}\label{l3} The map $p_j: V(\psi'_j) \to V_{\theta}$ is surjective for large $j$.
\end{lemma}
\begin{proof} Let $v\in V$.  For  large enough $j$,  $v\in V^{ K'_j \cap B^-}$ hence $e'_j(v)=f_j(v)$ and $p(e'_j(v))=p(v)$.  Lifting in that way a basis of the finite-dimensional space $V_{\theta }$  gives the
result. 
\end{proof}
\begin{lemma}\label{mum} Il  $V_{\theta}\neq 0$, then there is a partition $\mu$  in  $\Supp c_\pi$  with
$\mu \leq \lambda' \, \hat {}$.
\end{lemma}
\begin{proof}  Il  $V_{\theta }\neq 0$  is not 0, then by Lemma \ref{l3}, $V(\psi'_j) \neq 0$ for large $j$, so $\tr(\pi (e'_j))\neq 0$.
 Applying the germ expansion of $\pi$  to $e'_j$
 there is a minimal partition  $\mu$ of $n$ in $\Supp c_\pi$. By Remark \ref{+},  $c_\pi(\mu)=m(\psi_j,\pi_{P_\mu}) $ and $ \mu \leq \hat{\lambda'}$.
\end{proof}

\begin{lemma}\label{j'}  Let  $j_0$ be a positive integer such that $\pi$ has a germ expansion on $K_{j_0}$, and $j'_0=  j_0+2n-2$. If  $\lambda' \, \hat{}$  is minimal in  $\Supp c_\pi$ and  $j\geq j'_0$, then the 
endomorphism $v \to e'_j e'_{j+1}v$ of $V(\psi'_j)$ is a non-zero homothety.
\end{lemma}

In \cite{MW87}, that Lemma is given for unspecified large j by their Lemmas I.13 and I.15. They are rather more involved than Lemme 4 in \cite{Rodier}, which however applies only to non-degenerate Whittaker models and $D=F$.
The proof of Lemma \ref{j'} will be given later.

\begin{proposition}\label{a'}  If $\lambda' \, \hat{}$  is minimal in  $\Supp c_\pi$ and $j\geq j'_0$, then  $p_j$ is an isomorphism, so that 
$\dim_{\mathbb C}(V_\theta)=\dim_{\mathbb C}V(\psi'_j) $.
\end{proposition}

 \begin{proof}
 We already know by Lemma \ref{l3} that $p_j$ is surjective  for $j$ large.
We also know by Remark \ref{+} that $\dim_{\mathbb C}V(\psi'_j) =m(\psi'_j, \pi)$ is constant for$j\geq j_0$.
The main point is  Lemma \ref{j'} which  implies that for $j\geq j'_0$, the linear
map $q_j : V(\psi'_j) \to V(\psi'_{j+1}) ,  v \to v_1=e'_{j+1}v  $  is injective, hence is 
an isomorphism because the two spaces have the same dimension. Moreover a vector $v\in 
 V(\psi'_j)$ is already invariant under $K'_{j+1}\cap B$ so what was said before Lemme 7.7 we have
$e'_{j+1}v=f_{j+1}v$, and $ v_1=e'_{j+1}v$ has the same image in $V_\lambda$ as $v$. Iterating
the process we get for positive integers $k$,  vectors $v_k=e'_{j+k}v_{k-1}=f_{j+k}v_{k-1}$. 
By definition of the projector $f_j$, we have $f_{j+k}f_{j+k-1}=f_{j+k}$ and consequently 
$v_k=f_{j+k}v$. But $p(v)=0$ if and only if $f_{j+k}v=0$ for large $k$ (Bernstein-Zelevinsky xyz). As $ v_k=0$ implies $v_{k-1}=0$ by the injectivity already established, we get $\Ker(p_j)=0$. But for large $j$, $p_j$ is surjective  so  is an isomorphism, and 
$\dim_{\mathbb C}( V(\psi'_j)=\dim_{\mathbb C}(V_\theta)$. But for $j\geq j'_0$, the dimension of $V(\psi'_j)$ is constant so $p_j$  is an isomorphism  and the Proposition follows.
\end{proof}

 Proposition \ref{a'} implies Part 2 of Theorem \ref{Wh} and that a partition of $n$ which is  minimal in  $\Supp c_\pi$ belongs to the Whittaker support of $\pi$.
Conversely, let $\mu \in \mathfrak P(n)$  minimal in the Whittaker support of $\pi$. Then by Lemma \ref{mum}, there is a partition $\mu'  $ in  $\Supp c_\pi$ with $\mu'\leq \mu$, and we may assume that $\mu'  $ is  minimal in $\Supp c_\pi$. But by Proposition \ref{a'}, that implies that $\mu'$  belongs to the Whittaker support of $\pi$, so  $\mu'=\mu$. Assuming Lemma \ref{j'}, Theorem \ref{Wh} is proved.
\end{proof}

 It remains to prove Lemma \ref{j'}.
We can conjugate by $t^{1-2j}$ to transform $\psi'_j $ back to $\psi_j$, and even further conjugate (Lemma \ref{B}) by a permutation matrix $\sigma$ to transform $\psi_j $ into the character $\xi_j$ attached to the matrix $B$.
We need to prove that the endomorphism of $eV$ sending $v$ to $efv$ is a non-zero homothety, where $e$ is the $K_j$-projector onto the one dimensional space $eV=V(\xi_j)$ and $f$ is integration on the group $J=\sigma( t^2) (K_j\cap U)(\sigma( t^2)^{-1}$ against its character  $(1+x)\mapsto  \psi \circ  \trd (-B.(p_D)^{-1-2j}x)$.
Clearly $efe$ is an element of  $eHe$ where $H$ is the full Hecke algebra of $G$, so we may restrict the mentioned integration to elements in the support of the Hecke algebra $eHe$. Also if $j \geq 2n-2$, the group $J$ is contained in $K_{j-2n+2}$ so it normalizes $K_j$, and the support of $efe$ is contained in the normalizer of  $\xi_j$ in $K_{j-2n+2}$.

By  Lemma \ref{ouf}, the normalizer of $\xi_\lambda$ in $K_0=GL_n(O_D)$ is $P_{\lambda'}(O_D)K_j$. Take $j-2n+2\geq j_0$ and $g=1+x$ be in the support of  $efe$. The trace of $ege$ in $eV$ can be computed using the germ expansion of $\pi$  as the sum over  $\mu \in \mathfrak P(n)$ of $c_\pi(\mu)$  times the trace of $efe$ in  $\pi_{P_\mu}$. By our choice $\lambda' \, \hat{}$ is minimal in $\Supp c_\pi$, so the only contribution is $c_\pi(\lambda' \, \hat{})$. Applying that to any  $ege$ in the support of $efe$ gives Lemma \ref{j'}, and  even that the homothety is via a positive integer.

 \section{Jacquet-Langlands correspondence}\label{s:JL}   
 The  Jacquet-Langlands correspondence  extended  by Badulescu (\cite{Bad07} Th\'eor\`eme 3.1), is 
 a surjective morphism  $LJ$ with a section $JL$
  $$ LJ :\Gr_{\mathbb C}^{\infty}(GL_{dn}(F))\to \Gr_{\mathbb C}^{\infty}(G ), \ \  JL :\Gr_{\mathbb C}^{\infty}(G)\to \Gr_{\mathbb C}^{\infty}(GL_{dn}(F))$$
which is an injective morphism of $\mathbb Z$-modules
    extending the classical Jacquet-Langlands correspondence between essentially square integrable representations.     
     \begin{theorem} \label{tJL}For $\nu \in \Gr_{\mathbb C}^{\infty}(GL_{dn}(F))$  and $ \lambda \in \mathfrak P (n)$, 
  we have
   $(-1)^nc_{LJ (\nu)}( \lambda)=  (-1)^{dn}c_{\nu}(d\lambda)$.
   \end{theorem}

   \begin{corollary} \label{cJL}For $\nu \in \Gr_{\mathbb C}^{\infty}(G )$  and $ \lambda \in \mathfrak P (n)$, 
  we have
   $(-1)^nc_{\nu}( \lambda)=  (-1)^{dn}c_{JL (\nu)}(d\lambda)$.
   \end{corollary} 
   
   The remainer of this section gives the proof of the theorem.

   \subsection{}Badulescu-Jacquet-Langlands  correspondence.
  \subsubsection{Preliminaries}    Let $\Irr^{2}_{\mathbb C}(G)$ denote  the set of isomorphism classes of essentially square integrable irreducible smooth complex representations of $G$.
    Any irreducible smooth complex representation of $D^*$ is essentially square integrable.

 As in \S 1,  $P_\lambda=M_\lambda N_\lambda $ is a parabolic  subgroup of $G$ for $\lambda\in \mathfrak P(n)$. For $\mu\in \mathfrak P(dn)$, we denote now by  $P_{\mu}= P_{\mu}N_\mu$.  

A basis of the   Grothendieck group $\Gr_{\mathbb C}^\infty(G)$ is  $$\mathfrak B_G=\{[n.\ind_{P_\lambda}^G \sigma] \ | \ \sigma \in \Irr^{2}_{\mathbb C}(M_\lambda), \lambda\in \mathfrak P(n)\}$$ where 
 $n.\ind_{P_\lambda}^G $  the normalized parabolic induction (\cite{Bad07} Proposition 2.2).
  As $\Irr^{2}_{\mathbb C}(G)$  is stable by the twist by a smooth character of  $G$,      $$    \mathfrak B'_G =\{[\ind_{P_\lambda}^G \sigma] \ | \ \sigma \in \Irr^{2}_{\mathbb C}(M_\lambda), \ \lambda\in \mathfrak P(n)\}.$$
 is also a basis of $\Gr_{\mathbb C}^\infty(G)$).  Let  $C_d$ be  the submodule of $\Gr_{\mathbb C}^\infty(GL_{dn}(F))$ of basis the set 
  $$\mathfrak B_d'=\{ [\ind_{P_\mu}^{GL_{dn}(F))} \sigma] \ | \ \sigma \in \Irr^{2}_{\mathbb C}(M_\mu), \mu\in \mathfrak P(dn) \ \text{but} \  \mu\not\in d \mathfrak P(n)\}.$$
 The Aubert involution $\iota$ of $\Gr_{\mathbb C}^\infty(G)$ sends an irreducible representation $\pi$  to an irreducible representation  modulo a sign \cite{Au95}:
  \begin{equation}\label{Au} \iota (\pi) = (-1)^{n-r}   |\iota (\pi) | 
  \end{equation} where 
  $ |\iota (\pi) |$   is irreducible and $r$ is the number of elements of the cuspidal support of $\pi$,  meaning that   $\pi \subset \ind_{P_\lambda}^G \sigma $ for $\lambda=(\lambda_1,\ldots, \lambda_r)\in \mathfrak P(n)$ and $\sigma\in \Irr_{\mathbb C}^2(M_\lambda)$ cuspidal (\cite{Bad07} (3.4), \cite{T90} \S 1).    
  
Let $\lambda$ be a partition of $n$ and  $\delta_\lambda$  the modulus of  the parabolic subgroup $P_\lambda=M_\lambda N_\lambda $ of $G$, 
$\delta_\lambda(g)=  |(\det \Ad (g)|_{\Lie N_\lambda})|_F$ for $g\in P_\lambda$. 
For a partition $\mu$ of $dn$, let $\delta'_{\mu}$ denote the modulus of   the parabolic subgroup $P'_{\mu}=M'_{\mu} N'_{\mu}$ of $GL_{dn}(F)$.  

\begin{lemma}\label{le:mod} Let $L/F$ be an extension splitting $D$. We have $\delta_\lambda = \delta'_{d\lambda}  $  on $P_\lambda (L)= P'_{d\lambda}(L)$. 
 \end{lemma}
 \begin{proof}
 We have $G(L)= GL_{dn}(F)$ and $P_\lambda (L)=P'_{d\lambda}(L)$. The modulus 
$\delta_\lambda$  is an algebraic character, and can also be computed in $P_\lambda (L)$. Similarly for $ \delta'_{d\lambda} $. The reduced norm on $G$ becomes the determinant on $G(L)$.
   \end{proof}
   Let $L/F$ be an extension splitting $D$. The reduced characteristic polynomial $P_a$  of $a\in M_n(D)$ is the characteristic polynomial of $a\otimes 1 \in M_n(D)\otimes_F L \simeq M_{nd}(L)$, which belongs to $F[X]$,  does not depend on the choice of $L$, and  $P_a(a)=0$  \cite[\S 17 page 333 D\'efinition 1, page 336 Corollaire 2, (34)]{Bki-A8}, \cite[\S 2 Propositions 2.1 and 2.2]{Bad18}.  
   \begin{lemma}\label{le:cha} The reduced characteristic polynomial of a  matrix in $M_n(D)$ belongs to $O_F[X]$ if and only if the matrix is $GL_n(D)$-conjugate to an element of $M_n(O_D)$.
\end{lemma}
\begin{proof}We have $M_n(D) \simeq \End_D D^n$ where $D^n$ is seen as a right $D$-module. Let  $e_1,\ldots, e_{n}$ be a basis of  $D^n$ over $D$. When $P_a\in O_F[X]$, the $O_D$-module generated by the $a^ie_1, \ldots, a^ie_n$ for the positive integers $i$, is finitely generated because $P_a(a)=0$, hence $a$ stabilizing an $O_D$-lattice of $D^n$ is $GL_n(D)$-conjugate to an element of $M_n(O_D)$. Conversely, if $a\in M_n(O_D)$ then $a\otimes 1 \in M_{nd}(O_L)$ hence its characteristic polynomial  $P_a $ belongs to $O_F[X]$;  for $g\in GL_n(D)$ we have $P_{gag^{-1}}=P_a \in O_F[X]$.\end{proof}

 We identify the space $S$ of  unitary polynomials in $F[T]$ of degree $dn$ with $F^{dn}$ by taking the non-dominant coeffficients. The map sending $X\in M_n(D) $  to its reduced characteristic polynomial $P_X$ which belongs to $S$,   is continuous (\cite{Bki-A8} \S 17 D\'efinition 1, \cite{R75} \S9a).

 We recall  from \cite[Chapter 2, \S 2 to \S 6]{Bad18}:
  
  An element $g\in G$ is called regular semi-simple when  the roots of  $P_g $   in an algebraic closure $F^{ac}$ of $F$ have multiplicity $1$. 
  The set $G^{rs}$ of regular semi-simple elements of $G$ is open dense in $G$.  
   The conjugacy class of $g\in G^{rs}$ is the set of elements  $g'\in G$ with  $P_{g'}=P_g$.  Note that $g=1+p_F^j X\in G^{rs}$ is conjugate to an element of $1+p_F^j M_n(O_D)$ if and only if $X$ is conjugate to an element of $M_n(O_D)$ if and only the coefficients of $P_X(T)=p_F^{-jdn} P_g(Tp_F^j +1)$ belong to $O_F$. The set  $\{P_g \ | \ g\in G^{rs}\}$  consists of the monic polynomials in $ F[T] $  of degree $dn$ without multiple roots in $F^{ac}$, with a non-zero constant term and with all irreducible factors of degree divisible by $d$. Let $GL_{dn}(F)^{rs,d}$ be the set of $h\in GL_{dn}(F)^{rs}$ such that $P_h\in \{P_g \ | \ g\in G^{rs}\}$.
    We say that $g\in G^{rs}$ and $h\in GL_{dn}(F)^{rs,d}$ correspond and we write $g\leftrightarrow h$ when $P_g=P_h$.

  Let $g\in G^{rs}$.  The $G$-centralizer $T_g$ of $g$ is a maximal torus,  isomorphic to the group of units of $F[T]/(P_g)$. We put  on $G/T_g$ the quotient measure $dx^*$ of the  Haar measure on $G$ (\S 1) and on the Haar measure on $T_g$  giving the value $1$ to the maximal torus. The orbital integral of $f\in C_c^\infty(G;\mathbb C)$ at $g$  is 
 \begin{equation}
 \Phi(f,g)=\int_{G/T_g} f(x g x^{-1}) \, dx^*.
  \end{equation}
Let  $C_c^\infty(GL_{dn}(F)^{rs};\mathbb C)^{(d)}$ be the set of $\varphi \in C_c^\infty(GL_{nd}(F)^{rs};\mathbb C) $ with $\Phi (\varphi,h)=0$ when $h $ is not in $GL_{nd}(F)^{rs,d}$.   We say that  $f\in C_c^\infty(G^{rs};\mathbb C) $ and $\varphi \in C_c^\infty(GL_{nd}(F)^{rs};\mathbb C)^{(d)}$ correspond and we write $f\leftrightarrow \varphi$ when 
$\Phi(f,g)=\Phi(\varphi ,h)$ if $g\in G^{rs}$ and $h\in GL_{nd}(F)^{rs,d}$ correspond.  
For $f\in C_c^\infty(G^{rs};\mathbb C) $ there exists $\varphi \in C_c^\infty(GL_{nd}(F)^{rs};\mathbb C)^{(d)}$ 
such that 
 $f\leftrightarrow \varphi$, and conversely  (\cite{Bad18} Proposition 5.1). 
     
   \subsubsection{Jacquet-Langlands correspondence} The classical Jacquet-Langlands correspondence (\cite{DKV84}, \cite{Bad02}) 
 is the  unique bijective map  
\begin{align*}
&JL: \Irr^{2}_{\mathbb C}(G))\to \Irr^{2}_{\mathbb C}(GL_{dn}(F)) \ \ \ \text{ such that for } \  \pi \in  \Irr^{2}_{\mathbb C}(G),\\
& (-1)^n\, \trace (\pi(f(g) dg) )= (-1)^{dn}\, trace ( JL (\pi) (\varphi(h) dh) )
 \end{align*}
 when $ f\in C_c^\infty(G);\mathbb C)^{rs},\  \varphi \in C_c^\infty(GL_{dn}(F);\mathbb C)^{rs, d}, \ f\leftrightarrow \varphi$. The image by $JL$ of  the Steinberg representation of $G$ is  the Steinberg representation of $GL_{dn}(F)$.
The maps $JL$ extends  to 
 
 1) a bijective map 
 $$JL:  \Irr^{2}_{\mathbb C}(M_\lambda))\to \Irr^{2}_{\mathbb C}( M'_{d\lambda} )  \ \text{for any composition  $\lambda$ of $n$.}$$
 
 2)   an injective map  
$$JL: \mathfrak B_G\to \mathfrak B_{GL_{dn}(F)}
$$
 \begin{equation}\label{JLind} JL([n.\ind_{P_\lambda} ^G \sigma] = [n.\ind_{P_{d\lambda}}^{GL_{dn}(F)} JL(\sigma)]  \  \text{for} \  \sigma \in  \Irr^{2}_{\mathbb C}(M_\lambda), \lambda\in \mathfrak P(n),
 \end{equation}
and by linearity  to  an  injective homomorphism  
    $$ JL :\Gr_{\mathbb C}^{\infty}(G )\to \Gr_{\mathbb C}^{\infty}(GL_{dn}(F)),
  $$
 satisfying  (\cite{Bad07} Th\'eor\`eme 3.1):
   \begin{equation}\label{ct} (-1)^n\, \trace \nu (f(g) dg) = (-1)^{dn}\, \trace  JL (\nu ) (\varphi(h) dh)  
  \end{equation}
  when  $\nu\in \Gr_{\mathbb C}^{\infty}(G )$, $ f\in C_c^\infty(GL_n(D)^{rs};\mathbb C),\  \varphi \in C_c^\infty(GL_{dn}(F)^{rs};\mathbb C)^{(d)}, \ f\leftrightarrow \varphi$.   We have 
$$\Gr_{\mathbb C}^{\infty}(GL_{dn}(F)) =  JL(\Gr_{\mathbb C}^{\infty}(G )) \oplus C_d.$$
The homomorphism  $JL$ commutes with 

a) the twist by  smooth characters: 

$JL((\chi \circ \nrd )\otimes \nu)=(\chi \circ \det)\otimes JL(\nu)$ when $\chi$ is a smooth character of $F^*$,

 b) the normalized parabolic induction (\cite{Bad07} Th\'eor\`eme 3.6): 

 $JL (\ind_{P_\lambda}^G( \delta_\lambda^{1/2} \nu)= \ind_{P_{d\lambda}}^{GL_{dn}(F)}({\delta'_{d\lambda}}^{1/2} JL(\nu))$).

 3) a surjective  homorphism extending the inverse $LJ$ of the classical Jacquet-Langlands correspondence $JL$ for the Levi subgroups :   
 $$LJ: \mathfrak B_{GL_{dn}(F)}\to \mathfrak B_G
$$
\begin{equation}\label{JLind} LJ([n.\ind_{P_{\mu }} ^{^{GL_{dn}(F)} } \sigma] =
\begin{cases} [n.\ind_{P_{\lambda}}^G LJ(\sigma)] & \  \text{for} \  \sigma \in  \Irr^{2}_{\mathbb C}(M_\mu), \  \mu = d\lambda\in d\mathfrak P(n),\\
0  & \  \text{for} \  \sigma \in  \Irr^{2}_{\mathbb C}(M_\mu), \ \mu \in\mathfrak P(dn) \ \text{but} \ \mu \not\in d\mathfrak P(n)
\end{cases}
 \end{equation}
 giving by linearity  a surjective homomorphism (the Badulescu-Jacquet-Langlands correspondence):
    $$ LJ :\Gr_{\mathbb C}^{\infty}(GL_{dn}(F))\to \Gr_{\mathbb C}^{\infty}(G)
  $$
of  kernel $C_d$,  section $JL$, satisfying
\begin{equation}\label{ct} (-1)^{dn}\, \trace \nu (f(g) dg) = (-1)^{n}\, \trace  LJ(\nu ) (\varphi(h) dh) 
  \end{equation}
  when  $\nu\in \Gr_{\mathbb C}^{\infty}(GL_{dn}(F) )$, $ f\in C_c^\infty(GL_{dn}(F)^{rs};\mathbb C),\  \varphi \in C_c^\infty(G ^{rs};\mathbb C)^{(d)}, \ f\leftrightarrow \varphi$. 
The homorphism $LJ$ commutes  with   the twist by smooth characters: if  $\chi$ is a smooth character of $F^*$ and $\nu\in \Gr_{\mathbb C}^{\infty}(GL_{dn}(F) )$,
\begin{equation}\label{eqX}  LJ((\chi \circ \det)\otimes \nu)=(\chi \circ \nrd)\otimes LJ(\nu), \ \ 
\end{equation}
 the normalized parabolic induction:  if  $ \delta'_\mu$ the modulus of  $P'_\mu$ and $\nu \in  \Gr_{\mathbb C}^{\infty}(M'_\mu)$,  still denoting $JL :\Gr_{\mathbb C}^{\infty}(M'_\mu)\to \Gr_{\mathbb C}^{\infty}(M_\lambda)$  
the natural morphism,
 we have 
 \begin{equation}\label{LJpi}LJ ( \ind_{P'_\mu}^{GL_{dn}(F)}( {\delta'_\mu}^{1/2} \nu )  ) =  \begin{cases} 0 \ & \text{if} \ \mu\not\in d \mathfrak P(n) \\ 
  \ind_{P_\lambda}^{G }( \delta_\lambda^{1/2} LJ (\nu   ) )   \ & \text{if} \ \mu= d\lambda, \ \lambda  \in\mathfrak P(n)
  \end{cases}
 \end{equation}
and is compatible with the Aubert involution $\iota$ up to  a sign   (\cite{Bad07} Proposition 3.16):
\begin{equation}\label{LJio} (-1)
^{n}\iota \circ  LJ =LJ \circ  (-1)^{dn}\iota.
\end{equation} 
 As $LJ$ sends the Steinberg representation of $GL_{dn}(F)$ to the Steinberg representation of $G$,  the Aubert involution of the Steinberg representation is the trivial representation up to a sign, and $LJ$ commutes with the parabolic induction, we have:
\begin{equation}\  \label{LJ1}
 (-1)^{nd} \, LJ( \pi_{P'_\mu}  ) =\begin{cases}  (-1)^n \, \pi_{P_\lambda} &\ \text{if} \ \mu=d\lambda,\\
  0  \ & \ \text{otherwise}.
 \end{cases}
\end{equation}

 \subsection{}The theorem \ref{tJL} is an easy consequence of \eqref{ct},  \eqref{LJ1},  and of the linear independance of the restrictions  to $K\cap GL_n(D)^{rs}$  of the characters of the representations $\pi_{P_\mu}$ of $GL_{dn}(F)$ for $\mu\in \mathfrak P(dn)$, for any open compact subgroup $K$ of $GL_{dn}(F)$.
We give the details. 
 
Let $P=MN$ be a parabolic subgroup of $G$ of Levi $M$,  $\sigma \in \Irr^{2}_{\mathbb C}(M)$, $\pi=\ind_P^G\sigma$. Let   $c_\pi, c_{JL(\pi)}$  be  the maps and $K_\pi, K_{JL (\pi)}$   groups in  the germ expansions  \eqref{tm} of $[\pi], JL ([\pi])$, such that  for any $g\in K_\pi \cap G^{rs}$ there exists $h\in  K_{JL (\pi)}\cap  GL_{dn}(F)^{rs,d}$ with $g\leftrightarrow h$, as we can because for $g\in GL_n(D)^{rs}$, $h \in GL_{dn}(F)^{rs,d}$ with the same reduced characteristic polynomial $  P(T)$,  the coefficients of  $p_F^{-jdn} P(Tp_F^j +1)$  belong to $O_F$ if and only if $g$ is  conjugate to an element of $1+p_F^j M_n(O_D)$  if and only if
$h$ is conjugate to element of  $1+p_F^j M_{dn}(O_F)$ 
 (Lemma \ref{le:cha}).

Let  $f\in C_c^\infty(K_{LJ(\pi)}\cap G^{rs};\mathbb C) ,  \varphi \in C_c^\infty(K_{\pi}  \cap GL_{dn}(F)^{rs};\mathbb C)^{(d)}, \ f\leftrightarrow \varphi$.  The germ expansion \eqref{tm} applied to \eqref{ct}
$(-1)^n\, \trace LJ(\pi)(f(g)dg)= (-1)^{dn}\,\trace \pi(\varphi(g) dg)$  gives
   \begin{equation*}\label{clcee} (-1)^n\, 
 \sum_{\lambda\in \mathfrak P(n)}c_{LJ(\pi )}(\lambda)\, \trace \pi_{P_\lambda} (f(g)dg) =
  (-1)^{dn}\,  \sum_{\mu\in \mathfrak P(dn)} c_ \pi(\mu)\, \trace \pi_{P'_{\mu}} (\varphi(g) dg),
\end{equation*}
and applying  \eqref{ct},   then   \eqref{LJ1}  to  the RHS,  
   $$= (-1)^n\, 
\sum_{\mu\in \mathfrak P(dn)} c_{\pi}(\mu)\, \trace  LJ(\pi_{P'_\mu}) (f(g) dg)  =(-1)^{dn}\, 
\sum_{\lambda\in \mathfrak P(n)} c_{\pi }(d\lambda)\, \trace   \pi_{P_\lambda} (f(g) dg) .$$So, 
 $ (-1)^n\, 
 \sum_{\lambda\in \mathfrak P(n)}c_{LJ(\pi )}(\lambda)\, \trace \pi_{P_\lambda} (f(g)dg) =(-1)^{dn}\, 
\sum_{\lambda\in \mathfrak P(n)} c_{\pi }(d\lambda)\, \trace   \pi_{P_\lambda} (f(g) dg) .$ 
 
 The linear independence of the characters of $\pi_{P_{\lambda}} $ on  $K_{ LJ(\pi)}  $ for $ \lambda \in \mathfrak P(n)$    (Corollary \ref{LIpi}) and the local integrability of characters imply the \footnote{ Put $K= K_{LJ(\pi)}$. Any $f \in C_c^\infty(G;\mathbb C)$ with support in $K$  is a limit of (uniformly
bounded) functions $f_n$ with  support in $K\cap G ^{rs}$, so by the local integrability of characters and the Lebesgue dominated convergence theorem, 
$\trace \pi_{P_\lambda}(f(g) dg)= \lim_n \trace \pi_{P_\lambda }(f_n(g) dg).$}
  linear independence of the characters of $\pi_{P_{\lambda}} $ on  $K_{ LJ(\pi)}  \cap G ^{rs}$ for $ \lambda \in \mathfrak P(n)$  and 
  $$(-1)^{dn} c_{\pi}(\lambda)= (-1)^{n} c_{LJ(\pi)}( d\lambda)  \ \ \text{for } \  \lambda \in \mathfrak P (n).
  $$ 
 for any  $[\pi]$ in the basis  $\mathfrak B_G$ of $\Gr_{\mathbb C}^\infty(G)$.
   This ends the proof of the theorem  \ref{tJL}.
  
   \subsection{}{\bf Applications to  $c_\pi( (n))$}   For  $\pi\in  \Irr^{2}_{\mathbb C}( G)$ and a  division central $F$-algebra $D_{dn}$ of reduced degree $dn$, there exists a unique $\pi_{dn} \in \Irr_{\mathbb C}(D_{dn}^*)$ such that their images by the  classical Jacquet-Langlands correspondence in $ \Irr_{\mathbb C}^2(GL_{dn}(F))$ are equal.  The dimension of 
 $\pi_{dn}$ is finite and by Theorem \ref{tJL}) $(-1)^nc_\pi (n) = -c_{JL(\pi_{dn})}(dn)=-\dim_{\mathbb C} \pi_{dn}$.  
An irreducible smooth complex representation $\pi$ of $G$ is tempered if and only if  $\pi=\ind_P^G \sigma$   for  a parabolic subgroups $P=MN$ of $G$ and $\sigma\in \Irr^{2}_{\mathbb C}(M)$ (\cite{LMT16} A.11).
  
  For $\nu \in \Gr_{\mathbb C}^{\infty}(GL_{dn}(F))$  and $ \lambda \in \mathfrak P (n)$, 
  we have
   $(-1)^nc_{LJ (\nu)}( \lambda)=  (-1)^{dn}c_{\nu}(d\lambda)$.     
 
 \begin{corollary} \label{cJL} Let  $\pi \in \Rep_{\mathbb C}^\infty(G)$ irreducible and tempered. Then
 $$ c_{\pi}((n))=\begin{cases} (-1)^{n-1}\, \dim _{\mathbb C}  \pi_{dn} & \text{if} \  \pi \in  \Irr^{2}_{\mathbb C}( G)\\ 
 0  & \text{if}  \ \pi \not\in  \Irr^{2}_{\mathbb C}( G)
 \end{cases}.
 $$
    \end{corollary}

 \section{Coefficient  field of characteristic different from $ p$}\label{s:10}
Let $R$ be a  field.
  Our goal is to  show that  Theorem \ref{1.1} proved using the Harish-Chandra germ expansion remain valid for $R$-representations when the characteristic of $R$ is not $p$.  There are   two simple reasons: 
 
a) For a parabolic subgroup $P$ of $G$, 
the representation $\ind_P^G1$ is defined over $\mathbb Z$.

b) For  a field extension $R'/R$,  the scalar extension from $R$ to $R'$ 
 of smooth  representations  of a profinite group $H$ respects finite length,  and is an injection at the level of Grothendieck groups \cite{HV19}. For an irreducible smooth $R$-representation $\pi$ of $H$,  the $R'$-representation $R'\otimes _R \pi$   considered as an $R$-representation  is $\pi$-isotypic (a direct sum of representations isomorphic to $\pi$).

\bigskip From now on,   $\charf _R\neq p$.  
 When $\pi \in \Rep_R^{\infty,f}$  is equal to $\sum_{\lambda\in \mathfrak P(n)} c_\pi(\lambda) \ind_{P_\lambda}^G 1$ on $K_\pi$ as in Theorem \ref{1.1}, the map $c_\pi$ is unique because:
\begin{proposition}[Corollary \ref{LIpi}]  \label{LIpiR} Let   $K$ be an open pro-$p$ subgroup of $G$. For any non zero map $c:\mathfrak P(n)\to \mathbb Z$,   the restriction to   $K$  of 
$$\sum_{\lambda\in \mathfrak P(n)}c(\lambda) \, [\pi_{P_\lambda}]\in \Gr_{R}^{\infty}(G)$$ 
  is not $0$.
\end{proposition}
 
\begin{proof}  We can suppose $R$ algebraically closed by b).  The categories $\Rep_R^\infty(K)$ and $\Rep_{\mathbb C}^\infty(K)$ are equivalent  and the Grothendieck groups $\Gr_R^\infty(K)$ and $\Gr_{\mathbb C}^\infty(K)$ are isomorphic because $K$ is a pro-$p$ group and $\charf _R\neq p$. The proposition is true when $R=\mathbb C$ (Corollary \ref{LIpi})  and the representations $\pi_{P_\lambda}$ correspond. Hence the proposition is true for any $R$. \end{proof}

 We list  other  properties   which will  be  used   in the proof of the theorem \ref{1.1}. 
\subsection{} \label{sce} {\bf Twist by a character, image by an automorphism}  

Assume that $\pi\in \Rep_R ^{\infty, f}(G)$ has a germ expansion of map $c_\pi$ on $K_\pi$ (the first part of Theorem \ref{1.1}),  $\chi$ is a smooth $R$-character of $G$ and  $\sigma$  is an automorphism of $R$. Then the representations $\pi \otimes \chi$ and $\sigma(\pi)$ have  a germ expansion of maps 
 $c_{\pi \otimes \chi}=c_{\sigma(\pi)}=c_\pi$ on 
  $K_{\pi \otimes \chi}=K_{\sigma(\pi)}=K_\pi$ if $\chi $ is trivial on $K_\pi$. 
 The reason is a)  ($\sum_{\lambda\in \mathfrak P(n)} c_{\pi}(\lambda) \,[ \pi_{P_\lambda}]$  is defined over $\mathbb Z$).

\subsection{} {\bf  Germ expansion  on the Grothendieck group} Assume that any $\pi\in \Rep_R^{\infty, f}(G)$ has  a germ expansion of map $c_\pi$ on some open compact subgroup $K_{\pi}$ of $G$. Then, 
 the  linear  map $\nu \mapsto c_\nu: \Gr_R ^{\infty}(G)\to  \{ \mathfrak P(n)\to \mathbb Z\}$ such that $c_{[\pi]}=c_\pi$ for  $\pi\in \Rep_R^{\infty, f}(G)$, has the property that  the restrictions  to some open  compact subgroup $K_{\nu}$ of $G$ of    $\nu$ and of $\sum_{\lambda\in \mathfrak P(n)} c_{\nu}(\lambda) \, [\pi_{P_\lambda}]$  are  isomorphic.  

Parabolic induction: For a parabolic subgroup $P$ of $G$ of Levi $M$, the   parabolic induction 
$\ind_P^G$ is   exact and respects finite length
and passes to a linear map between the Grothendieck groups: 
$$\ind_P^G: \Gr_R^\infty(M)\to  \Gr_R^\infty(G), \ \ \  \ind_P^G [\sigma] = [\ind_P^G \sigma] \ \text{for } \sigma\in  \Rep_R ^{\infty, f}(M).$$
When $\nu\in  \Gr_R ^{\infty}(M)$ has  a germ expansion of map $c_\nu$, then $\ind_P^G\nu$ has a germ expansion of map induced by $c_\nu$ (Theorem \ref{th:71}).

\subsection{}\label{pos}  For $j\in \mathbb N_{>0}$ and $\lambda $ is a composition of $n$, the  values of the character $\xi_\lambda$ of $K_j=1+M_n(P_D^j)$  defined by \eqref{xi}  and of the character $\theta_\lambda$ of $U$ defined by  \eqref{thetal} are  roots of $1$ of order 
 powers of $p$. Assume that  the  field  $R$  contains roots of unity
of any $p$-power order, 
$$\text{we write} \  
 \mu_{p^\infty} \subset R,\ \text{
  implying} \  \charf _R\neq p.$$
  We can define   $\xi_\lambda$ and $\theta_\lambda$  over $R$ as before, and the Whittaker support  of   an irreducible smooth  $R$-representation of $G$ as before Theorem \ref{Wh}.

Let  $\pi\in \Rep_{R}^{\infty,f}(G)$ having  a germ expansion of map $c_\pi:  \mathfrak P(n) \mapsto \mathbb Z$: for a positive integer $j_0$ the restriction of $\pi$  and 
  of  $\sum_{\lambda\in \mathfrak P(n)} c_\pi(\lambda) \ind_{P_\lambda}^G1$  to $K_{j_0}$ are equal. With the same proofs as for  $R=\mathbb C$, we have:

 \begin{theorem}  \label{WhR}  1) For any integer  $j\geq j_0$ and any 
 $\lambda$  partition of $n$, we have
   \begin{align}\label{cm} 
   c_\pi (\lambda)&=m(\xi_{\lambda }, \pi) - \sum_{\mu \in \mathfrak P(n), \mu <\lambda } c_\pi(\mu) \, 
m(\xi_{\lambda }, \pi_{P_\mu}).\end{align}
  In particular when  $\lambda $ is minimal in the support of $c_\pi$, $c_\pi (\lambda)=m(\xi_{\lambda }, \pi ) $  is positive and  independent of $j\geq j_0$.
  
 2) Theorem  \ref{Wh} is valid.
 \end{theorem}

 An  algebraically closed field $R$ with $\charf_R\neq p$ contains $\mu_p^\infty$. To prove Theorem \ref{1.1} for $R$  algebraically closed, by   Theorem \ref{WhR} and Proposition \ref{LIpiR}, 
we have only to prove that any $\pi\in \Rep_R^{\infty, f}(G)$ has a germ expansion: there exists a map $c_\pi: \mathfrak P(n)\to \mathbb Z$ such that $\pi$ and $\sum_\lambda c_\pi(\lambda) \ind_{P_\lambda}^G 1$ have equal on some open compact subgroup $K_\pi$ of $G$.

\bigskip    
 We prove now Theorem  \ref{1.1}  going from  $R=\mathbb C$   to $R=\mathbb Q_\ell^{ac}$   to  $R=\mathbb F_\ell^{ac}$, $\ell\neq p$, 
to an algebraically closed field $R$, and finally to a  not necessarily algebraically closed field $R$.

\subsection{} {\bf  $R  \simeq R'$.}  For any prime number $\ell$, the fields $\mathbb C$ and $\mathbb Q_\ell^{ac}$ are isomorphic. 
It is easy to see that  if  Theorem \ref{1.1} is  true for a field $R$, it is also true for an isomorphic field $R'$. 
 Indeed, a field   isomorphism $j:R\to R'$  induces  isomorphisms of categories  
   $j_G: \Rep^{\infty}_{R}(G)\to  \Rep^{\infty}_{R'}(G)$ and 
    $j_K: \Rep^{\infty}_{R}(K)\to  \Rep^{\infty}_{R'}(K)$ for any open compact subgroup $K$ of $G$.    The isomorphisms  commute with   the restriction to $K$ and the parabolic induction $\ind_P^G$.     
  For $\pi\in \Rep^{\infty}_{R}(G)$ and $\sigma\in \Rep^{\infty}_{R}(M)$, 
    $$j_K (\pi|_K)=j_G (\pi)|_K, \ \ \ind_P^G(j_M (\sigma))=j_G (\ind_P^G\sigma).$$
When the theorems are true for $R$ they are also true for $R'$.  For $\pi\in \Rep^{\infty,f}_{R}(G)$, then $c_\pi= c_{j_G (\pi)}$ and we can take $K_{j_G(\pi)}=K_\pi$.

   \subsection{}\label{redell} {\bf $R\simeq \mathbb F_\ell^{ac}$ for    $\ell \neq p$. }
   
  The  theorems  over  $\mathbb Q_\ell^{ac}$ imply the theorem  over $\mathbb F_\ell^{ac}  $  by reduction modulo $\ell $ for  $\ell\neq p$. We denote by $\mathbb Z_\ell^{ac}$   the ring of integers of $\mathbb Q_\ell^{ac}$.  A  lattice in a $\mathbb Q_\ell^{ac}$-vector space $V$ is a  
  free $\mathbb Z_\ell^{ac}$-submodule  generated by a $\mathbb Q_\ell^{ac}$-basis of $V$.

 Let $\pi \in \Rep_{\mathbb Q_\ell^{ac}}^{\infty,f}(G)$.
 One says that $\pi$ is  integral when the space of $\pi$  contains  a $G$-stable   lattice $\mathfrak L_\pi$. Then,  the reduction modulo $\ell$   of  $\mathfrak L_\pi$ equal to $\mathbb F_\ell^{ac} \otimes_{ \mathbb Z_\ell^{ac}} \mathfrak L_\pi$  belongs  to $\Rep_{\mathbb F_\ell^{ac}}^{\infty,f}(G)$ and  its image in the Grothendieck 
group  $\Gr_{\mathbb F_\ell^{ac}}^{\infty } (G)$ does not depend on the choice of $\mathfrak L_\pi$. It is called the {\bf reduction modulo $\ell$} of $\pi$, and  denoted  by  $ r_\ell (\pi)$.     The subcategory of  integral representations $\Rep_{\mathbb Q_\ell^{ac}}^{\infty,f,int}(G)$ in $\Rep_{\mathbb Q_\ell^{ac}}^{\infty,f}(G) $ is abelian \cite{V96}; let  $\Gr_{\mathbb Q_\ell^{ac}}^{\infty,int} (G)$ be its Grothendieck group. The reduction modulo $\ell$  passes to a surjective (not injective)  map between   the Grothendieck groups:     $$r_\ell: \Gr_{\mathbb Q_\ell^{ac}}^{\infty,int} (G) \to   \Gr_{\mathbb F_\ell^{ac}}^{\infty} (G), $$
and 
 there  is an explicit subset  $E(G)$ of
 $ \Rep_{\mathbb Q_\ell^{ac}}^{\infty,f, int}(G)$  such that the set  $\{ r_\ell (\pi) \ | \ \pi \in E(G)\}$ is a basis of the Grothendieck 
group  $\Gr_{\mathbb F_\ell^{ac}}^{\infty } (G)$   (\cite{MS14} Th\'eor\`eme  9.35). 

 For a parabolic subgroup $P$ of $G$ with Levi $M$, 
the  parabolic induction  $\ind_P^G: \Rep_{\mathbb Q_\ell^{ac}}^{\infty}(M)\to \Rep_{\mathbb Q_\ell^{ac}}^{\infty}(G)$  is exact, respects finite length and integrality  hence passes to the Grothendieck groups  and $r_\ell \circ \ind_P^G = \ind_P^G \circ r_\ell$ on  $\Rep_{\mathbb Q_\ell^{ac}}^{\infty,f,int}(M)$.

The representation $\pi_P$ over $\mathbb Q_\ell^{ac}$ are integral, with a canonical  integral structure (the functions with values in $\mathbb Z_\ell^{ac}$: $\pi_P$ over  $\mathbb Z_\ell^{ac}$) of reduction modulo $\ell$ the representation $\pi_P$ over $\mathbb F_\ell^{ac}$.

If $\pi \in \Rep_{\mathbb Q_\ell^{ac}}^{\infty, f, int}(G)$ has a  germ expansion of map  $c_\pi$ on  $K_\pi$, then  $r_\ell(\pi) \in \Gr_{\mathbb F_\ell^{ac}}^{\infty}(G)$ has a  germ expansion of map  $c_\pi$ on  $K_\pi$.

\begin{lemma}\label{cred} Let  $\pi, \pi' \in  \Rep_{\mathbb Q_\ell^{ac}}^{\infty,f, int} (G ) $ with  $r_\ell (\pi)= r_\ell (\pi')$. Then  
$c_\pi=c_{\pi'}.$   
\end{lemma}
\begin{proof} When $j$ is large, we have  \eqref{cm} for $\pi$ and $\pi'$.
  As $K_j$ is a pro-$p$ group,  $m(\xi_\lambda, \pi ) =m(r_\ell(\xi_\lambda), r_\ell(\pi )).$
Therefore $r_\ell(\pi )=r_\ell(\pi ')$ implies   $m(\xi_\lambda, \pi )=m(\xi_\lambda, \pi ')$. By induction  on $\lambda$ we deduce $c_\pi =c_{\pi '}$.
\end{proof}

As the $r_\ell (\pi)$ for $\pi\in E(G)$ generate  $\Gr_{\mathbb F_\ell^{ac}}^{\infty} (G)$, Lemma \ref{cred} gives the existence of   a  linear map  
 $$c:\Gr_{\mathbb F_\ell^{ac}}^{\infty} (G) \to \{  \mathfrak P(n)\to \mathbb Z\} \ \text{
defined by } c_{r_\ell (\pi)}= c_\pi \ \text{ for  $\pi \in \Rep_{\mathbb Q_\ell^{ac}}^{\infty,f, int} (G )$} . $$
For  $\pi \in  \Rep_{\mathbb F_\ell^{ac}}^{\infty,f } (G)$, the  restrictions  of    $\pi $ and of $\sum_{\lambda\in \mathfrak P(n)} c_{\pi}(\lambda) \, r_\ell (\pi_{P_\lambda})$ to  some open  pro-$p$ group $K_\pi$  of $G$ are  isomorphic. 
Theorem \ref{1.1}  when $R=\mathbb F_\ell^{ac} $ is proved.

\subsection{} \label{RR'}{\bf    $R'/R $ algebraically closed fields}\label{3}  Given  an extension $R'/R$ of  algebraically closed  fields of characteristic different from $ p$,  we prove that   the germ expansion   over $R$ for all $n\geq 1$  is equivalent to the germ expansion over $R'$ for all $n\geq 1$.  
Therefore  we get Theorem \ref{1.1} over any   algebraically closed field $R$,  because  we already proved  for $R=\mathbb C$  and   $R=\mathbb F_\ell^{ac}$  when $\ell\neq p$.

The proof   relies on  properties, that we now recall, of  the scalar extension $\pi\mapsto R'\otimes_R \pi: \Rep_R^\infty(G)\to \Rep_{R'}^\infty(G)$ from $R$ to $R'$ and of the representations  of $G$ parabolically induced from Speh representations of the Levi subgroups of $G$.  Fix the same square root of $q=p^f$ in $R$ and in $R'$.

The scalar extension from $R$ to $R'$  respects  irreducible smooth representations and cuspidality, is exact and  passes to    an  injective linear map $\nu \mapsto R'\otimes_R \nu: \Gr_{R }^{\infty}(G)\to \Gr_{R'}^{\infty}(G)$ between the Grothendieck groups, commutes with the parabolic induction  and    for any open pro-$p$ subgroup $K$ of $G$ is an isomorphism of categories  $\delta \mapsto R'\otimes_R \delta: \Rep_{R }^{\infty}(K)\to \Rep_{R'}^{\infty}(K)$  \cite{HV19}. When $\pi\in  \Rep_R^{\infty, f}(G)$ the multiplicity $m(\delta,\pi)$  in $\pi$ of $\delta\in \Rep_{R }^{\infty}(K)$ irreducible   is equal to 
 $m(R'\otimes_R \delta,R'\otimes_R \pi) $. 
  Any irreducible cuspidal $R'$-representation $\rho'$ of $G$ is the twist by an unramified smooth $R'$-character $\chi$ of $G$ of an  irreducible cuspidal $R$-representation $\rho$ of $G$, $\rho' = \chi \otimes (R'\otimes_R \rho)=\chi \otimes_R\rho$ \cite{V96}. By  Lemma \ref{twistS} below, this is also true  for Speh representations.

  Let $m$ be a divisor  of $n=mr,r\geq 1,$ and $\rho$ an irreducible cuspidal $R$-representation  of $GL_m(D))$. To $(\rho,n) $ are attached in  \cite{MS14}:
 
 \begin{itemize}
 \item an unramified smooth $R$-character $\nu_\rho$ of $GL_m(D)$ depending only on the inertia class of $\rho$ (loc.cit. $\S 5.2$).
\item  a cuspidal $R$-segment $\Delta_{\rho,n}= (\rho, \nu_{\rho} \otimes \rho, \ldots ,\nu_\rho ^{-1+r}\otimes \rho)$ of length $r$, denoted $[0, -1+r]_\rho$ in  (loc.cit. $\S 7.2$). 
 \item  an irreducible  subrepresentation $Z(\Delta_{\rho,n})\in  \Rep_R^\infty(GL_n(D))$ (a Speh representation) of the normalized parabolic induction $\rho \times \ldots \times (\nu_\rho ^{-1+r}\otimes \rho)$ of $\rho\otimes   \ldots \otimes (\nu_\rho ^{-1+r}\otimes \rho) \in \Rep_R^\infty M_{\lambda}$ for $\lambda=(m,\ldots, m)\in \mathfrak P(n)$ (loc.cit. $\S 7.2$).
 \end{itemize}
 
\begin{lemma} \label{twistS}  For each unramified  smooth $R'$-character $\chi$ of $F^*$, 
 $$(\chi\circ \nrd) \otimes   _{R} Z(\Delta_{\rho,n}) \simeq Z(\Delta_{(\chi  \circ \nrd) \otimes _{R} \rho, n}) .$$  \end{lemma}
This important property   is stated in \cite{MS17}[(8.1.2)] (c.f.\cite[Lemme 5.9]{DS23}). 
  
\bigskip  To a composition  $ (n_1,\ldots, n_r)$ of $n$,  a divisor $m_i$ of $n_i$ and  an irreducible cuspidal $R$-representation $\rho_i$  of $GL_{m_i}(D))$  for $1\leq i \leq r$,  are associated 
 \begin{itemize}
 \item  a cuspidal $R$-multisegment
$\mathfrak M=(\Delta_{\rho_1, n_1},\ldots,\Delta_{\rho_r,n_r})$,
 \item  a Speh  $R$-representation $Z(\mathfrak M)= Z(\Delta_{\rho_1,n_1})\otimes \ldots Z(\Delta_{\rho_r,n_r})$ of $M= M_{(n_1, \ldots, n_r)}$,
 \item   the normalized parabolic induction $n.I(\mathfrak M) =\ind_P ^G ( Z (\mathfrak M) \delta_{P }^{1/2})$  of  $Z(\mathfrak M)$ where $P=P_{(n_1, \ldots, n_r)}$ and 
 $\delta_P $ is the module of $P$.
  \end{itemize}
 The Grothendieck group $\Gr_R^{\infty}(G)$ is generated by the  $[n.I(\mathfrak M)]$ for the cuspidal $R$-multisegments $\mathfrak M$ of $GL_n(D)$  (\cite{MS14} proof of Lemma 9.36 with Proposition 9.29).  
 
 But  $Z (\mathfrak M) \delta_{P }^{1/2}$ is also a Speh representation $Z (\mathfrak M')=Z(\Delta_{\rho'_1,n_1})\otimes \ldots Z(\Delta_{\rho'_r,n_r})$ where $\rho'_i$ is the twist of $\rho_i$ by an unramified character. Therefore $\Gr_R^{\infty}(G)$ is also generated by the images of the parabolic induction $I(\mathfrak M) = \ind_P ^G ( Z (\mathfrak M))$  for the cuspidal  $R$-multisegments $\mathfrak M$.
  If the Speh $R$-representations $Z(\mathfrak M)$  of $G$ have a germ expansion then the $I(\mathfrak M)$ have a germ expansion (Theorem \ref{th:71}) and any 
$\pi\in \Rep_R^{\infty,f}(G)$ has a germ expansion. 

\bigskip We are now ready to prove that the existence of a germ expansion over $R$  is equivalent to  the existence of a germ expansion over $R'$.   Let $\mathfrak M'=(\Delta_{\rho'_1, n_1},\ldots,\Delta_{\rho'_r,n_r})$ be a cuspidal $R'$-multisegment of $GL_n(D)$. For $i=1 \ldots, r$, $\rho'_i$ is an irreducible smooth cuspidal $R'$-representation of $GL_{m_i}(D)$ for a divisor $m_i$ of $n_i$; there exists an unramified  smooth $R'$-character $\chi'_i$ and  an irreducible smooth cuspidal $R'$-representation of $GL_{m_i}(D)$  such that $\rho'_i=\rho_i \chi_i$ and $Z(\Delta_{\rho'_i, n_i})=\chi'_i  Z(\Delta_{\rho_i, n_i})$.  Let $\mathfrak M=(\Delta_{\rho_1, n_1},\ldots,\Delta_{\rho_r,n_r})$ and $\chi'$ the unramified $R'$-character of $M_{n_1, \ldots, n_r}$ corresponding to the $\chi'_i$. Then $Z(\mathfrak M')=\chi' Z(\mathfrak M)$. The Speh $R'$-representation
$Z(\mathfrak M')$ has a germ expansion if and only if the Speh $R$-representation
$Z(\mathfrak M)$ has a germ expansion.
    
  \subsection{} {\bf    $R $ not necessarily algebraically closed} Let $R$ be a field of characteristic different from $p$. 
 We prove that  there is a germ expansion  over  $R$ when there   is a germ expansion over an  algebraic closure $R^{ac}$ of $R$, using the following properties  of the scalar extension  from $R$ to $R^{ac}$ \cite{HV19}:

 For $\pi\in \Rep_R^\infty(G)$ irreducible, the $R^{ac}$-representaton $R^{ac}\otimes_R \pi$ has finite length because $\pi$ is admissible as the characteristic of $R$ is different from $p$. Assume that  there is a map $c: \mathfrak P(n)\to \mathbb Z$ such that  $R^{ac}\otimes_R \pi=  R^{ac}\otimes_R(\sum_\lambda c(\lambda) \pi_{P_\lambda}$ on an open compact subgroup $K$ of $G$. 
The scalar extension  
 $\Gr_R^{\infty} (K)\to \Gr_{R^{ac}}^{\infty} (K)$ from $R$ to $R^{ac}$ is injective.

 Therefore $\pi =\sum_\lambda c(\lambda) \pi_{P_\lambda}$ on $K$. The representation  $\pi$ has a germ expansion with the same map $c_\pi = c_{R^{ac}\otimes_R \pi}=c$.  
The above properties of the scalar extension from $R$ to $R^{ac}$ imply:

 For any irreducible subquotient $\pi'$ of $R^{ac}\otimes_R \pi$, we have 
 \begin{equation}c_\pi = \ell_\pi c_{\pi'} \ \ \text{where $\ell_\pi$ is the length of }  R^{ac}\otimes_R \pi.  
 \end{equation}
Therefore  $c_\pi$ and $c_{\pi'}$ have the same support. As   $c_{\pi'}(\lambda)>0$  when $\lambda$ is minimal in  the support of  $c_{\pi'}$ (Theorem \ref{WhR}), $c_{\pi}(\lambda)>0$.
 This ends the proof of   Theorem \ref{1.1}.

 \subsection{} {\bf The  Jacquet-Langlands correspondence}

The  classical Jacquet-Langlands correspondence  $JL$ between  essentially square integrable representations on both sides, is compatible with character twists and equivariant
under the action of $\Aut(\mathbb C)$.  
Transported 
to $\mathbb Q_\ell^{ac} $ \footnote{ (for the root of $q$ in $\mathbb Q_\ell^{ac} $ image of $\sqrt q\in \mathbb C$ via the isomorphism)}, 
  $$  JL :\Irr^2_{\mathbb Q_\ell^{ac} }(G)\to \Irr^2_{\mathbb Q_\ell^{ac} } (GL_{dn}(F))$$
 preserves the
property of being integral, and two  integrals representations of $G$ are
congruent modulo $\ell$  if and only if their images under $JL$  are
congruent modulo $\ell$ (\cite{MS17} Theorem 1.1). Once a square root of $q=p^f$ in 
$\mathbb Q_\ell^{ac}$ has been chosen  when $f$ is odd to normalize parabolic induction, the Jacquet-Langlands  correspondence $LJ$  transported to the Grothendieck groups
of $\mathbb Q_\ell^{ac}$-representations
 does reduce modulo $\ell$ thus yielding a  map  
for $\mathbb F_\ell^{ac}$-representations (\cite{MS17} Theorem 1.16)
$$LJ: \Gr_{\mathbb F_\ell^{ac}}^{\infty}(GL_{dn}(F))\to \Gr_{\mathbb F_\ell^{ac}}^{\infty}(G).$$
By our argument of reduction modulo $\ell $ in \S \ref{redell} we see that Theorem \ref{tJL} is
valid   for $\mathbb F_\ell^{ac}$-representations.  When $R$ is an algebraically closed field of characteristic different from $p$, 
the reasoning of \S \ref{RR'}  then gives a map 
$$LJ: \Gr_{R}^{\infty}(GL_{dn}(F))\to \Gr_{R}^{\infty}(G)$$
satisfying
Theorem \ref{tJL}  for $R$-representations.
    \begin{theorem} \label{tJLR} (Theorem \ref{tJL}). When $R$ is an algebraically closed field of characteristic different from $p$, for $\nu \in \Gr_{R}^{\infty}(GL_{dn}(F))$  and $ \lambda \in \mathfrak P (n)$, 
  we have
   $(-1)^nc_{LJ (\nu)}( \lambda)=  (-1)^{dn}c_{\nu}(d\lambda)$.
   \end{theorem}

\begin{remark} When $D\neq F$, there are cuspidal complex representations of $GL_n(D)$  that are   isomorphic to their  complex conjugate, and  not   the scalar extension of a real representation. So the Jacquet-Langlands correspondence does not descend to an arbitrary fied $R$.

A  counter-example occurs already  for $D^*$ and $D$ is a quaternion field over $F$ with $q\equiv 3  \mod  4$. Take a regular complex character $\chi$ of $k_D^*$  of order $4$, seen as a character of $O_D^*$ and extended by $-1$ on a uniformizer $p_F$ of $F$. The induced representation $\ind _{F^* O_D^*}^{D^* } \chi$  has dimension $2$ and its image is the quaternion group of order $8$ which is not defined over $\mathbb R$.  The  irreducible representation $\pi^0=JL(\ind _{F^* O_D^*}^{D^* } \chi)$ of  $GL_2(F)$ is cuspidal of level $0$ and can be explicited. For example for $F=\mathbb Q_3$, the irreducible cuspidal representation $\sigma^0$ of $ GL_2(\mathbb F_3)$ corresponding to $\pi^0$   has dimension $2$ and is defined over $\mathbb R$. As the central character of $\pi^0$ is trivial on $O_F^*$,
$\sigma^0$    factorizes by $ PGL_2(\mathbb F_3)=S_4 $ which has all its irreducible representations defined over $\mathbb R$ and even over $\mathbb Q$. \end{remark}

 \section{Invariant vectors  by Moy-Prasad subgroups}\label{11} 
 We prove in this section Theorem \ref{2.2}.
  Let $R$ be a field, $P$   a parabolic subgroup of $G$ of Levi $M$ and $K$ an  open compact subgroup of $G$.
The positive integer 
$$\dim _R (\pi_P)^K= |P\backslash G/K|$$
  depends only on  $[\pi_P]$, hence only on the conjugacy class of $M$ and of $K$.  We can suppose that $P=P_\lambda$ for $\lambda\in \mathfrak P(n)$ and $K\subset K_0$. 
   We have $G=P_\lambda K_0$ and  $P_\lambda \backslash G/K \simeq (P_\lambda \cap K_0) \backslash K_0/K $. 
   
   \begin{example}
We have $(P_\lambda \cap K_0)\backslash K_0/1+ M_n(P_D) \simeq  P_\lambda (k_D)\backslash GL_n(k_D)$ where $k_D=O_D/P_D$ is the residue field of $D$, $q_D$ its cardinality. We deduce
$$|P_\lambda \backslash G/1+ M_n(P_D)|=[n!]_{q_D}/\prod_i [\lambda_i!]_{q_D},$$
where $[n!]_q=\prod_{m=1}^n [m]_q, \ [m]_q= (q^m-1)/(q-1)$ 
 (\cite{Su22} Lemma 1.13).
\end{example}
      
\begin{proposition}Let $G_{x,r}$ denote the  a Moy-Prasad subgroup of $G$ fixing an element $x$ of the building of the adjoint group $\mathcal {BT}$ of $G$,  and $r$ is a positive real number, and $j\in\mathbb N$. We have
\begin{equation}\label{index}     |P \backslash G/ G_{x,r+j/d}| =   |P \backslash G/G_{x,r} | \, q^{d \, d_\lambda j }.
\end{equation}\end{proposition}
 
When $K'$ is a normal  open subgroup of $K$,  
    \begin{equation*}\label{lambdaK}  |P \backslash G/K'| = \sum   _{g\in P \backslash G/K} |P \backslash P g K/K'| , \ \   |P \backslash P g K/K'|=\frac{ [K:K']}{ [(K\cap g^{-1} P  g): (K'\cap g^{-1} P  g)]} .    \end{equation*} 
The group  $G_{x,r+j/d}$  is normal in $G_{x,r}$, and   \eqref{index} follows from  :  
 
\begin{proposition}\label{xrP}  We have $[G_{x,r}\cap  P : G_{x,r+1/d}\cap  P ]=  q^{d\,(n^2-d_\lambda)} $.\end{proposition} 
Note that the index is the same for all  $(x,r)$.  The $D$-dimension of the Lie algebra $\mathfrak p$ of $P$ is  $n^2-d_\lambda$ where $\lambda \in \mathfrak P(n)$ is the partition such that 
 $P$  is  associated to   $P_\lambda$.

\begin{example} 
When $P=G$, then  $\lambda= (n) , d_{(n)}=0$,  $ [G_{x,r}:G_{x,r+ 1/d}]=  q^{d\,n^2}$.

When $P=B$, then $\lambda=(1,\ldots, 1), d_{(1,\ldots, 1)}=n(n-1)/2$, $ [G_{x,r}\cap B:G_{x,r+ 1/d}\cap B]= q^{d\, (n (n+1)/2)} $.
\end{example}

 \begin{proof}  It is more convenient  to use lattice functions rather than points in the Bruhat-Tits building  $\mathcal{BT}$. For that we follow \cite{Broussous-Lemaire02}  denoted here by  [BL]. Recall that a lattice function is a map $\Phi$  from
$\mathbb R$  to $O_D$-lattices in $D^n$ satisfying the conditions of ([BL] Definition 2.1); in 
particular 
\begin{equation}\label{Lambda*} \Phi (s+1/d)=P_D\, \Phi(s) \ \ \text{for any} \  s \in \mathbb R.
\end{equation}
The group $\mathbb R$  acts on lattice functions by translations, and to a lattice function is associated a
point in $\mathcal{BT}$. That point is the same for a translate, and one gets in that way
a $G$-equivariant bijection from the set of lattice functions up to translation onto $\mathcal{BT}$.
For any lattice function  $\Phi$  and  any $r \in \mathbb R$, one defines a lattice  in  $M_n(D)$
$$ \mathfrak g_{\Phi,r}=\{ A\in M_n(D) \ | \  A( \Phi(s))=\Phi(r+s)  \ \ \text{for any} \  s \in \mathbb R\}.
$$
 In their introduction [BL] indicate that
$\mathfrak g_{\Phi,r}=\mathfrak g_{x,r}$ where $x\in \mathcal{BT}$ corresponds to $\Phi$ and $\mathfrak g_{x,r}$ is the lattice in $M_n(D)$ defined
by Moy and  Prasad. They also say that the subgroup 
$G_{x,r}$ for $r\geq 0$, of $G$ defined by Moy and  Prasad satisfies:
$$G_{x,0}=(\mathfrak g_{\Phi,0})^*, \ \ \ G_{x,r}=1+\mathfrak g_{\Phi,r}   \ \ \text{if} \   r>0.$$
They refer to their Appendix A, written by B.Lemaire; the relevant comments are in the lines before their
Proposition A.3.6.   

An immediate consequence of condition \eqref{Lambda*} is that $ \mathfrak g_{\Phi,r+1/d}=P_D\, \mathfrak g_{\Phi,r}$. That implies in particular
that 
$$[\mathfrak g_{\Phi,r}:  \mathfrak g_{\Phi,r+1/d}]= q^{d\, n^2}   \ \ \text{for any} \   r>0.$$ More generally, if $W$ is a sub-$D$-vector space of $M_n(D)$,
$ \mathfrak g_{\Phi,r+1/d}\cap W=P_D\, ( \mathfrak g_{\Phi,r} \cap W)$. Applying that to $\mathfrak p $,
we get  
$$[  \mathfrak g_{\Phi,r }\cap \mathfrak p :   \mathfrak g_{\Phi,r+1/d } \cap \mathfrak p ]= q^{d\, dim_D(\mathfrak p)}   \ \ \text{for any} \   r>0.$$
This  proves the proposition because  $[G_{x,r}\cap  P : G_{x,r+1/d}\cap  P ]=  [  \mathfrak g_{\Phi,r }\cap \mathfrak p :   \mathfrak g_{\Phi,r+1/d } \cap \mathfrak p ] $ {\color{blue} for   $r>0$} and $\dim_D(\mathfrak p)= n^2-d_\lambda$.
 \end{proof}
 
 We deduce:
\begin{corollary} Let $P$ be  a parabolic subgroup of $G$ associated to $P_\lambda$ for $\lambda\in \mathfrak P(n)$, and  $G_{x,r+j/d}$ a  Moy-Prasad subgroup for $x\in \mathcal{BT}, r\in \mathbb R, r>0$ and $j\in \mathbb N$. We have for $g\in G$, 
\begin{align}\label{cxrP} |P \backslash P g G_{x,r}/G_{x,r+1/d}| &= \frac{  [G_{x,r}:G_{x,r+ 1/d}]}{ [(G_{x,r}\cap   P  ): (G_{x,r+1/d}\cap  P  )]} = q^{d\, d_\lambda}.
\end{align}
\end{corollary}

Clearly,  \eqref{index} follows from \eqref{cxrP}.

\begin{example} 1) For a vertex  $x $ of $\mathcal{ BT}$, the Moy-Prasad group $G_{x,0}$ is conjugate to $K_0=GL_n(O_D)$ and  $G_{x, r}$ is conjugate to $K_1=1+p_D M_n(O_D)$ for  $0<r\leq 1/d$.    
Hence $$|P_\lambda \backslash G/G_{x,r} | = \begin{cases} |P_\lambda \backslash G/ K_0|=1 \ \text{ if } \ r=0,\\
 |P_\lambda \backslash G/  K_1|= \frac{[n]_{q^d} ! }{ \prod_k [\lambda_k]_{q^d}!}\ \text{ if } \ 0<r\leq 1/d.
 \end{cases}
 $$
 where $[n]_q !=\frac{q-1}{q-1}\ldots\frac{q^n-1}{q-1}$. Indeed    $ |P_\lambda \backslash G/ K_0|=1$ because $G=P_\lambda K_0$, and   
  $ |P_\lambda \backslash G/ K_1|=[GL_n(\mathbb F_{q^d}):P_\lambda (\mathbb F_{q^d})] $.
  
 2) For  the barycenter $x$ of an   alcove, $G_{x,0}$ is  conjugate to the Iwahori group $I$, inverse image  in $K_0$ of the upper triangular group of $GL_n(k_D)$,   and $G_{x,r}$ is conjugate to  the pro-Iwahori group  $I_{1/d}$, inverse image of the strictly upper triangular group of $GL_n(k_D)$,  for $0<r\leq 1/d$. Write  $\mathfrak J$ for the lattice of $(x_{i,j}) \in M_n(O_D)$ with $x_{i,j}\in P_D$ when $i>j$, and $\mathfrak J_{1/d} $ for   the lattice  of $(x_{i,j}) \in M_n(O_D)$ with $x_{i,j}\in P_D$ when $i\geq j$.   Then,
 $$  I = \mathfrak I^*, \ \   I_{1/d}=1+ \mathfrak J_{1/d}   \ \ \text{for} \ 0<r\leq 1/d.$$
 We have
     $ P_\lambda \backslash G/ I \simeq P_\lambda \backslash G/ I_{1/d} \simeq (S_{\lambda_1}\times \ldots \times S_{\lambda_r})\backslash S_n$  hence
$$|P_\lambda \backslash G/G_{x,r} | = |P_\lambda \backslash G/ I|=  |P_\lambda \backslash G/  I_{1/d}|= \frac{n!}{ \prod_k \lambda_k !}. 
 $$
 \end{example}
   \begin{remark} \label{rxrP}  Proposition \ref{xrP} reduces the computation of
$|P_\lambda \backslash G/G_{x,r} |$ for $r>0$ to the  case   $0<r <1/d$. For $g\in G,x\in \mathcal {BT}, r\geq 0$, we have  $g G_{x,r} g^{-1} = G_{g(x), r}$; this reduces 
the computation of
$|P_\lambda \backslash G/G_{x,r} |$ for $x\in \mathcal {BT}$ to the case where $x$ belongs to the
 the closed alcove  $\mathcal A$ of $\mathcal {BT}$ determined by $B$.  
 \end{remark}
      
Theorem \ref{1.1}  implies for $\pi\in \Rep_R ^{\infty, f}(G)$, 
   \begin{equation}\label{piK} \dim_R\pi^{G_{x,r+j/d}} = \sum_{\lambda\in \mathfrak P(n)} c_\pi(\lambda) \, |P_\lambda \backslash G/G_{x,r+j/d}|. 
  \end{equation}  
and the integer $c_{\pi}(\lambda) $ is positive  if  $d_\lambda =  d(\pi)$ then $\lambda $ is minimal in the support of $ c_\pi$.  Applying    \eqref{index}, we deduce Theorem \ref{2.2}.

\begin{remark}\begin{enumerate}
 
\item The polynomial $P_{\pi, G_{x,r}}(X)$   is  determined by those  where $x$ is in a closed alcove of $\mathcal BT$ 
and  $0<r <1/d$  because $$ P_{\pi, G_{x,r+j/d}}(X)=P_{\pi, G_{x,r}}(q^{dj}X) \ \ \text{for } \ 0<r <1/d, \ j\in \mathbb N.$$
  $$ P_{\pi, G_{x,r}}(X)=P_{\pi, G_{g(x),r}}(X)  \ \ \text{for } \ 0\leq r,  \ g\in G.$$
  
\item For   $\pi\in \Rep_R ^{\infty, f}(G)$, and any Moy-Prasad pro-$p$ group $G_{x,r}$ of $G $
  $$ \dim_{R}\pi^{G_{x,r+j/d}}\  \sim \ a_{\pi, G_{x,r}} \,  q^ {d(\pi) d j}   \ \ \text{when $ j\in \mathbb N$ goes to infinity}.$$
The  integer  $d(\pi)$   can be called the  {\bf Gelfand-Kirillov dimension} of $\pi$.

\end{enumerate}  
\end{remark}    
  \section {$G=GL_2(D)$}\label{n=2}
In this section we assume that $G=GL_2(D)$, $R$ is a field of characteristic different from $p$ except in \S \ref{=p}  where its characteristic is $p$, and we give more details on the polynomial $P_{\pi,K}(X)$ attached to $\pi\in \Rep_R^{\infty, f}(G)$ and   a Moy-Prasad  subgroup $K$.

\subsection{}

The Moy-Prasad open compact subgroups of $G $ are conjugate to the open compact subgroups
$$K_0  \supset I_0 \supset I_{1/2} \supset K_1  \supset I_1 \supset I_{3/2} \supset K_2 \supset I_2 \supset \ldots ,  $$
where $K_0=GL_2(O_D), \  I_0= \mathfrak j^*=\red^{-1} B(k_D) $ an Iwahori subgroup,  $I_{1/2}=1+ \mathfrak j_{1/2}=\red^{-1} U(k_D)$ a pro-$p$ Iwahori subgroup, for $j\in \mathbb N$,  
 $$I_{j+1/2}= 1 + p_D^{j} \mathfrak j_{1/2},  \ \ K_{j+1}=1+ p_D^{j+1} M_2(O_D), \ \   I_{j+1}= 1 + p_D^{j+1}\mathfrak j,$$ 
  where $ \mathfrak j$ is the lattice of $(x_{i,j})\in M_2(O_D) $ with $x_{2,1}\in P_D$,   
 $ \mathfrak j_{1/2}$ is the lattice of $(x_{i,j})\in  \mathfrak j$ with $x_{1,1}, x_{2,2}\in P_D$, and  $\red: GL_2(O_D)\to GL_2(k_D)$ is the reduction modulo $p_D$.  

  The  parahoric subgroups of $G$ are conjugate to $K_0$ and $I_0$. The Moy-Prasad subgroups of $G$ which are pro-$p$ groups are conjugate of $I_{j+1/2}, K_{j+1},  I_{j+1}$ for $j\in \mathbb N$ \footnote{The indices of the preceding section have been multiplied by $d$}.
  
  \bigskip  To justify the preceding assertions, it is convenient to use  of lattice functions $\Phi$ from $\mathbb R$ to in $D\oplus D$, as  in the proof of Proposition \ref{xrP}.
The lattice function $\Phi_0$ with value $L_0=O_D\oplus O_D$ at $0$ and $P_DL_0 $ at $s$ for $0<s<1/d$ gives a vertex $x_0$ in the Bruhat-Tits tree ${\mathcal BT}$ of $G$, and  
 $G_{x_0,0} =\mathfrak g_{\Phi_0, 0}^*$ is the stabilizer $K_0$ of $L_0$, whereas $G_{x_0,r}=1+\mathfrak g_{\Phi_0, r}$ for $r>0$ so that
$G_{x_0,r}=K_{j+1}$  if $dr=j+s$ with $0<s\leq 1$. This gives the groups
$K_j$ in the list and accounts for all Moy-Prasad subgroups associated to the vertices of ${\mathcal BT}$. 

Any point in ${\mathcal BT}$ is sent by $G$ to a point in the segment with ends $x_0$ and the vertex $x_1$  corresponding to $L_1=O_D\oplus P_D$  so it suffices to look at the Moy-Prasad
subgroups $G_{x_\alpha,r}$ when  $x_\alpha$ is a barycenter $\alpha x_0+(1-\alpha)x_1 $ with  $0<\alpha <1$. Since there is an element of $G$ exchanging
$x_0$ and $x_1$, we need only look at $0<\alpha <1/2$  which we now assume.
A lattice function $\Phi_\alpha$ giving $x_\alpha$ takes value $L_0$ at $0$, $L_1$ at $s$ if $0<s\leq \alpha/d$ and $P_DL_0$ if $\alpha/d<s \leq 1/d$.
Then $G_{x_\alpha,0 }$ is the intersection of the stabilizers of $L_0 $ and $L_1$, that is $I_0$.

\noindent For $0<dr\leq \alpha$,   $G_{x_\alpha,r+j/d}=I_{j+1/2}$  for  any $j\in \mathbb N$, as
 $\mathfrak g_{\Phi_\alpha, r} $ is the set of $X\in M_2(D) $ sending $L_0 $ in $L_1$, and $L_1 $ in $P_DL_0$.

\noindent For $\alpha<dr\leq 1-\alpha$ (which cannot happen if $\alpha=1/2$),  $G_{x_?,r+j/d}=K_{j+1}$ for  any $j\in \mathbb N$, as  $\mathfrak g_{\Phi_\alpha, r}  $ is the set of $X\in M_2(D) $ sending $L_0  $ and $L_1$ in $P_DL_0$.

\noindent When $1-\alpha<dr<1 $ we find similarly that $G_{x_\alpha,r+j/d}=I_{j+1}$   for  any $j\in \mathbb N$.

\bigskip

  The indices between two consecutive groups  are 
$$[K:I]=q+1, \ [I : I_{1/2}]=  (q-1)^2, \ [I_{1/2} :K_1]=q, \ [K_1:I_1]=q, \  [I _1: I_{3/2}]=q^2,\ [I_{3/2} :K_2]=q,$$
 and so on.  
 Proposition \ref{xrP}, Corollary \ref{cxrP} and Remark \ref{rxrP} give the integers 
   
 \begin{itemize}
 
\item  $|B\backslash G/K_0|=1$ as $G=BK_0$,

\item $|B\backslash G/I_0|=|B\backslash G/I_{1/2}| =2$ as $G=BI \sqcup BsI=BI_{1/2} \sqcup BsI_{1/2}$, where $s$ is the antidiagonal matrix with coefficients $1$. 

\item $|B\backslash G/K_1|=(q^{2d}-1)(q^{2d}-q^d)/q^d (q^d-1)^2=q^d+1$.

\item $|B\backslash G/I_1|=2q^d$ because $B\backslash G/I_1=B\backslash BI/I_{1} \sqcup  B\backslash BsI/I_{1}$  and 

$B\backslash BI/I_{1} \simeq (B\cap I)\backslash I/I_1 \simeq ((B\cap I)/ (B\cap I)_1) \backslash(I/I_1)$,  

$B\backslash BsI/I_{1} \simeq B^-\backslash G/I_1 \simeq ((B^-\cap I)/ (B^-\cap I)_1) \backslash(I/I_1)$,

$|(I_1\cap B)\backslash (I\cap B)|=|(I_1\cap B ^-)\backslash (I\cap B^-)|=  (q^d-1)^2 q^d$ and $[I:I_1]=(q^d-1)^2 q^{2d}$.
\item $|B\backslash G/I_{j+1/2}|= 2 q^{dj}$.
\item $|B\backslash G/K_{j+1}|= (q^d+1) q^{dj}$.
\item $|B\backslash G/I_{j+1}|= 2  q^{d(j+1)}$.
\end{itemize}  

\subsection{}There are  only two nilpotent orbits $\{0\}$ and $\mathfrak O\neq \{0\}$ corresponding to the partitions $(2)$ and $(1,1)$ of $2$. By the germ expansion 
for $\pi \in \Rep_{R}^{\infty, f}(G)$ (Theorem \ref{1.1}), there exists $a_\pi, b_\pi \in \mathbb Z$ and an integer $j_\pi\geq 0$ such that for any integer $j\geq j_\pi$
 \begin{itemize}
 \item $ \dim _{\mathbb C}\pi^{I_{1/2+j}}=a_\pi + \, 2\, b_\pi  \,q^{dj}   $,
  \item $ \dim _{\mathbb C}\pi^{K_{1+j}}=a_\pi  \, + \, (q^d+1) \, b_\pi  \,  q^{dj} $,
 \item $ \dim _{\mathbb C}\pi^{I_{1+j}}=a_\pi  \, + \, 2  q^d\, b_\pi \, q^{dj}$.
 
 \end{itemize}
\subsection{}  The  maps $\pi\mapsto a_\pi $  and  $\pi \mapsto b_\pi $  are additive  hence  determined by their values on irreducible representations.    For  $\pi\in \Rep_R^\infty(G)$ irreducible,
  \begin{itemize}
   \item   $ a_\pi   =\dim_R \pi ,  \  b_\pi = 0 $ \   if  the  dimension of $\pi$ is finite ( $\dim_R \pi =1$ if $R$ is algebraically closed),
  \item  $ b_\pi  > 0 $ \  if  the  dimension of $\pi$ is infinite.
  \end{itemize}
  The dimension of  $\sigma \in  \Rep_{\mathbb C}^{\infty,f}(T)$  is finite and by Theorem \ref{th:71} for  $\pi=\ind_B^G \sigma$,
\begin{itemize}
   \item  
  $a_\pi=0, \ b_\pi=\dim_{R} \sigma$.
  \end{itemize} 
The $R$-representation  $\ind_B^G 1$ contains the trivial representation $1$ of $G$ and the quotient $\St$ is called the Steinberg representation. By additivity,  $a_1+a_{\St}=a_{\ind_B^G1}, \ b_1+b_{\St}=b_{\ind_B^G1}$ hence 
 \begin{itemize}
   \item  
  $a_{\St}=-1 , \ b_{\St}  =1$.
  \end{itemize} 
  For $g\in G$,  let  $v_D(g)$  be the integer such that $|\nrd(g)|=q^{v_D(g)}$.
  \begin{proposition}  The Steinberg $R$-representation $\St$ of $G$ is reducible if and only if   $\St$  is indecomposable of  length
$2$, with a cuspidal  subrepresentation $c\St$ and the character $(-1)^{v_D(g)}$  as a quotient,  if and only if   $\charf _R=\ell>0$ divides $q^d+1$. 

The representation $\ind_B^G 1$ is indecomposable except when $\charf _R=\ell $ is odd and divides $q^d-1$. \end{proposition}

\begin{proof} This is  proved in \cite{V96} if $D=F$, and follows from \cite{MS14} in general. 
 We indicate how to get the result using  techniques of  \cite{V96}. The restriction of $\ind_B^G(1)$ to $B$  is the direct sum $\ind_B^G1=1\oplus \tau$ of the trivial representation $1$ on the
line of constant functions and of  the representation $\tau$ on the space of functions vanishing at $1$, i.e.
with support in $BsN$, 
 isomorphic to the representation of $B$ by  conjugation on $C_c^\infty(N;R)$.
 Integrating such functions on $N$ against a Haar measure
(that is taking coinvariants) gives that the modulus $\delta_B$ of $B$ is a quotient of $\tau$. Moreover  $\delta_B$ does not split as a subrepresentation of $\tau$  since $\delta_B$ is trivial on $N$ and obviously the restriction of
$\tau$ to $N$ has no trivial subrepresentation. One proves as in (\cite{BH06} 8.2) that the corresponding subrepresentation $\tau^0$ of $B$ is irreducible, 
 so  $\tau$  is indecomposable of length $2$ with quotient $\delta_B$.
 
Thus $\ind_B^G(1)$ has length $\leq 3$, and it can have length $3 $ only if  $ \delta_B$ extends to an  $R$-character $G$. This latter property is equivalent to
$q^{2d}=1$ in $R$ because  $\delta_B$  is the inflation of the character $\nu^d\otimes \nu^{-d}$ of $T$ where $\nu$ is the character $\nu(x)= |\nrd (x)|$ of $D^*$.  If  $\charf _R=0$ or $\charf _R =\ell>0$ not dividing   $q^{2d}-1$, then $\St$ is irreducible. Otherwise, $\delta_B$ extends to the 
  the character $\nu^d$ of $G$ where $\nu(g)=|\nrd (g)|$ for $g\in G$,  the contragredient $ \ind_B^G(\delta_B)=  \nu^d \otimes \ind_B^G 1$ of $\ind_B^G(1)$   has a unique 
one-dimensional subrepresentation   $\nu^d$, which is consequently a quotient of 
$\ind_B^G(1)$.  
If $\ell$ divides $q^d+1$ but not  $q^d-1$,  the character $\nu^d=(-1)^{\val_D}$ is
not trivial,  then $\ind_B^G(1)$ is indecomposable of length $3$ and  $\St_G$ is indecomposable of  length $2$ with quotient $(-1)^{\val_D}$.

 If $\ell$ divides $q^d-1$, $\delta_B$ is trivial and $B\backslash G$ admits a $G$-invariant measure giving  volume $ 0$ to $B\backslash G$ if  $\ell$ divides also $q^d+1$ (which means $\ell=2$) and $1$ otherwise. Integration on $B\backslash G$  implements the duality between $\ind_B^G(1) $ and itself.   The 
integration on $B\backslash G$ is  $0$ on the constant functions if $\ell$  divides $q^d+1$ and the identity otherwise. 
Therefore if  $\ell$  divides $q^d+1$,
the space of constant functions is isotropic,
so its orthogonal has codimension $1$, and again $\ind_B^G(1)$  is indecomposable of length $3$ and $\St$ is indecomposable of length $2$ with quotient the trivial representation. But if $\ell$ does not   divides $q^d+1$,  $\ind_B^G(1)=1 \oplus \St$  and $\St$ is irreducible otherwise it would have a cuspidal
subquotient which would be contained in $\ind_B^G1$  (autodual)  which is impossible by Frobenius.   \end{proof}
 By additivity,
 \begin{itemize}
   \item  
  $a_{c\St}=-2 , \ b_{c\St}=1$.
  \end{itemize}

  When $\mu_{p^\infty} \subset R$, there are  two kinds of Whittaker spaces for $\pi$: the trivial one, dual of the $U$-coinvariants $\pi_U$ of $\pi$, 
and the non-degenerate one, dual of the coinvariants $\pi_{U,\theta}$ of $\pi$ by a non trivial character $\theta$ of $U$. By Theorem \ref{Wh} we have for $\pi$ irreducible 
 \begin{itemize}
   \item  
   $ b_\pi  =  \dim_{R}(\pi_{U,\theta})$,
  \end{itemize}
 This equality is valid   when $\pi$ has finite length because the $\theta$-coinvariant functor is exact.  In particular for $\sigma \in  \Rep_{\mathbb C}^{\infty,f}(T)$ 
    \begin{itemize}
   \item  
   $  \dim_{R}(\ind_B^G \sigma)_{U,\theta} =  \dim_{R}\sigma$
  \end{itemize}
  
 \subsection{}Assume   $R=\mathbb C$ and $\sigma \in  \Rep_{\mathbb C}^{\infty}(T)$  irreducible. The normalized parabolic induction $ \ind_B^G (\delta_B^{1/2}\otimes \sigma) $  of $\sigma$ is reducible if and only if $\sigma = \rho \otimes ( \chi_\rho \otimes \rho) $  where $\rho$ is
 an irreducible representation of $D^*$, and   $\chi_\rho$ the unramified character of $D^*$ giving the cuspidal segment $\Delta_{\rho,2}= \{\rho,   \chi_\rho \otimes \rho \}$ (\cite{LMT16} for a proof which does not use the Jacquet-Langlands correspondence). In this case,  $ \ind_B^G (\delta_B^{1/2}\otimes \sigma) $ is  indecomposable of length $2$, one irreducible subquotient is the Speh representation $Z(\Delta_{\rho,2} )$ and the other subquotient is an
 essentially square integrable representation $L(\Delta_{\rho,2} )$.  

The Speh subrepresentation $Z(\Delta_{\rho,2} )$   is a character if and only if $\dim_{\mathbb C} \rho= 1$. In that case,  $L(\Delta_{\rho,2} )$  is the twist of the Steinberg representation $\St$ by this character.  The twist of $\pi$ by a character does not change the value of the  $a_\pi, b_\pi$. Hence 
 \begin{itemize}
   \item  
  $a_{L(\Delta_{\rho,2} )}=-1 , \ b_{L(\Delta_{\rho,2} )}  =1$ \ if   $\dim_{\mathbb C}\rho =1$,
  \end{itemize}
 by unicity of the Whittaker model as $ b_{L(\Delta_{\rho,2} )}  =\dim _{\mathbb C} (L(\Delta_{\rho,2})_{U,\theta} > 0$.
 
 When $D\neq F$, there are irreducible complex representations $\rho$ of $D^*$ of dimension $>1$. In that case,  the Speh representation $Z(\Delta_{\rho,2} )$ is infinite  dimensional hence is generic.     The  essentially square integrable  representation  $L(\Delta_{\rho,2} )$ is also infinite dimensional hence generic; it
 corresponds by Jacquet-Langlands to an irreducible representation $\pi_{\rho,2}$ of the multiplicative group $D_{2d}^*$ of a central division $F$-algebra  of reduced dimension $2d$.  Recalling  Corollary \ref{cJL},  we have  when  $\dim_{\mathbb C}\rho>1$:
 \begin{itemize}
   \item  $a_{Z(\rho,2)}  = - a_{ L(\rho,2)} =  \dim _{\mathbb C}   \pi_{\rho,2} $, 
   \item    $ b_{Z(\rho,2)} + b_{L(\rho,2)}=\dim _{\mathbb C} (\ind_B^G\sigma)_{U,\theta}= \dim_{\mathbb C} \sigma, $ \ 
   
 \noindent      $b_{Z(\rho,2)}  =\dim _{\mathbb C} Z(\Delta_{\rho,2})_{U,\theta} >0$, \  $b_{L(\rho,2)}  =\dim _{\mathbb C} L(\Delta_{\rho,2})_{U,\theta} >0$. 
    \end{itemize} 
The $T$-stabilizer  of the  non-trivial character  $\theta(u)=\psi \circ \trd (v)$ for $u=1+v$ in $U$, 
$$T_\theta =\{\diag(d,d) \ | \  d\in D^*\},$$
 acts  naturally   on the $\theta$-coinvariants of a representation of $G$.    How does one identify 
  the two  factors  of $(\ind_B^G\sigma)_{U,\theta} =Z(\rho,2)_{U,\theta}  \oplus L(\rho,2)_{U,\theta} $ ? We shall come back to that question in future work.\footnote{After this paper was written, we received a paper of S. Nadimpalli and M. Sheth \cite{NS23} calculating the dimensions of the two factors for certain $\rho$}
 
 \bigskip  When $\pi\in \Rep_{\mathbb C}^\infty (G)$ irreducible  is not isomorphic to a subquotient of $\ind_B^G\sigma$ for $\sigma \in  \Rep_{\mathbb C}^{\infty}(T)$ irreducible,  it is called supercuspidal. Its dimension is infinite, it is essentially square integrable and corresponds  by Jacquet-Langlands to an irreducible representation $\pi_{2}$ of  $D_{2d}^*$. We have for $\pi\in \Rep_{\mathbb C}^\infty (G)$ irreducible supercuspidal  (Corollary \ref{cJL}): 
\begin{itemize}
\item  $a_{\pi }  =  - \dim _{\mathbb C}    \pi_2,  \ \ \ b_\pi    = \dim_{\mathbb C} (\pi)_{U,\theta} >0$.
  \end{itemize} 
  For some supercuspidal representation $\pi$, D. Prasad and A. Raghuram  computed $\dim_{\mathbb C} (\pi)_{U,\theta}$ \cite{PR00}.
When $D=F$, $b_\pi =1$ by the unicity of the non-degenerate Whittaker model. The explicit classification  of the irreducible cuspidal  representations of $GL_2(F)$ or the explicit Jacquet-Langlands correspondence alllows to compute explicitely $a_\pi $.  The normalized level $\ell(\pi)$  of $\pi\in\Rep_{\mathbb C}^\infty (GL_2(F))$ irreducible defined in  (\cite{BH06}  12.6)   is
the minimum of two half-integers: the smallest integer $j $ such that $\pi^{K_{j+1}}\neq 0$  and the smallest
element $j\in 1/2\mathbb Z$  such that  $\pi^{ I_{j+1/2}}\neq 0$.  It is equal to the depth of  $\pi$ defined in  \cite{MP96}. Since $a_\pi$ stays the same if we twist  $\pi$ by a character, we may assume that $\pi$ is minimal in the sense that $\ell(\pi)\leq \ell(\pi\otimes \chi)$
for any character $\chi$ of $GL_2(F)$.

\begin{proposition}   For  $\pi\in\Rep_{\mathbb C}^\infty (GL_2(F))$ irreducible cuspidal and minimal, we have 

$a_\pi=-2q^{\ell(\pi)}$ if  $\ell(\pi)$  is an integer, and 
$a_\pi=-(q+1)q^{\ell(\pi) -1/2}$ otherwise.
\end{proposition}

\begin{proof}  It is easier to use the Jacquet-Langlands correspondence. We compute $\dim _{\mathbb C}    \pi_2 $, where $\pi_2$ is the irreducible smooth representation of  $D^*_2$ corresponding to $\pi$. The  level  $\ell(\pi_2)$ of $\pi_2$   is the smallest integer $j$  such that $\pi_2$ is trivial on $1+P_{D_2}^{j+1}$, and 
shows that $\ell(\pi_2)=2\ell(\pi)$ (\cite{BH06} 56.1). Since the Jacquet-Langlands correspondence is compatible with character twists, $\pi_2$ is minimal.
By  (\cite{BH06} 56.4 Proposition) $\pi_2$  is induced from a representation $\Lambda$  of a subgroup $J$ of $D_2^*$  described in (\cite{BH06} 56.5 Lemmas 1 and 2). If $\ell(\pi_2)=2j+1 $ is odd, then $J=E^*(1+P_{D_2}^{j+1})$ where $E/F$  is a ramified
quadratic extension in the quaternion division algebra $D_2$, and $\Lambda$ is a character, so that  $\dim_{\mathbb C} \pi_2= (q+1)q^j$, confirming the second case in the proposition.
If  $\ell(\pi_2)=2j$  is a multiple of $4$, then $J=E^*(1+P_{D_2}^{j+1})$ where $E/F$ is now   unramified    and $\Lambda$  is again
a character, so that $\dim_{\mathbb C} \pi_2=2q^j$. Finally if $\ell(\pi_2)=2j $ is not a multiple of $4$, then $J $ contains $E^*(1+P_{D_2}^{j+1}) $ with index $q^2$, where
again $E/F$ is   unramified, but $\Lambda$ has dimension $q$, so that $\dim_{\mathbb C} \pi_2=2q.q^{2j}/q^{j-1}q^2=2q^j$ as expected.
\end{proof}
 
\begin{remark}1)  If $\pi$  is cuspidal and minimal, and  $\pi^{I_j}=0$ for an integer $j>0$ then  $\pi^{K_j}=0$, so that the exponent of $q$ in the proposition is the smallest integer such that $\pi ^{K_{j+1}}\neq 0$.
 
 2) As pointed out in (\cite{BH06} Chapter 13, 56.9: Comments), the Jacquet-Langlands correspondence
there is characterized by its compatibility with character twists and preservation of the $\epsilon$-factors.
But since the Jacquet-Langlands characterized by equality of characters possesses those properties,
both correspondences are the same.
 
 3) Instead of using the Jacquet-Langlands correspondence in the proof we could have used the known fact
that the character of $\pi$ is constant, equal to $-\delta(\pi)/\delta(\St_G)$, on elliptic regular elements close to identity,  where
$\delta$ denotes the formal degree (\cite{Howe74} when $\charf_F=0$, 
\cite{BHL10} when $\charf_F=p$). The quotient $\delta(\pi)/\delta(\St_G)$ has been computed
 for $GL_n(F)$ when $n$ is prime in  (\cite{Ca84} Section 5).
\end{remark}
\subsection{}\label{=p}{\bf Coefficient field of characteristic $p$} Up to now the characteristic of the coefficient field $R$ was $p$.  But some results may remain true for a field $R$ of characteristic $p$, for example the  dimension of the invariants  of  an irreducible admissible non-supersingular $R$-representation of  $G=GL_2(D)$, by congruence subgroups of Moy-Prasad subgroups of $G$ (Theorem \ref{2.2}).

Let  $R$ be a field of characteristic $p$ and   $\sigma =\rho \otimes \rho' \in \Rep_R^\infty (T)$ irreducible,  $\rho, \rho' \in \Rep_R^\infty D^*$. 
If the inflation $\tilde \sigma $ of $\sigma$ to $B$ does not extend to $G$, the  parabolically induced representation   $\ind_B^G \sigma$ is irreducible. Otherwise,  $\rho\simeq \rho'$, $\ind_B^G \sigma$ is indecomposable of length $2$, contains the (unique)    finite dimensional representation $\sigma_G$  extending $\tilde \sigma$, of quotient   $\sigma_G\otimes  \St$ where $\St=\ind_B^G1/1$ is the Steimberg representation. 
Those  are the not supersingular irreducible representations  (\cite{AHHV17} when $R$ is algebraically closed and \cite{HV19} in general). 

\begin{lemma} When $\tilde \sigma $  extends to a representation $\sigma_G$ of $G$, we have  $\sigma_G=\tau \otimes \nrd_{G/F^*} $,  and $\sigma\simeq \rho \otimes \rho $ with $\rho \simeq \tau \otimes \nrd_{D^*/F^*}$ for $\tau\in\Rep_{R}^\infty F^* $ irreducible. 
\end{lemma}
\begin{proof}
 When  $R$ is algebraically closed,  this follows from  Lemma \ref{fd}. In general,  let $R^{ac}/R$ be an an algebraic closure.
There exists  a character $\chi\in \Rep_{R^{ac}}^\infty F^*$ such that 
$\chi \otimes \nrd_{G/F^*} , \chi \otimes \nrd_{D^*/F^*}$ is a subquotient to 
$R^{ac}\otimes_R\sigma_G, R^{ac}\otimes_R\rho \simeq R^{ac}\otimes_R\rho' $. Let $\tau\in\Rep_{R}^\infty F^* $ irreducible such that  $\chi $ is a subquotient to $R^{ac}\otimes_R\tau $.  Then $\sigma_G=\tau \otimes \nrd_{G/F^*} $,  $\rho\simeq \rho'\simeq \tau \otimes \nrd_{D^*/F^*}$. 
\end{proof}
\begin{proposition} Let $\pi\in \Rep_{R}^\infty(G)$ irreducible not supersingular. For $j\geq 0$, we have
 \begin{itemize}
 \item $ \dim _{R}\pi^{I_{1/2+j}}=a_\pi + \, 2\, b_\pi  \,q^{dj}   $,
  \item $ \dim _{R}\pi^{K_{1+j}}=a_\pi  \, + \, (q^d+1) \, b_\pi  \,  q^{dj} $,
 \item $ \dim _{R}\pi^{I_{1+j}}=a_\pi  \, + \, 2  q^d\, b_\pi \, q^{dj}$,
  \end{itemize} 
where
$$(a_\pi , b_\pi) = \begin{cases}  

(0, \dim_{R}\sigma) &\ \text{ if } \pi =\ind_B^G\sigma\\
 
(\dim_{R}\sigma, 0) &\ \text{ if } \pi = \sigma_G\\ 
(-\dim_{R}\sigma, 1)&\ \text{ if } \pi =  \sigma_G \otimes \St
\end{cases} ,  
$$  and  $\sigma\in \Rep_R^\infty(T^*)$. \end{proposition}
 \begin{proof} The  formulas  for  a finite dimensional representation  and  for 
$\ind_B^G\sigma$ are clear. They imply the 
  formula for the  twisted Steinberg representations   by the next proposition.
  \end{proof}
  
  \begin{proposition} Let $R$ be a field,   and $K$  a Moy-Prasad pro-$p$ subgroup of $G$. The natural map $(\ind_B^G 1)^K\to \St^K$ is surjective.
  
  \end{proposition}
  \begin{proof} When  $\charf _R\neq p$, the $K$-invariant functor is exact and the surjectivity is clear. When  $\charf_R=p$,  one can argue as follows. 

The image of  $f\in \ind_B^G1$    in $\St$   is $K$-invariant if and only if  there exists a map $c_f:K\to R$
 such that 
$f(gk) = f(g)+ c_f(k)$ for any $g\in G, k\in K$.
As $f(gkk') = f(g)+ c_f(kk')= f(gk)+c_f(k')= f(g)+c_f(k)+ c_f(k')$ for  $k,k'\in K$, the map  $c_f $ is an homomorphism.
For $k\in K\cap B$ we have $f(k)=f(1)$ hence $c_f(k)=0$. For  $k\in K\cap sBs$ 
we have   
$f(sk)=f(skss)=f(s)$  because $sks\in B$, hence  $c_f(k)=0$. As $K\cap B$ and $K\cap s B s$ generate $K$, we deduce that $c_f=0$. 
   \end{proof}

  Let $G=GL(2,\mathbb Q_p)$ and $\pi\in \Rep_{\mathbb F_p^{ac}}(G)$   irreducible  supersingular.     By (\cite{Morra13} Proposition 4.9, Corollary 4.15), for $p$ odd and $j\geq 0$, we have:
 \begin{itemize}
 \item $ \dim _{\mathbb C}\pi^{I_{1/2+j}}=a_\pi+ 2 b_\pi \, p^j$,  where $ (a_\pi, b_\pi)=(-2 , 2)$, 
  \item $ \dim _{\mathbb C}\pi^{K_{1+j}}= a'_\pi \, + \, (p+1) \, b_\pi  \,  p^{j} $, where
 $$a'_\pi = \begin{cases}-3  \ \text{  if \ }\pi=  \pi_0\otimes (\chi \circ \det) \ \text{ for a character } \chi \in \Rep_{\mathbb F_p^{ac}}
 ^\infty F^* \\
 -4 \ \text{ otherwise} 
 \end{cases}.$$
 \end{itemize} 
Here  $\pi_0$ denote  the supersingular irreducible  quotient of $\mathbb F_p^{ac}[GL(2,\mathbb Z_p) Z\backslash G ]$, $Z$ the center of $G$.


\begin{thebibliography}{} 
  
 \bibitem[Aubert95]{Au95} A.-M. Aubert -- Dualit\'e dans le groupe de Grothendieck de la cat\'egorie des repr\'esentations lisses de longueur finie d'un groupe r\'eductif $p$-adique.
Transactions of the American Mathematical Society, Vol. 347, No. 6. (1995), pp. 2179-2189.

 \bibitem[Aizenbud-BS22]{ABS22} Avraham Aizenbud, Joseph Bernstein, Eitan Sayag -- Strong density of spherical characters attached to unipotent subgroups arXiv:2202.04984 (2022)
 
 \bibitem[Abe-Herzig23]{AH23} N. Abe,   F. Herzig -- On the irreducibility of $p$-adic Banach principal series of $p$-adic reductive groups.  arXiv:2303.13287v2 (2023)
  
\bibitem[AbeHenniartHerzigVign\'eras17]{AHHV17} N. Abe,  G. Henniart, F. Herzig, M.-F. Vign\'eras. --  {A classification of admissible irreducible modulo $p$ representations of reductive $p$-adic groups}. J. Amer. Math. Soc. 30 (2017), 495-559.

\bibitem[Badulescu02]{Bad02} A.I.Badulescu, Correspondance de Jacquet-Langlands en caract\'eristique non nulle, Ann. Scient. Ec. Norm. Sup. 35 (2002), 695-747.

\bibitem[Badulescu07]{Bad07} A.I. Badulescu -- Jacquet-Langlands et unitarisabilit\'e, Journal of the Institute of Mathematics of Jussieu, Vol. 6, Issue 3,  (2007)  pp. 349 - 379.

\bibitem[Badulescu18]{Bad18} A.I. Badulescu -- The trace formula and the proof of the Jacquet-Langlands correspondence. In Relative Aspects in Representation Theory, Langlands Functoriality and Automorphic Forms, Lecture Notes in Mathematics  2221, pp.197-250, 1st Edition 2018.



\bibitem[Bernstein-Zelevinski 77] {BZ77}J.Bernstein, A. Zelevinski -- Induced representations of reductive p-adic groups. I, Ann. Sci. \'{E}cole Norm. Sup. (4) 10 (1977), no. 4, 441-472.  

\bibitem[BourbakiA-8] {Bki-A8} N.~Bourbaki --  \'El\'ements de math\'ematiques.   Alg\`ebre, Chap.\ 8.  Modules et anneaux semi-simples. Springer, Berlin-Heidelberg  (2012).

\bibitem[Bushnell-Henniart02]{BH02} C.J. Bushnell, G. Henniart [[ On the derived subgroups of certain
unipotent subgroups of reductive groups over infinite fields. Transform. Groups,
7(3):211-230, 2002.


\bibitem[Broussous-Lemaire02]{Broussous-Lemaire02}  P. Broussous, B.Lemaire -- Building of $GL(m,D)$ and centralizers. Transformation groups, Vol.7, No.1, 2002, pp.15-50.



\bibitem[Bushnell-Henniart06]{BH06} C. J. Bushnell, G. Henniart -- The Local Langlands Conjecture for $GL(2)$. Grundlehren der mathematischen Wissenschaften, volume 335. Springer 2006.

\bibitem[Bushnell-Henniart-Lemaire10]{BHL10} C. J. Bushnell, G. Henniart, B. Lemaire -- Caract\`ere et degr\'e formel pour les formes int\'erieures de $GL(n)$ sur un corps local de caract\'eristique non nulle. Manuscripta mathematicas 131, 11-24 (2010).

\bibitem[Carayol84]{Ca84}  H.Carayol -- Repr\'esentaitons cuspidales du groupe lin\'eaire. Ann. sci. de l'E.N.S. tome 17, No 2 (1984) p. 191-225.


\bibitem[Casselman73]{Cas73}  W.Casselman -- The Restriction of a Representation of $GL_2(k)$ to $GL_2(o)$.   Mathematische Annalen - 206 (1973) p.311 - 318. 

\bibitem[Cassels67] {Cassels67} J.W.S. Cassels --   Global fields in Algebraic Number Theory, edited by J.W.S. Cassels, A.Frohlich, Academic Press (1967).
  


\bibitem[Dat05]{D05} J.-F.~Dat -- Nu-tempered representations of p-adic groups I : l-adic case. Duke Math. J. 126 (3) p.397-469 (2005).

\bibitem[Dat09]{D09} Jean-Francois Dat -- Finitude pour les repr\'esentations lisses de groupes $p$-adiques. J. Inst. Math. Jussieu, 8(2):261-333, 2009.

\bibitem[DKV84]{DKV84} P.Deligne, D.Kazhdan, M.-F.Vign\'eras -- Repr\'esentations des alg\`ebres centrales simples $p$-adiques, Repr\'esentations des groupes r\'eductifs sur un corps local, Hermann, Paris 1984. 

\bibitem[DS23]{DS23} B.Drevon,V.S\'echerre -- D\'ecomposition en blocs de la cat\'egorie des repr\'esentations $\ell$-modulaires lisses de longueur finie de $GL(m,D)$,  Ann. Inst. Fourier 73 (2023), 2411-2468.

\bibitem[Harish-Chandra70]{HC70} Harish-Chandra. -- Harmonic analysis on reductive $p$-adic groups. Notes by Van-Dick. Lecture notes in math, 162 Springer 1970.

\bibitem[Harish-Chandra78]{HC78} Harish-Chandra -- Admissible invariant distributions on reductive p-adic groups. Lie Theories and
their applications, Queen's Papers in Pure and Applied Mathematics, Queen's University,
Kingston, Ontario, 1978, p. 281-347.  
  University Lecture Series, vol. 16, American Mathematical Society, Providence, RI, 1999, Preface and notes by Stephen DeBacker and Paul J. Sally, Jr..
 


\bibitem[Henniart-Vign\'eras19]{HV19} G. Henniart,  M.-F. Vign\'eras -- Representations of a $p$-adic group in characteristic $p$.  Proceedings of Symposia in Pure Mathematics Volume 101, 2019 in honor of J. Bernstein, 171-210. 



\bibitem[Howe74] {Howe74}  R. Howe -- The Fourier Transform and Germs of Characters. Math. Ann. 208, 305--322 (1974).


\bibitem[Jantzen04]{Jan04} J. C. Jantzen --  Nilpotent Orbits in Representation Theory -  Lie Theory, Lie Algebras and representations  J.-P.Anker, B.Orsted, ed. Springer, Progress in Math. PM 228 (2004)  pp. 1-212. 
 

\bibitem[Lapid-Minguez-Tadic16]{LMT16} E.Lapid, A.Minguez, M.Tadic -- The admissible dual of $GL_m(D)$. Appendix A to Erez Lapid, Alberto Minguez:   
On parabolic induction on inner forms of the general linear group over a non-archimedean local field.  Select. Math.  22 No. 4 163-183 (2016), 2347-2400.


\bibitem[Lemaire04]{L04} B.Lemaire -- Int\'egrabilit\'e locale des caract\`eres tordus de $GL_n(D)$. J.Reine Angew.Math. 566 (2004), 1-39.


\bibitem[Lusztig-Spaltenstein79]{LS79} G.Lusztig,  N. Spaltenstein --  Induced unipotent classes. Journal of the London Mathematical Society, Volume s2-19, Issue 1,  1979, Pages 41-52

\bibitem[Meyer-Solleveld12]{MS12} R. Meyer, M. Solleveld -- 
 Characters and growth of admissible representations of reductive p-adic groups.
 J. Inst. Math.   Jussieu   11  No2 (2012) pp. 289 - 331.

 \bibitem[Minguez-S\'echerre14]{MS14} A. Minguez, V. S\'echerre -- 
 Repr\'esentations lisses modulo $\ell$  de $ GL(m,D)$. Duke Math. J. 163 (2014), 795-887
 
 
  \bibitem[Minguez-S\'echerre17]{MS17} A. Minguez, V. S\'echerre -- 
 Correspondance de Jacquet-Langlands locale et congruences modulo $\ell$. Invent. Math. 208 (2017), no2, 553-631.

\bibitem[Moeglin-Waldspurger87]{MW87} C. Moeglin, J.-L. Waldspurger  --  Mod\`eles de Whittaker d\'eg\'en\'er\'es
pour des groupes p-adiques. Math. Z. 196, 427-452 (1987)

\bibitem[Morra13]{Morra13} S. Morra --  Invariant elements for $p$-modular representations of $GL_2(Q_p)$, Trans. Am. Math. Soc.
365 (2013), no. 12, 6625-6667.

\bibitem[Moy-Prasad96]{MP96} A.~Moy, G.~Prasad -- Jacquet functors and unrefined minimal $K$-types,  Comment.Math.Helv.   {71}, p.~98--121  (1996).


\bibitem[Murnaghan91]{M91} F. Murnaghan --  Asymptotic behaviour of supercuspidal characters of $p$-adic $GL_3$ and $GL_4$, the generic unramified case. Pacific J.of Math   148, 107-130 (1991)

\bibitem[Nadimpalli-Sheth23]{NS23} S. Nadimpalli, M. Sheth --  Twisted Jacquet modules of Steinberg and Speh representations, ArXiv:2310.00735 (2023).

\bibitem[Prasad00]{P00} D. Prasad -- Comparison of germ expansion on inner forms of GL(n), Manuscripta math. 102, 263-268 (2000).

\bibitem[Prasad-Raghuram00]{PR00} D. Prasad, A. Raghuram -- Kirillov theory for $GL_2(D)$  where $D$ is a division algebra over a non-archimedean local field, Duke   Mah. J. 104, 19-44 (2000).
 


\bibitem[Reiner75]{R75} M. Reiner  -- Maximal orders. Academic Press 1975.

  \bibitem[Rodier74]{Rodier} F. Rodier -- Mod\`ele de Whittaker et caract\`eres de repr\'esentations, in non commutative harmonic analysis, J. Carmona et M. Vergne, Lect. Notes Math. 466, pp. 151-71. Berlin-Heidelberg-New-York: Springer 1974.

\bibitem[Suzuki22]{Su22} K. Suzuki -- Gelfand-Kirillov dimension of representations
of $GL_n$ over a non-archimedean local field, arXiv:2208.05139v3,  2022.
\bibitem[Tadic90] {T90} M. Tadic -- Induced representations of $GL (n, A)$ for $p$-adic division algebras $A$. Journal f\"ur die reine und angewandte Mathematik, de Gruyter  (1990).

\bibitem[Vigneras96] {V96} M.-F. Vign\'eras  -- Repr\'esentations $\ell$-modulaires d'un groupe r\'eductif fini $p$-adique avec $\ell \neq p$. Birkhauser Progress in math. 137 (1996).

\bibitem[Weil67]{W67} A. Weil -- Basic Number Theory, Grundlehren 114, 1967.

\bibitem[Zelevinski80]{Ze80}  A.Zelevinski -- Induced representations of reductive p-adic groups II, Ann. Scient. Ec. Norm. Sup. 13 (1980), 165-210.
  \end{thebibliography}
  \end{document}